\documentclass[reqno]{amsart}
\usepackage[parfill]{parskip}
\usepackage{algorithm2e}

\usepackage{amssymb,amsfonts,amsmath}
\usepackage{comment} 
\usepackage[all,arc]{xy}
\usepackage{enumerate}
\usepackage{mathrsfs,mathtools}
\usepackage{booktabs}
\usepackage{stmaryrd}
\usepackage{marvosym}
\usepackage{graphicx}
\usepackage{pdflscape}
\usepackage{scalerel}
\usepackage{float}
\usepackage{microtype}
\usepackage{enumitem}

\usepackage[table,xcdraw]{xcolor}
\usepackage{tikz}
\usepackage{tikz-cd}
\usetikzlibrary{trees}
\usetikzlibrary[shapes]
\usetikzlibrary[arrows]
\usetikzlibrary{patterns}
\usetikzlibrary{patterns.meta}
\usetikzlibrary{fadings}
\usetikzlibrary{backgrounds}
\usetikzlibrary{decorations.pathreplacing}
\usetikzlibrary{decorations.pathmorphing}
\usetikzlibrary{positioning}
\usepackage{quiver}
\usepackage{diagbox}

\usepackage{graphicx}
\SelectTips{cm}{10}
\definecolor{dark-green}{rgb}{0.15,0.6,0.15}

\usepackage{hyperref}
\hypersetup{
    colorlinks,
    linkcolor=rainbow-indigo,
    citecolor = rainbow-blue
}

\usepackage{caption}
\usepackage{subcaption}

\usepackage[nameinlink,capitalise,noabbrev]{cleveref}

\newtheorem{thm}{Theorem}[section]
\newtheorem{theorem}[thm]{Theorem}
\newtheorem{corollary}[thm]{Corollary}
\newtheorem{proposition}[thm]{Proposition}
\newtheorem{lemma}[thm]{Lemma}
\newtheorem{conjecture}[thm]{Conjecture}

\newtheorem*{theorem*}{Theorem}
\newtheorem*{proposition*}{Proposition}
\newtheorem*{remark*}{Remark}
\newtheorem*{conjecture*}{Conjecture}

\newtheorem{thmx}{Theorem}

\theoremstyle{definition}
\newtheorem{definition}[thm]{Definition}
\newtheorem{example}[thm]{Example}

\theoremstyle{remark}
\newtheorem{remark}[thm]{Remark}
 
\newtheorem{construction}[thm]{Construction}

\DeclareMathOperator{\Sub}{Sub}
\DeclareMathOperator{\Div}{Div}

\DeclareMathOperator{\SR}{SR}
\DeclareMathOperator{\DR}{DR}

\newcommand{\NN}{\mathbb{N}}

\newcommand{\w}{\rotatebox[origin=c]{180}{$\mathfrak{m}$}}

\newcommand{\grid}[2]{
    \foreach \i in {0,...,#1}
    {
        \foreach \j in {0,...,#2}
        {
            \node [fill=dark-green!50,circle,draw,inner sep = 0pt, outer sep = 0pt, minimum size=1.5mm] at (\i,\j) {};
        }    
    }
}

\newcommand{\gridbox}[4]{
    \draw[draw=black] (#1 - 0.3, #2 - 0.3) rectangle (#3 + 0.3, #4 + 0.3);
}

\newcommand{\ul}{\underline}
\DeclarePairedDelimiter\ceil{\lceil}{\rceil}
\DeclarePairedDelimiter\floor{\lfloor}{\rfloor}

\definecolor{rainbow-red}{HTML}{e81416}
\definecolor{rainbow-orange}{HTML}{ffa500}
\definecolor{rainbow-yellow}{HTML}{faeb36}
\definecolor{rainbow-green}{HTML}{79c314}
\definecolor{rainbow-blue}{HTML}{487de7}
\definecolor{rainbow-indigo}{HTML}{4b369d}
\definecolor{rainbow-violet}{HTML}{70369d}

\usepackage{tabularray}
\usepackage{fullpage}
\usepackage{caption}
\usepackage{subcaption}

\title{On Minimal Bases in Homotopical Combinatorics}

 \author[Adamyk]{Katharine Adamyk}
 \address{Mathematics Department, Hamline University, US}
  \email{kadamyk01@hamline.edu}
  
 \author[Balchin]{Scott Balchin}
 \address{Mathematical Sciences Research Centre, Queen's University Belfast, UK}
  \email{s.balchin@qub.ac.uk }

 \author[Barrero]{Miguel Barrero}
 \address{Department of Mathematics, University of Aberdeen, UK}
  \email{miguel.barrero@abdn.ac.uk}
  
 \author[Scheirer]{Steven Scheirer}
 \address{Department of Mathematics and Computer Science, Susquehanna University, US}
  \email{scheirer@susqu.edu}

 \author[Wisdom]{Noah Wisdom}
 \address{Department of Mathematics, Northwestern University, US}
  \email{noahankney2026@u.northwestern.edu}

   \author[Zapata Castro]{Valentina Zapata Castro}
 \address{Department of Mathematics, University of Virginia, US}
  \email{vz6an@virginia.edu}
  
\date{\today}

\begin{document}

\begin{abstract}
    We present a development in the computational suite for the study of $N_\infty$ operads for a finite group $G$. This progress is achieved using the simple yet powerful observation that Rubin's generation algorithm can be interpreted as a closure operator. Leveraging this perspective, we establish the existence of minimal bases for $N_\infty$ operads.  By investigating these bases for certain families of groups we are led to introduce and analyze several novel combinatorial invariants for finite groups.
\end{abstract}

\maketitle\markboth{\shorttitle}{\shorttitle}

\tableofcontents

\newpage

\section{Introduction}

\subsection*{Context}

Homotopical combinatorics is an area of mathematics which inhabits the rich intersection of equivariant homotopy theory and combinatorics. One of the central problems in the field is, for a fixed finite group $G$, to understand the poset $N_\infty(G)$ of (homotopy classes of) \emph{$N_\infty$ operads} for $G$. These are gadgets which control the behavior of multiplicative commutativity in the presence of a $G$ action \cite{blumberghill}. Once the underlying set of $N_\infty(G)$ is understood, one then investigates further structure which is inherent to it.

While the definition of $N_\infty$ operads is somewhat cumbersome, work of multiple authors has proved that they are realized by a simple, discrete, structure called a \emph{transfer system} \cite{bbr, BP_operads, GW_operads, Rubin_comb} which is a collection of \emph{norm maps} satisfying certain relations. That is, the problem of investigating $N_\infty(G)$ can be equivalently recast as investigating the set of transfer systems, which we denote $\mathsf{Tr}(G)$.

In practice, $\mathsf{Tr}(G)$ is only completely understood as a poset for the family of groups of the form $C_{p^n}$ for $p$ prime, where Balchin--Barnes--Roitzheim proved that this poset is exactly the ubiquitous associahedron \cite{bbr}. We moreover have enumeration results for abelian $p$-groups of rank 2, and groups of the form $C_{p^nq}$ and $D_{p^n}$ \cite{bao2023transfersystemsrankelementary, BMO_enumeration}.

These results, as well as many of the other partial results existing in the literature, have usually been informed by computational evidence. However, the computational evidence that we have access to is somewhat sparse, as naive algorithms have infeasible scalability. As such, any development which would expand the database of computational evidence that we have is invaluable.

\subsection*{Main results}

The first major development of this paper is a discussion of an efficient algorithm for the computation of $\mathsf{Tr}(G)$. This algorithm is particularly powerful in the non-abelian setting, something which up-until-now has not seen much investigation outside of a specific class of non-abelian groups \cite{BMO_lift}.

The key input is Rubin's algorithm which generates the smallest transfer system from a given collection of norm maps \cite{rubin}. Writing $\mathbb{I}(\Sub(G))$ for the collection of non-trivial intervals of subgroups $1 \leqslant H < K \leqslant G$, we can repackage Rubin's algorithm as a \textit{closure operator} $\langle - \rangle \colon \mathcal{P}(\mathbb{I}(\Sub(G))) \to \mathcal{P}(\mathbb{I}(\Sub(G)))$ of which the closed sets are  the transfer systems. There exist many algorithms, amenable to parallelization, for computing the closed sets of a closure operator in the literature. The implementation of such an algorithm using data regarding the group from a computer algebra package such as \texttt{Sage} provides us with a powerful new tool, and has been released as a header-only \texttt{C++} library by the second author in \cite{balchin2025ninftysoftwarepackagehomotopical}.

Using this algorithm optimized for HPC facilities we are able to generate far more data than we had access to previously. For example, prior to this paper, we only knew $|\mathsf{Tr}(S_i)|$ for $i=1,2,3$, where $S_i$ is the symmetric group of degree $i$, while with this algorithm we are able to compute:
    \begin{align*}
        |\mathsf{Tr}(S_4)| & = 8691, \\
        |\mathsf{Tr}(S_5)| & = 183598202.
    \end{align*}

While these new results are interesting in their own right, this paper is concerned with investigating the combinatorics that arise from this algorithm. By running Rubin's algorithm in reverse, one is able to find a set of generators for a transfer system. Given that Rubin's algorithm is a closure operator, we can think of such a set as a \emph{basis} for the transfer system. This basis is minimal in the sense that there are no redundant elements. Our first result tells us that all minimal bases for a transfer system have the same size.

\begin{thmx}[\cref{cor:well-defined-map}]
    Let $G$ be a finite group and $\mathsf{T} \in \mathsf{Tr}(G)$. Then any minimal basis for $\mathsf{T}$ has the same cardinality. In particular, there is a well-defined assignment 
    \begin{align*}
    \mathfrak{m} \colon \mathsf{Tr}(G) \to \mathbb{N}
    \end{align*}
    which takes a transfer system $\mathsf{T}$ to the size of a minimal basis of $\mathsf{T}$.
\end{thmx}

\newpage

It is the behavior of the map $\mathfrak{m}$ for various families of groups which occupies \cref{part:combi} of this paper. We can immediately extract two numerical invariants from $\mathfrak{m}$ which are of significant interest:
\begin{itemize}[leftmargin=*]
    \item the \emph{width} $\w$ of $G$: the cardinality of a minimal basis for the complete transfer system (i.e., the transfer system with all possible norm maps);
    \item the \emph{complexity} $\mathfrak{c}$ of $G$: the maximum value of $\mathfrak{m}(\mathsf{T})$ as $\mathsf{T}$ ranges over transfer systems for $G$. 
\end{itemize}

The width of a group is somewhat familiar invariant of $G$. Recall that a subgroup of $G$ is said to be \emph{meet-irreducible} if it is not the intersection of two proper subgroups.

\begin{thmx}[\cref{prop:width}]
    Let $G$ be a finite group. Then $ \w (G)$ is equal the number of conjugacy classes of meet-irreducible subgroups of $G$.
\end{thmx}

The complexity of $G$, however, is a more exotic invariant. In terms of the algorithm, it counts the number of iterations that the algorithm needs to complete. It is difficult to compute, but it does provide us with a non-trivial lower bound on the number of transfer systems, in particular, $2^{\mathfrak{c}(G)} \leqslant |\mathsf{Tr}(G)|$ \cite{luo2024latticeweakfactorizationsystems}.

We begin by tackling the issue of complexity for the family of square-free cyclic groups. The following provides a lower bound for the complexity, and we further conjecture--with good reason--that this is indeed the complexity (see \cref{rem:symm}).

\begin{thmx}[\cref{thm:maincube}]\label{thmx:cube}
    Let $G = C_{p_1 \cdots p_n}$ for $p_i$ distinct primes. Then
    \[
    \mathfrak{c}(G) \geqslant \begin{cases}
        \sum\limits_{i=0}^{\floor{\frac{n-1}{2}}} \binom{n}{n-i} \times \binom{n-i}{i+1} & n = 2, 4, 6 \\
        &\\
        \sum\limits_{i=0}^{\floor{\frac{n-1}{2}}} \binom{n}{n-i} \times \binom{n-i}{i} & n \neq 2, 4, 6. \\
        \end{cases}
    \]
\end{thmx}

The change at behavior at $n=8$ is the main technical substance of this result, and reflects specificity regarding the growth of binomial coefficients.

We then turn our attention to groups of the form $G=C_{p^nq^m}$ where $p$ and $q$ are distinct primes. While we only obtain partial results for $m > 1$, for $m=1$ we obtain the following theorem, which is an exact calculation of the complexity:

\begin{thmx}[\cref{cor:Cpnq-complexity}]
    Let $G=C_{p^nq}$ for $p,q$ distinct primes. Then
    \[
    \mathfrak{c}(G) =
    \begin{cases}
        3k+1 & n=2k, \\
        &\\
        3k+2 & n=2k+1.
    \end{cases}
    \]
\end{thmx}

That is, the sequence $\mathfrak{c}(C_{p^{n}q})$ as $n$ increases is the sequence of \emph{non-multiples} of 3. In particular, we can conclude that the complexity in this case grows linearly in the number of terms in a prime factorization of the group order as opposed to the square-free cyclic case.

While we have focused on the complexity and width of groups with respect to the generation algorithm, this is only the first direction of investigation. In \cref{sec:future} we present multiple further lines of inquiry which we expect will lead to further fruitful connections being found between equivariant homotopy theory and combinatorics.

\begin{remark}\label{rem:luo}
    Luo--Rognerud have independently introduced the notion of the complexity of a group in their own work, albeit under a different guise. In \cite{luo2024latticeweakfactorizationsystems} the authors discuss various properties of the lattice of weak factorization systems on a finite lattice. Given the bijection between weak factorization systems on the lattice $\Sub(G)$ and transfer systems for $G$, when $G$ is abelian \cite{fooqw}, this is relevant to the current paper. In the process of their results, they discuss the concept of lower covers and elevating sets for weak factorization systems. What they denote by $\mathrm{mcov}_{\downarrow}(\mathsf{Tr}(G))$ is what we would denote $\mathfrak{c}(G)$. In terms of the lattice theory of $\mathsf{Tr}(G)$, the complexity exactly computes the \emph{breadth} of $\mathsf{Tr}(G)$. Here we opt to keep the dependence on the group $G$ itself, hence the usage of the term \emph{complexity}.
    
    While Luo--Rognerud do not spend much time in \cite{luo2024latticeweakfactorizationsystems} investigating $\mathrm{mcov}_{\downarrow}(\mathsf{Tr}(G))$, they do provide a lower bound for square-free cyclic groups as we do in \cref{thmx:cube} (\cite[Question 10.10]{luo2024latticeweakfactorizationsystems}). Their lower bound coincides with ours until $n=8$, where they did not expect the growth change for even $n$ as is presented in our theorem. As already discussed, it is this growth change which is the main technical component of our result.
\end{remark}

\subsection*{Notation}

Fix a finite group $G$.

\begin{itemize}
    \item $\Sub(G)$ is the collection of subgroups of $G$ ordered by inclusion.
    \item $\mathsf{Tr}(G)$ is the collection of transfer systems of $G$ ordered by inclusion.
    \item $\mathbb{I}(\Sub(G))$ is the collection of intervals in the lattice $\Sub(G)$. An element of $\mathbb{I}(\Sub(G))$ will be denoted via a pair $(H,K)$.
    \item For a subset $S \subseteq \mathbb{I}(\Sub(G))$ we write $\langle S \rangle$ for the transfer system generated by $S$ (see \cref{cons:rubin}).
    \item $\Div(\NN)$ denotes the divisibility lattice of the natural numbers.
    \item We denote by $\binom{n}{x,y-x,n-y}$ the trinomial coefficient which is equivalent to the product of binomial terms $\binom{n}{x} \times \binom{n-x}{y-x}$.
\end{itemize}

\subsection*{Acknowledgments}

This research was initiated at the AMS funded MRC on \textit{Homotopical Combinatorics}, which was partially supported by NSF grant 1916439. The authors thank the AMS and the organizers of this event for providing them with this opportunity, as well as our fellow participants for many helpful informal discussions. We are moreover thankful to Dave Benson for various numerical computations and Baptiste Rognerud for bringing \cite{luo2024latticeweakfactorizationsystems} to our attention. Finally, we thank Yuri Sulyma for conversations at an early stage of this project.
 
We are grateful for use of the computing resources from the Northern Ireland High Performance Computing (NI-HPC) service funded by EPSRC (EP/T022175). MB is supported by the EPSRC grant EP/X038424/1 “Classifying spaces, proper actions and stable homotopy theory”.

\newpage

\part{The theory of minimal generation}\label{part:theory}

In this part we present the general theory of generation of transfer systems and the concepts of minimal bases. \cref{part:combi} will then be dedicated to the combinatorics associated to this theory for a wealth of examples.

\section{Generation of transfer systems}

\subsection{Rubin's algorithm as a closure operator}

We begin by recalling the notion of a $G$-transfer system. Recall that $\mathbb{I}(\Sub(G))$ is the collection of intervals in the lattice $\Sub(G)$ with elements pairs $(H,K)$ where $1 \leqslant H \leqslant K \leqslant G$. We will sometimes refer to the pairs $(H,K)$ as \emph{arrows}.

\begin{definition}
    Let $G$ be a finite group. A \emph{($G$-)transfer system} is a subset $\mathsf{T} \subseteq \mathbb{I}(\Sub(G))$ such that:
    \begin{itemize}
        \item (Identity) $(H,H) \in \mathsf{T}$ for all $H \leqslant G$;
        \item (Composition) If $(L,K) \in \mathsf{T}$ and $(K,H) \in T$ then $(L,H) \in \mathsf{T}$;
        \item (Conjugation) If $(K,H) \in \mathsf{T}$ then $(K^g , H^g) \in \mathsf{T}$ for all $g \in G$;
        \item (Restriction) If $(K,H) \in \mathsf{T}$ and $L \leqslant H$ then $(K \cap L, L) \in \mathsf{T}$.
    \end{itemize}
    We denote the collection of all transfer systems of $G$ by $\mathsf{Tr}(G)$.
\end{definition}

\begin{remark}
    The axioms of a transfer system moreover imply that they are pullback closed. That is, if we we have $(H_1 , K)$ and $(H_2, K)$ in $\mathsf{T}$, then we also have $(H_1 \cap H_2, K) \in \mathsf{T}$ \cite[Proposition 4.2]{fooqw}.
\end{remark}

The following algorithm of Rubin is the key machinery of this paper. It provides a way of generating the smallest transfer system from a given collection of arrows.

\begin{construction}[Rubin's algorithm {\cite[Construction A.1]{rubin}}]\label{cons:rubin}
    Let $G$ be a finite group, and $S \subseteq \mathbb{I}(\Sub(G))$. Define
    \begin{align*}
        S_0 \coloneqq & S, \\[5pt]
        S_1 \coloneqq & \bigcup_{(K,H) \in S_0} \{ (K^g, H^g) \mid g \in G \},\\[5pt]
        S_2 \coloneqq & \bigcup_{(K,H) \in S_1} \{ (L \cap K,L) \mid L \subseteq H \},\\[5pt]
        S_3 \coloneqq & \{(K,H) \mid \exists n \geqslant 0 \text{ and subgroups } H_0, \dots H_n \text{ with } (K,H_0), (H_0, H_1), \dots , (H_n,H) \in S_2\}.
    \end{align*}

    Then $\langle S \rangle \coloneqq S_3$ is the minimal $G$-transfer system which contains $S$. In particular we say that $S$ is a \emph{generating set} for $\langle S \rangle$.
\end{construction}

The reason that Rubin's algorithm is so useful is that it is an example of a \emph{closure} operator:

\begin{definition}
    Let $X$ be a set. A \emph{closure operator} on $X$ is a function $\operatorname{cl} \colon \mathcal{P}(X) \to \mathcal{P}(X)$ such that for all $A,B \subseteq X$:
    \begin{enumerate}
        \item $A \subseteq \operatorname{cl}(A)$,
        \item $A \subseteq B \Rightarrow \operatorname{cl}(A) \subseteq \operatorname{cl}(B)$,
        \item $\operatorname{cl}(\operatorname{cl}(A)) = \operatorname{cl}(A)$.
    \end{enumerate}
    A subset $A \subseteq X$ is \emph{closed} if $A= \operatorname{cl}(A)$.
\end{definition}

\begin{lemma}
    Let $G$ be a finite group. Then Rubin's Algorithm (\cref{cons:rubin}) is a closure operator $\langle - \rangle \colon \mathcal{P}(\mathbb{I}(\Sub(G))) \to \mathcal{P}(\mathbb{I}(\Sub(G)))$. The closed sets of this closure operator are exactly the transfer systems for $G$.
\end{lemma}

\begin{proof}
    It is clear by construction that for $S \subseteq \mathbb{I}(\Sub(G))$ we have $S \subseteq \langle S \rangle$. If $S \subseteq S'$ then $S_i \subseteq S_i'$ for each $i \in \{0,1,2,3\}$ in \cref{cons:rubin}, validating the second axiom. The third axiom, as well as the claim regarding the closed sets follows from \cite[Theorem A.2]{rubin}.
\end{proof}

The power of casting Rubin's algorithm as a closure operator is that there are highly efficient algorithms for computing the closed sets of closure operators for finite sets. In \cref{appendix:a} we present a description of a naive algorithm, some pseudocode, and a worked example for the case of $C_{p^2q}$. A full implementation of this algorithm in \texttt{C++} along with other operations for transfer systems can be found in \cite{balchin2025ninftysoftwarepackagehomotopical}.

The algorithm filters the set $\mathsf{Tr}(G)$ into subsets based on which stage the transfer system is found. This stratification in fact is by the size of the minimal generating set of the transfer systems as we will discuss in the next section. It is hoped that considering this stratification may provide an insight on how to prove results regarding $|\mathsf{Tr}(G)|$. We discuss this idea in more detail in \cref{sec:filtration}.

\subsection{Minimal Generating Sets}

Now that we have recalled how to generate the smallest transfer system containing a set of arrows $S$, we can begin to study the generating sets for a fixed transfer system $\mathsf{T}$. As usual, we are interested in the smallest such set(s):

\begin{definition}
    Let $\mathsf{T} \in \mathsf{Tr}(G)$. Then a subset $S \subseteq \mathsf{T}$ is a \emph{minimal generating set for $\mathsf{T}$} if:
    \begin{enumerate}
        \item $\langle S \rangle = \mathsf{T}$ and 
        \item $\langle S \setminus \{s\} \rangle \subsetneq \mathsf{T}$ for all $s \in S$.
    \end{enumerate}
\end{definition}

\begin{remark}
    From now on we will disregard the pairs $(H,H)$ whenever we discuss transfer systems or their generation as these elements do not need to be explicitly considered when computing transfer systems. We will say that an element $(H,H) \in \mathbb{I}(\Sub(G))$ is an \emph{identity element}, and as such will ask that $S$ consists of non-identity elements of $\mathbb{I}(\Sub(G))$.
\end{remark}

Minimal generating sets are the primary focus of this paper. It is clear that there is a minimal generating set for every transfer system, however the following simple example demonstrates that minimal generating sets need not be unique.

\begin{example}\label{ex:diff-min}
    Let $G = C_{p^2}$ and consider the complete transfer system $\mathsf{T}_C \in \mathsf{Tr}(G)$. That is, the transfer system given by $\mathsf{T}_C = \{ (1,C_p), (1,C_{p^2}), (C_p,C_{p^2}) \}$. Then there are two minimal generating sets for $\mathsf{T}_C$, namely:
    \begin{enumerate}
        \item $S_a =  \{ (1,C_p), (C_p,C_{p^2}) \}$, and
        \item $S_b =  \{ (1,C_{p^2}), (C_p,C_{p^2}) \}$.
    \end{enumerate}
    These generating sets have different properties. The set $S_a$ has the shortest possible arrows, and this is our preferred representation. The generating set $S_b$ instead favors the maximal arrows, and are the generating sets that have been previously considered for this specific case  in \cite[Section 5]{fooqw}.
\end{example}

While we have now seen that even in the simplest case there may be more than one minimal generating set for a given transfer system, we will now prove that all minimal generating sets have the same cardinality, and as such, the function $\mathfrak{m} \colon \mathrm{Tr}(G) \to \mathbb{N}$ assigning to each transfer system the cardinality of a minimal generating set is well defined. We first introduce some required intermediate constructions.

\begin{lemma}
\label{lemma-generatingedge}
    Let $G$ be a finite group and let $S$ be a set of non-identity elements of $\mathbb{I}(\Sub(G))$. Let $(K , H) \in \mathbb{I}(\Sub(G))$. If $(K, H) \in \langle S \rangle$, then there exist $K' \leqslant H' \leqslant G$ and $g \in G$ with $(K', H') \in S$ and $K \leqslant (K')^g$, $H \leqslant (H')^g$.
\end{lemma}
\begin{proof}
    Since $(K, H) \in \langle S \rangle$, we know by \cref{cons:rubin} that $(K , H)$ can be written as a composition of restrictions of conjugations of elements of $S$, in other words, there exists an $n \geqslant 1$, groups $K_i, H_i, L_i$, and elements $g_i \in G$ for each $0 < i \leqslant n$, such that for each $i$ we have:
    \begin{enumerate}
        \item $(K_i, H_i) \in S$,
        \item $L_i \leqslant H_i^{g_i} $,
        \item $L_1 \cap K_1^{g_1} = K$,
        \item $L_n = H$,
        \item $L_i \cap K_i^{g_i} = L_{i-1}$ if $i >1$.
    \end{enumerate}

    By (2) and (4) we know that $H \leqslant H_n^{g_n} $. By (3) and (5) we know that $K \leqslant K_n^{g_n}$.
    Therefore by (1) the pair $(K_n , H_n)$ and the element $g_n$ are the required ones.
\end{proof}

We will use \cref{lemma-generatingedge} in combination with a filtration on $\Sub(G)$ to prove that all minimal generating sets are of the same cardinality:

\begin{definition}
\label{definition-rankmap}
Let $G$ be a finite group. Consider the order-preserving map $|-| \colon \Sub(G) \to \Div(\NN)$ that sends a subgroup $H \leqslant G$ to $\lvert H \rvert$. As conjugate subgroups have the same order, this map is $G$-equivariant for the conjugation action on $\Sub(G)$ and the trivial action on $\Div(\NN)$. 

Consider also the rank map $r \colon \Div(\NN) \to \NN$, which sends $N = p_1^{n_1} p_2^{n_2} \dots p_k^{n_k}$ to the sum of the exponents $n_i$. This map is order-preserving, and we denote the composition of these two maps by $P \colon \Sub(G) \to \NN$. 

By abuse of notation we will also write $P \colon \mathbb{I}(\Sub(G)) \to \mathbb{I}(\NN)$ for the extension of this mapping to intervals.
\end{definition}

\begin{construction}
    Let $G$ be a finite group and let $S$ be a set of non-identity elements of $\mathbb{I}(\Sub(G))$. For $N \in \NN$, we define $S_N\coloneqq\{ (K, H) \in S \mid P(K) \geqslant N\}$. That is, $S_N$ is the set of elements of $S$ whose smaller group has total exponent at least $N$. This provides a decreasing filtration $\cdots \subseteq S_2 \subseteq S_1 \subseteq S_0 = S$.
\end{construction}

We will now use Rubin's algorithm (\cref{cons:rubin}) to characterize the images under the map $P \colon \mathbb{I}(\Sub(G)) \to \mathbb{I}(\NN)$ of elements that generate a given transfer system.

\begin{lemma}
\label{lemma-generatefiltration}
    Let $G$ be a finite group and let $S$ be a set of non-identity elements of $\mathbb{I}(\Sub(G))$. Let $(K, H)$ be an arbitrary element of $\mathbb{I}(\Sub(G))$, and let $N=P(K)$. If $(K, H) \in \langle S \rangle$ then there exists a subset $S' \subseteq S_N$ such that $(K, H) \in \langle S' \rangle$ and at most one morphism of $S'$ is not contained in $S_{N+1}$.
\end{lemma}

\begin{proof}
    We begin by following the same proof as in \cref{lemma-generatingedge}. In particular, there again exists an $n \geqslant 1$, groups $K_i, H_i, L_i$, and elements $g_i \in G$ for each $0 < i \leqslant n$, such that (1)-(5) of that proof hold.

    By (1) and (3) the interval $(K_1, H_1)$ belongs to $S_N$. We can assume that the subgroup $L_1 \cap g_1 K_1 g_1^{-1} = K$ is a proper subgroup of $L_1$, else we could just remove that arrow from the composition. This, together with (1) and (5), means that for each $i > 1$ the element $(K_i, H_i)$ belongs to $S_{N+1} \subseteq S_N$. The subset $S'$ is formed by all the intervals $(K_i, H_i)$, and hence indeed $(K, H) \in \langle S' \rangle$ as required.
\end{proof}

\begin{proposition}\label{prop:all-same-size}
    Let $G$ be a finite group and $\mathsf{T} \in \mathsf{Tr}(G)$. Let $S$ and $R$ be two minimal generating sets for $\mathsf{T}$. Then there exists a bijection $f \colon S \to R$ which for each $n \in \NN$ sends $a \in S_n \setminus S_{n+1}$ to $f(a) \in R_n \setminus R_{n+1}$.
\end{proposition}
\begin{proof}
    First, let $N=P(G)$. For any $n \leqslant N$ and $a \in S_n$, since $a \in \mathsf{T}=\langle R \rangle$, by \cref{lemma-generatefiltration} we obtain that $a \in \langle R_n \rangle$. Applying this also to elements of $R$, we obtain that $\langle S_n \rangle = \langle R_n \rangle$.
    
    For a given $n \leqslant N$, we will construct a function $f_n \colon S_n \setminus S_{n+1} \to R_n \setminus R_{n+1}$ which is a bijection; the $f_n$ will assemble to give the desired $f$.

    Consider the following relation $f_n$ on the set $X = (S_n \setminus S_{n+1}) \cup (R_n \setminus R_{n+1})$. For $a, b \in X$ we say that $a f_n b$ if $a\in \langle b, \langle S_{n+1} \rangle \rangle = \langle b, \langle R_{n+1} \rangle \rangle$. We prove some properties of $f_n$.
    
    If $a f_n b$, the arrows $a$ and $b$ cannot both belong to the same of $S_n \setminus S_{n+1}$ or $R_n \setminus R_{n+1}$ unless $a=b$, otherwise it would contradict the fact that $S$ and $R$ are minimal generating sets.

    Next, we prove that $f_n$ is transitive. If $a f_n b$ and $b f_n c$, then $a\in \langle b, \langle S_{n+1} \rangle \rangle$ and $b\in \langle c, \langle S_{n+1} \rangle \rangle$, so $a\in \langle \langle c, \langle S_{n+1} \rangle \rangle, \langle S_{n+1} \rangle \rangle = \langle c, \langle S_{n+1} \rangle \rangle$.

    For each $a \in S_n \setminus S_{n+1}$, we have that $a \in \mathsf{T} = \langle R \rangle$, so we apply \cref{lemma-generatefiltration}, obtaining a subset $R' \subset R_n$ that satisfies that $a \in \langle R' \rangle$, and we use $b$ to denote the unique element of $R'$ that belongs to $R_n \setminus R_{n+1}$. This means that for each $a$ there exists a $b \in R_n \setminus R_{n+1}$ such that $a f_n b$. By the same argument we can find such an $a$ for each $b \in R_n \setminus R_{n+1}$, with $b f_n a$.

    Given $a f_n b$ with $a \in S_n \setminus S_{n+1}$, from the previous paragraph we obtain a $c \in S_n \setminus S_{n+1}$ with $b f_n c$. From the transitiveness of $f_n$ we obtain that $a f_n c$, which means that $a = c$, and therefore $f_n$ is symmetric and an equivalence relation.

    Finally, assume that there are $a$ in $S_n \setminus S_{n+1}$ and $b, b'$ in $R_n \setminus R_{n+1}$ with $a f_n b$ and $a f_n b'$. Then by symmetry $b f_n a$, and by transitivity $b f_n b'$, so $b = b'$. This means that each equivalence class of $f_n$ contains exactly two elements, one in $S_n \setminus S_{n+1}$ and the other in $R_n \setminus R_{n+1}$. Therefore it determines a bijection between these two sets, which we also denote by $f_n$.
\end{proof}

\begin{corollary}\label{cor:well-defined-map}
    Let $G$ be a finite group. Then there is a well-defined mapping $\mathfrak{m} \colon \mathsf{Tr}(G) \to \NN$ defined by sending a transfer system $\mathsf{T}$ to the cardinality of a minimal generating set of $\mathsf{T}$.
\end{corollary}

\section{Rainbows}\label{sec:rainbows}

We now introduce a particular class of minimal generating set which will occur throughout. 

\begin{definition}
    A \emph{rainbow} on $\NN$ is a finite set of ordered pairs $\{(a_i,b_i)\}_{i \in [0,k]}$ of natural numbers such that if $i < j$ then $a_i < a_j < b_j < b_i$. 
\end{definition}

In other words, a rainbow is a finite set of nested arcs on $\NN$:

 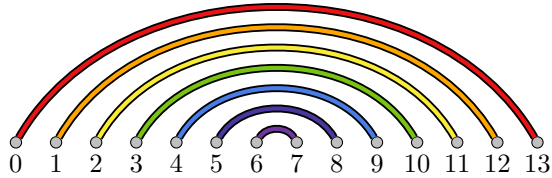
\begin{figure}[h!]
\centering
\begin{tikzpicture}[scale=0.8]
\foreach \i in {0,...,13}
{
    \node[fill=black!25,circle,draw,inner sep = 0pt, outer sep = 0pt, minimum size=1.5mm] (\i) at (\i/1.5,0) {};
    \node at (\i/1.5,-0.35) {\i};
}
\draw[-, line width=3pt, black] (0) edge[bend left=60] (13);
\draw[-, line width=3pt, black] (1) edge[bend left=60] (12);
\draw[-, line width=3pt, black] (2) edge[bend left=60] (11);
\draw[-, line width=3pt, black] (3) edge[bend left=60] (10);
\draw[-, line width=3pt, black] (4) edge[bend left=60] (9);
\draw[-, line width=3pt, black] (5) edge[bend left=60] (8);
\draw[-, line width=3pt, black] (6) edge[bend left=60] (7);
\draw[-, ultra thick, rainbow-red] (0) edge[bend left=60] (13);
\draw[-, ultra thick, rainbow-orange] (1) edge[bend left=60] (12);
\draw[-, ultra thick, rainbow-yellow] (2) edge[bend left=60] (11);
\draw[-, ultra thick, rainbow-green] (3) edge[bend left=60] (10);
\draw[-, ultra thick, rainbow-blue] (4) edge[bend left=60] (9);
\draw[-, ultra thick, rainbow-indigo] (5) edge[bend left=60] (8);
\draw[-, ultra thick, rainbow-violet] (6) edge[bend left=60] (7);
\end{tikzpicture}
\caption{An example of a particularly nice rainbow.}
\end{figure}

\begin{definition}
\label{definition-partialrainbow}
    Let $G$ be a finite group. Let $S$ be a set of non-identity elements of $\mathbb{I}(\Sub(G))$.  We say that $S$ is a \emph{partial rainbow} on $G$ if:
    \begin{enumerate}
        \item the image of $S$ under the map $P \colon \mathbb{I}(\Sub(G)) \to \mathbb{I}(\NN)$ is a rainbow, and
        \item for any $g \in G$ and pair $(K, H)$ in $S$, if $(K^g, H^g) \neq (K, H)$ then $(K^g, H^g)$ does not belong to $S$. In other words, $S$ contains at most one pair from each conjugacy class of pairs of subgroups of $G$.
    \end{enumerate}
\end{definition}

\begin{proposition}
\label{proposition-rainbowsareminimal}
    Let $G$ be a finite group and let $S$ be a partial rainbow on $G$. Then the set $S$ is a minimal generating set of the transfer system $\langle S \rangle$.
\end{proposition}
\begin{proof}
    Assume for contradiction that there exists a pair $a \in S$ with $a = (K, H)$ such that $\langle S \setminus a \rangle = \langle S \rangle$. By ~\cref{lemma-generatingedge} we obtain a pair $(K', H') \in S \setminus a$ and $g \in G$ such that $K \leqslant (K')^g$ and $H \leqslant (H')^g$.

    Since the image of $S$ under $P \colon \Sub(G) \to \NN$ is a rainbow, we in fact know that $H = (H')^g$ and $K = (K')^g$, otherwise the arcs of $(K, H)$ and $(K', H')$ which both belong to $S$ would intersect in $\NN$. With Condition~(2) of \cref{definition-partialrainbow} we obtain that $g = e$, and $K = K'$, $H = H'$, reaching a contradiction since $(K', H') \in S \setminus a$.
\end{proof}

While \cref{prop:all-same-size} tells us that all minimal generating sets for a fixed transfer system have the same size, we can say something even stronger in the case of partial rainbows:

\begin{proposition}
    Let $S,$ $S'$ be two partial rainbows for the finite group $G$. If the generated $G$-transfer systems $\langle S \rangle$ and $\langle S' \rangle$ are equal, then for each pair $(K, H)$ in $S$, there exists a $g \in G$ such that $(K^g, H^g) \in S'$. That is, $S$ and $S'$ coincide up to conjugation.
\end{proposition}
\begin{proof}
    Let $(K, H)$ be a pair in $S$. Since $(K, H) \in \langle S \rangle = \langle S' \rangle$, by \cref{lemma-generatingedge} we obtain a pair $(K', H') \in S'$ and $g \in G$ such that $K \leqslant (K')^{g}$ and $H \leqslant (H')^g$. Now $(K', H') \in \langle S' \rangle = \langle S \rangle$, so applying \cref{lemma-generatingedge} again we obtain $(K'', H'') \in S$ and $g' \in G$ such that $K' \leqslant (K'')^{g'}$ and $H' \leqslant (H'')^{g'}$. Since $S$ is a partial rainbow, we know that $K = (K')^g = (K'')^{gg'}$ and $K = (H')^g = (H'')^{gg'}$, otherwise the respective arcs of $(K, H)$ and  $(K'', H'')$ would intersect in $\NN$.
\end{proof}

\begin{remark}
    While every partial rainbow gives a minimal generating set, it should be noted that there are transfer systems whose minimal generating sets cannot be realized as a partial rainbow. Such an example is already provided in \cref{ex:diff-min}. Indeed, for each choice of minimal generating set we observe that we do not get a partial rainbow under the mapping $P$.
\end{remark}

\section{Complexity and width of groups}

In \cref{cor:well-defined-map} we introduced the well-defined map $\mathfrak{m} \colon \mathrm{Tr}(G) \to \NN$ which sends a transfer system to the cardinality of one of its minimal generating sets. We immediately extract two numerical invariants.

\begin{definition}
Let $G$ be a finite group.
\begin{enumerate}
    \item The \emph{width} of $G$, denoted $ \w (G)$, is $\mathfrak{m}(\mathsf{T}_C)$ where $\mathsf{T}_C$ is the complete transfer system for $G$.
    \item The \emph{complexity} of $G$, denoted $\mathfrak{c}(G)$, is $\max\{\mathfrak{m}(\mathsf{T}) \mid \mathsf{T} \in \mathsf{Tr}(G)\}$.
\end{enumerate}
\end{definition}

\begin{example}
    Let $G = C_{p^n}$. Then the width and complexity of $G$ are computed in \cite[\S 5]{fooqw} as $ \w (G) = \mathfrak{c}(G) = n$. A minimal generating set realizing both of these can be taken as $\{(C_{p^i}, C_{p^{i+1}}) \mid 0 \leqslant i < n\}$.
\end{example}

Before continuing with a more exotic example, we should discuss how, in practice, one can access these invariants. That is, given a transfer system $\mathsf{T} \in \mathsf{Tr}(G)$, how does one find a minimal generating set for $\mathsf{T}$? In \cite[\S 5]{fooqw} an algorithm for minimal generating sets for transfer systems for $C_{p^n}$ is developed using the notion of \emph{maximal arrows}. However, one quickly sees that this algorithm breaks down for other groups. We shall instead once again consider \cref{cons:rubin}, and note that it is possible to run this algorithm backwards. In short, we have the following:

\begin{construction}[Reverse Rubin algorithm]\label{cons:rev-rubin}
    Let $G$ be a finite group and $\mathsf{T} \in \mathsf{Tr}(G)$. Define
    \begin{align*}
        T_0 \coloneqq & \, \mathsf{T},\\
        T_{-1} \coloneqq & \, \{(K,H) \in T_0 \mid (K,H) \text{ cannot be written as a composite of elements of }T_0\} ,\\
        T_{-2} \coloneqq & \, \{(K,H) \in T_{-1} \mid (K,H) \text{ is not of the form } (L \cap K, K) \text{ for some } L \leqslant H\},\\
        T_{-3} \coloneqq & \, T_{-2} / G \text{ where } G \text{ acts via conjugation}.
    \end{align*}
    Then $T_{-3}$ is a minimal generating set for $\mathsf{T}$.
\end{construction}

Up to making a choice of representative in each conjugacy class, this algorithm produces a preferred minimal generating set for a given transfer system. In general, there seem to be many minimal generating sets for a given transfer system, even when $G$ is abelian (in which case there is no choice in the reverse Rubin algorithm). We feel that there is more mileage to be gained from the existence of this preferred minimal generating set, beyond our own uses of it.

Let us now provide a simple example which demonstrates a computation of the width and complexity, as well as some other somewhat surprising phenomena.

\begin{example}\label{ex:inclusion_not_respected}
    Let $G = C_{p^2q}$ for $p$ and $q$ distinct primes. 

We can consider the two following sets and the transfer systems that they generate under \cref{cons:rubin} (noting that we can retrieve the given minimal generating sets using \cref{cons:rev-rubin}):
\[
S_a = \qquad
\begin{gathered}
\begin{tikzpicture}[scale= 0.7]
\node[fill=dark-green!50,circle,draw,inner sep = 0pt, outer sep = 0pt, minimum size=1.5mm] (1) at (0,0) {};
\node[fill=dark-green!50,circle,draw,inner sep = 0pt, outer sep = 0pt, minimum size=1.5mm] (p) at (2,0) {};
\node[fill=dark-green!50,circle,draw,inner sep = 0pt, outer sep = 0pt, minimum size=1.5mm] (p2) at (4,0) {};
\node[fill=dark-green!50,circle,draw,inner sep = 0pt, outer sep = 0pt, minimum size=1.5mm] (q) at (0,2) {};
\node[fill=dark-green!50,circle,draw,inner sep = 0pt, outer sep = 0pt, minimum size=1.5mm] (qp) at (2,2) {};
\node[fill=dark-green!50,circle,draw,inner sep = 0pt, outer sep = 0pt, minimum size=1.5mm] (qp2) at (4,2) {};
\draw[->, black!70] (1) edge (p);
\draw[->, black!70] (1) edge[bend right = 40] (p2);
\draw[->, black!70] (1) edge (qp);
\draw[->, black!70] (1) edge (qp2);
\draw[->, black!70] (1) edge (q);
\draw[->, black!70] (p) edge (p2);
\draw[->, black!70] (p) edge (qp);
\draw[->, black!70] (p) edge (qp2);
\draw[->, black!70] (q) edge (qp);
\draw[->, black!70] (q) edge[bend left = 40] (qp2);
\draw[->, black!70] (qp) edge (qp2);
\draw[->, black!70] (p2) edge (qp2);
\draw[->,line join=round,decorate, decoration={zigzag,segment length=4,amplitude=.9,post=lineto,post length=2pt}]  (-3.2,1) -- (-0.8,1);
\node[fill=white] at (-2,1) {$\langle - \rangle$};
\begin{scope}[xshift = -8cm]
    \node[fill=dark-green!50,circle,draw,inner sep = 0pt, outer sep = 0pt, minimum size=1.5mm] (1) at (0,0) {};
\node[fill=dark-green!50,circle,draw,inner sep = 0pt, outer sep = 0pt, minimum size=1.5mm] (p) at (2,0) {};
\node[fill=dark-green!50,circle,draw,inner sep = 0pt, outer sep = 0pt, minimum size=1.5mm] (p2) at (4,0) {};
 \node[fill=dark-green!50,circle,draw,inner sep = 0pt, outer sep = 0pt, minimum size=1.5mm] (q) at (0,2) {};
\node[fill=dark-green!50,circle,draw,inner sep = 0pt, outer sep = 0pt, minimum size=1.5mm] (qp) at (2,2) {};
\node[fill=dark-green!50,circle,draw,inner sep = 0pt, outer sep = 0pt, minimum size=1.5mm] (qp2) at (4,2) {};
\draw[->, black!70] (q) edge (qp);
\draw[->, black!70] (qp) edge (qp2);
\draw[->, black!70] (p2) edge (qp2);
\end{scope}
\end{tikzpicture}
\end{gathered}
\qquad = \langle S_a \rangle
\]
\vspace{5mm}
\[
S_b = \qquad
\begin{gathered}
\begin{tikzpicture}[scale= 0.7]
\node[fill=dark-green!50,circle,draw,inner sep = 0pt, outer sep = 0pt, minimum size=1.5mm] (1) at (0,0) {};
\node[fill=dark-green!50,circle,draw,inner sep = 0pt, outer sep = 0pt, minimum size=1.5mm] (p) at (2,0) {};
\node[fill=dark-green!50,circle,draw,inner sep = 0pt, outer sep = 0pt, minimum size=1.5mm] (p2) at (4,0) {};
\node[fill=dark-green!50,circle,draw,inner sep = 0pt, outer sep = 0pt, minimum size=1.5mm] (q) at (0,2) {};
\node[fill=dark-green!50,circle,draw,inner sep = 0pt, outer sep = 0pt, minimum size=1.5mm] (qp) at (2,2) {};
\node[fill=dark-green!50,circle,draw,inner sep = 0pt, outer sep = 0pt, minimum size=1.5mm] (qp2) at (4,2) {};
\draw[->, black!70] (1) edge (p);
\draw[->, black!70] (1) edge[bend right = 40] (p2);
\draw[->, black!70] (1) edge (qp);
\draw[->, black!70] (1) edge (qp2);
\draw[->, black!70] (1) edge (q);
\draw[->, black!70] (p) edge (p2);
\draw[->, black!70] (p) edge (qp);
\draw[->, black!70] (q) edge (qp);
\draw[->,line join=round,decorate, decoration={zigzag,segment length=4,amplitude=.9,post=lineto,post length=2pt}]  (-3.2,1) -- (-0.8,1);
\node[fill=white] at (-2,1) {$\langle - \rangle$};
\begin{scope}[xshift = -8cm]
    \node[fill=dark-green!50,circle,draw,inner sep = 0pt, outer sep = 0pt, minimum size=1.5mm] (1) at (0,0) {};
\node[fill=dark-green!50,circle,draw,inner sep = 0pt, outer sep = 0pt, minimum size=1.5mm] (p) at (2,0) {};
\node[fill=dark-green!50,circle,draw,inner sep = 0pt, outer sep = 0pt, minimum size=1.5mm] (p2) at (4,0) {};
\node[fill=dark-green!50,circle,draw,inner sep = 0pt, outer sep = 0pt, minimum size=1.5mm] (q) at (0,2) {};
\node[fill=dark-green!50,circle,draw,inner sep = 0pt, outer sep = 0pt, minimum size=1.5mm] (qp) at (2,2) {};
\node[fill=dark-green!50,circle,draw,inner sep = 0pt, outer sep = 0pt, minimum size=1.5mm] (qp2) at (4,2) {};
\draw[->, black!70] (q) edge (qp);
\draw[->, black!70] (p) edge (qp);
\draw[->, black!70] (p) edge (p2);
\draw[->, black!70] (1) edge (qp2);
\end{scope}
\end{tikzpicture}
\end{gathered}
\qquad
= \langle S_b \rangle
\]

From the generating set $S_a$ we can conclude that $ \w (C_{p^2q}) = 3$, and computationally we verify that $S_b$ is has maximal size among all generating sets, and as such $\mathfrak{c}(C_{p^2q}) = 4$. Hence we immediately see, perhaps contrary to intuition, that the complete transfer system does not realize the complexity of the group (and in fact we shall  see in \cref{sec:subes} that the complexity and width can differ arbitrarily).
\end{example}

While we shall see that the complexity of a group is mysterious, the width has a tangible interpretation in terms of the group itself. We recall that a subgroup $H \leqslant G$ is \emph{meet-irreducible} if it cannot be expressed as the intersection of two distinct subgroups of the group. This happens if and only if the subgroup, when viewed as an element of $\Sub(G)$, is meet-irreducible in the lattice theoretic sense. A key lattice theoretic fact that we will require is that every element of a lattice can be obtained as the meet of meet-irreducible elements.

\begin{proposition}\label{prop:width}
    Let $G$ be a finite group. Then $ \w (G)$ is equal to the number of conjugacy classes of meet-irreducible subgroups of $G$. In particular, the collection
    \[S = \{(I,G) \mid I \leqslant G \text{ is meet-irreducible}\} / G\]
    is a minimal generating set for the complete transfer system.
\end{proposition}

\begin{proof}
    We begin with the observation that the collection $S_0 = \{(H,G) \mid H \leqslant G\}$ is a generating set for the complete transfer system. Indeed, pick an arbitrary $(H,K) \in \mathbb{I}(\Sub(G))$. Then we can obtain $(H,K)$ from $(H,G)$ via the restriction rule for transfer systems.

    It now suffices to prove that the claimed set $S$ is the minimal generating set contained in $S_0$. Let $H$ be a subgroup that is not meet-irreducible. Then $H = \bigcap_\lambda I_\lambda$ for $I_\lambda$ some meet-irreducible subgroups. Then the element $(H,G)$ can be obtained via pullback closure of the collection $\{ (I_\lambda,G) \}_\lambda$. We moreover observe that none of the $(I,G)$ can be obtained through compositions or restrictions. Finally, we are free to pick a single conjugacy class of each meet-irreducible subgroup. Hence $S$ is a minimal generating set as claimed.
\end{proof}

\begin{remark}
    The generating set $S$ given in \cref{prop:width} may not be the one that would be produced by \cref{cons:rev-rubin}, however, in light of \cref{prop:all-same-size} this does not affect the computation of the width.
\end{remark}

\begin{example}\label{ex:s5-width}
    Consider $G= S_5$. Then the following is a list of (conjugacy classes of) meet-irreducible subgroups of $S_5$:
    \[
        [F_5] \cong [C_5 \rtimes C_4], \, A_5, \, [S_4], \, [D_6], \, [D_4], \, [C_6], \, [C_5]. 
    \]
    As such we conclude that $ \w (S_5) = 7$. We can visualize the generating set $S$ provided in \cref{prop:width} on the poset $\Sub(G)/G$ as in \cref{fig:s51}. We see, for example, that the element $(C_4, S_5)$ can be obtained by taking the pullback-closure of $(F_5,S_5)$ and $(D_4,S_5)$.
    
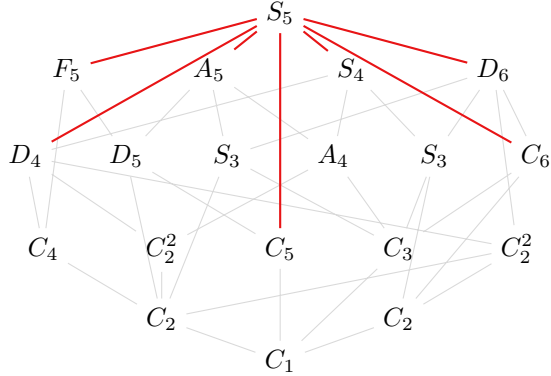
\begin{figure}[h]
\centering
  \begin{tikzpicture}[xscale=0.9, scale=0.7]
  \node at (5.5,0) (1) {$C_1$};
  \node at (8,0.803) (2) {$C_2$};
  \node at (3,0.803) (3) {$C_2$};
  \node at (8,2.14) (4) {$C_3$};
  \node at (5.5,2.14) (5) {$C_5$};
  \node at (3,2.14) (6) {$C_2^2$};
  \node at (10.5,2.14) (7) {$C_2^2$};
  \node at (0.5,2.14) (8) {$C_4$};
  \node at (10.9,3.89) (9) {$C_6$};
  \node at (8.75,3.89) (10) {$S_3$};
  \node at (4.38,3.89) (11) {$S_3$};
  \node at (2.25,3.89) (12) {$D_5$};
  \node at (0.125,3.89) (13) {$D_4$};
  \node at (6.62,3.89) (14) {$A_4$};
  \node at (10,5.54) (15) {$D_6$};
  \node at (1,5.54) (16) {$F_5$};
  \node at (7,5.54) (17) {$S_4$};
  \node at (4,5.54) (18) {$A_5$};
  \node at (5.5,6.62) (19) {$S_5$};
  \draw[thin, black!15] (1)--(2) (1)--(3) (1)--(4) (1)--(5) (3)--(6) (2)--(7) (3)--(7)
     (3)--(8) (2)--(9) (4)--(9) (2)--(10) (4)--(10) (3)--(11) (4)--(11)
     (3)--(12) (5)--(12) (6)--(13) (7)--(13) (8)--(13) (4)--(14) (6)--(14)
     (7)--(15) (9)--(15) (10)--(15) (11)--(15) (12)--(16) (8)--(16) (13)--(17)
     (14)--(17) (10)--(17) (12)--(18) (14)--(18) (11)--(18) (15)--(19) (16)--(19)
     (17)--(19) (18)--(19);
    \draw[rainbow-red,  thick] (18) -- (19);
    \draw[rainbow-red,  thick] (17) -- (19);
    \draw[rainbow-red,  thick] (16) -- (19);
    \draw[rainbow-red,  thick] (15) -- (19);
    \draw[rainbow-red,  thick] (13) -- (19);
    \draw[rainbow-red,  thick] (9) -- (19);
    \draw[rainbow-red,  thick] (5) -- (19);
    \end{tikzpicture}
    \caption{A minimal generating set for the complete transfer system for $S_5$.}\label{fig:s51}
\end{figure}

To obtain from this the generating set which would be obtained via \cref{cons:rev-rubin}, one needs to consider removing elements which are better obtained via composites. For example, the element $(D_4,S_5)$ would be replaced by $(D_4,S_4)$ as it can constructed as the composite of $(D_4,S_4)$ and $(S_4,S_5)$.
\end{example}

\begin{remark}\label{rem:F8spicy}
    The width of $G$ is measured as conjugacy classes of meet-irreducible subgroups, which is equivalent to conjugacy classes of meet-irreducible elements of $\Sub(G)$. Recall that for a lattice the meet-irreducible elements are in bijection with those elements which are covered by exactly one element. Although $\Sub(G)/G$ need not be a lattice, one may ask if the width of $G$ can be obtained by considering those elements of $\Sub(G)/G$ which are covered by singletons. 
    
    This is not the case: Consider the group $F_8 = C_2^3 \rtimes C_7$. Then $ \w (F_8) = 3$ as can be seen from \cref{fig:f8subnoconj}. Indeed, the conjugacy classes of meet-irreducible subgroups are $C_2^3$, $[C_7]$, and $[C_2^2]$. However, in $\Sub(F_8)/F_8$ (\cref{fig:f8sub}) one observes that there are 4 elements with one cover as all copies of $C_2^2$ and $C_2$ are conjugate. This raises the question of whether there is a characterization of those groups for which this issue does not arise, similar to the study undertaken in \cite{BMO_lift}. We return to this question in more detail in \cref{subsec:lifting}.
\begin{figure}[h]
\centering
\begin{minipage}{.45\textwidth}
  \centering
  \begin{tikzpicture}[yscale=0.85, xscale=0.9]
  \node[circle] at (1,0) (1) {$C_1$};
  \node at (0.9,0.81) (21) {$C_2$};
  \node at (1.5,0.81) (22) {$C_2$};
  \node at (2.1,0.81) (23) {$C_2$};
  \node at (2.7,0.81) (24) {$C_2$};
  \node at (3.3,0.81) (25) {$C_2$};
  \node at (3.9,0.81) (26) {$C_2$};
  \node at (4.5,0.81) (27) {$C_2$};
  \node at (0.125,1.61) (31) {$C_7$};
  \node at (0.125 - 0.5,1.61) (32) {$C_7$};
  \node at (0.125 - 1,1.61) (33) {$C_7$};
  \node at (0.125 - 1.5,1.61) (34) {$C_7$};
  \node at (0.125 - 2,1.61) (35) {$C_7$};
  \node at (0.125 - 2.5,1.61) (36) {$C_7$};
  \node at (0.125 - 3,1.61) (37) {$C_7$};
  \node at (0.125 - 3.5,1.61) (38) {$C_7$};
  \node at (0.9,2) (41) {$C_2^2$}; 
  \node at (1.5,2) (42) {$C_2^2$}; 
  \node at (2.1,2) (43) {$C_2^2$}; 
  \node at (2.7,2) (44) {$C_2^2$}; 
  \node at (3.3,2) (45) {$C_2^2$}; 
  \node at (3.9,2) (46) {$C_2^2$}; 
  \node at (4.5,2) (47) {$C_2^2$}; 
  \node[circle] at (2.5,2.8) (5) {$C_2^3$}; 
  \node[circle] at (1,3.21) (6) {$F_8$};
  \draw[thin, black!55] (5)--(6);
  \draw[thin, black!55] (1)--(31) (1)--(32) (1)--(33) (1)--(34) (1)--(35) (1)--(36) (1)--(37) (1)--(38);
  \draw[thin, black!55] (6)--(31) (6)--(32) (6)--(33) (6)--(34) (6)--(35) (6)--(36) (6)--(37) (6)--(38);
  \draw[thin, black!55] (21)--(41) (21)--(42) (21)--(46);
  \draw[thin, black!55] (22)--(42) (22)--(43) (22)--(47);
  \draw[thin, black!55] (23)--(43) (23)--(44) (23)--(41);
  \draw[thin, black!55] (24)--(44) (24)--(45) (24)--(42);
  \draw[thin, black!55] (25)--(45) (25)--(46) (25)--(43);
  \draw[thin, black!55] (26)--(46) (26)--(47) (26)--(44);
  \draw[thin, black!55] (27)--(47) (27)--(41) (27)--(45);
  \draw[thin, black!55] (5)--(41) (5)--(42) (5)--(43) (5)--(44) (5)--(45) (5)--(46) (5)--(47);
  \draw[thin, black!55] (1)--(21) (1)--(22) (1)--(23) (1)--(24) (1)--(25) (1)--(26) (1)--(27);
\end{tikzpicture}
  \captionof{figure}{$\Sub(F_8)$}\label{fig:f8subnoconj}
  \label{fig:test1}
\end{minipage}
\begin{minipage}{.5\textwidth}
  \centering
      \begin{tikzpicture}[yscale=0.85,xscale=0.9]
  \node at (1,0) (1) {$C_1$};
  \node at (1.5,0.81) (2) {$[C_2]$};
  \node at (-0.125,1.61) (3) {$[C_7]$};
  \node at (1.5,2) (4) {$[C_2^2]$};
  \node[circle] at (2.5,2.8) (5) {$C_2^3$};
  \node[circle] at (1,3.21) (6) {$F_8$};
  \draw[thin, black!55] (1)--(2) (1)--(3) (2)--(4) (4)--(5) (3)--(6) (5)--(6);
\end{tikzpicture}
  \captionof{figure}{$\Sub(F_8)/F_8$}\label{fig:f8sub}
  \label{fig:test2}
\end{minipage}
\end{figure}
\end{remark}

\begin{remark}
    Although we have an intrinsic description of the width of a group, it does not follow that this invariant is always easy to compute. For example we do not even have a complete description of the number of maximal subgroups of $S_n$ up to conjugacy, only asymptotic behavior, and this relies on the Classification of Finite Simple Groups and the O'Nan--Scott Theorem \cite{max_subgroups}. We would expect the calculation of the number of conjugacy classes of meet-irreducible subgroups to be at least as hard as this calculation. Dave Benson kindly computed the first few terms of the resulting sequence as $1,2,5,7,14,23,41,57,101,139, \dots$. This sequence does not appear in the OEIS.
\end{remark}

\clearpage

\part{Examples of generation combinatorics}\label{part:combi}

In \cref{part:theory} we introduced the concept of minimal generation for a transfer system. Of particular note was the definition of the width and complexity of a group. In this part, we will investigate the combinatorics associated to minimal generation for several families of groups.

\section{Complexity of cyclic groups of square-free order}\label{sec:subes}

In this section we will consider the generation combinatorics of square-free cyclic groups, that is, groups of the form $G = C_{p_1 \cdots p_n}$ for $p_i$ distinct primes. In this case, $\Sub(G)$ is isomorphic to the Boolean lattice $[1]^n$. 

\begin{lemma}
    Let $G = C_{p_1 \cdots p_n}$. Then $ \w (G) = n$.
\end{lemma}

\begin{proof}
Write $C_{p_1 \cdots \widehat{p_i} \cdots p_n}$ for the unique subgroup of $G$ which is of index $p_i$. Then the meet-irreducible subgroups of $G$ are simply the maximal subgroups which are those of the form $C_{p_1 \cdots \widehat{p_i} \cdots p_n}$. There are exactly $n$ of these; we conclude using \cref{prop:width}.
\end{proof}

\subsection{Symmetric minimal generating sets}\label{rem:symm}

We now move to the far more lofty goal of describing the complexity of $G$. To prove this we will appeal to the theory of rainbows as introduced in \cref{sec:rainbows}. The proof will proceed by constructing the largest possible partial rainbow in the sense of \cref{definition-partialrainbow}.

Before we continue, however, it is worth highlighting the role that rainbows play in this setting. Any transfer system for $G$ has a natural action of $S_n$ which permutes the prime factors in the subgroups. We shall say that a transfer system has a \emph{symmetric minimal generating set} if it has a generating set which is fixed under this action. It is clear to see that such symmetric sets correspond exactly to partial rainbows in this setting. 

The authors are conflicted as to whether or not any transfer system which realizes the complexity for $G$ must have a symmetric minimal generating set and thus be represented by a partial rainbow; it is true in every computable case, although we have not been able to construct a proof of this. As such, our main theorem of this section results in a lower bound for the complexity of $G$, and, if one could prove the aforementioned statement regarding a symmetric generating set, then this inequality is in fact an equality.

We reiterate however, that the combinatorics required to compute the size of the maximal rainbow seem to be where the true complexity of this calculation lies, and this is what is presented in \cref{subsec:maxrainbow}.

\subsection{Maximally sized rainbows}\label{subsec:maxrainbow}

In light of \cref{rem:symm} we will proceed by constructing the maximally sized rainbow for $G$.  

Let $G=C_{p_1p_2\cdots p_n}$. The following arc, denoted $\underline{1} \to \underline{n-1}$, is the rainbow which represents all possible arrows between the subgroups of the form $C_{p_i}$ and the subgroups of the form $C_{p_1\cdots \hat{p_i}\cdots p_n}$, i.e. subgroups with order the product of $n-1$ primes. 
\[
\begin{tikzcd}
    0&1\ar[rrr,bend left,no head]&2&\cdots&n-1&n
\end{tikzcd}
\]
More generally, the following arc, denoted $\underline{x} \to \underline{y}$, represents all possible arrows between the subgroups generated by $x$ primes and the subgroups generated by $y$ primes. Each arc represents $\binom{n}{x}\times\binom{n-x}{y-x}$ arrows; this number is also known as the trinomial coefficient $\binom{n}{x,y-x,n-y}$.

\[
\begin{tikzcd}
    0&\cdots&x\ar[rr,bend left,no head]&\cdots&y&\cdots&n
\end{tikzcd}
\]

A rainbow is a set of such arcs, $R= \{\alpha_i: \ul{x_i} \to \ul{y_i}\}_{0 \leq i \leq m}$, where $x_0<x_1<\cdots<x_m<y_m<\cdots<y_1<y_0$. The number of arrows in the minimal generating set represented by this rainbow is given by
\[
\sum_{i=0}^m|\alpha_i|=\sum_{i=0}^m\binom{n}{x_i,y_i-x_i,n-y_i}
\]

We refer to this number as the \emph{size} of a rainbow and phrases like ``largest rainbow'' are always consistent with this notion of size.

\begin{definition} \label{def:maxRb}\leavevmode
For odd $n$, let the \emph{complete rainbow} be the rainbow where all classes from $\ul{0}$ to $\ul{n}$ appear as endpoints of arcs in the rainbow.

For even $n$ and $0 \leqslant x \leqslant n$, let $R[x]$ be the rainbow with $\frac{n}{2}$ arcs containing all classes but $\underline{x}$ as arc endpoints.
\end{definition}

We provide examples of these rainbows in \cref{fig:exampleofrainbows}; our main goal in this subsection is to show that these rainbows are the largest---i.e. correspond to the greatest number of arrows in $\Sub(G)$---among all other rainbows.

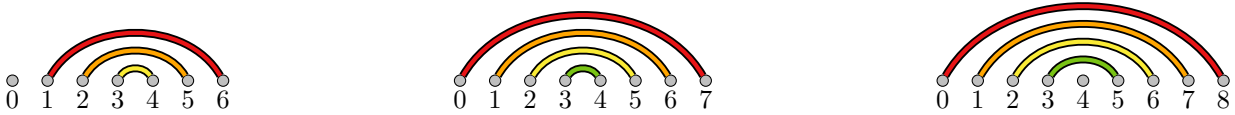
\begin{figure}[h!]
\centering
\begin{tikzpicture}[scale=0.7]
\foreach \i in {0,...,6}
{
    \node[fill=black!25,circle,draw,inner sep = 0pt, outer sep = 0pt, minimum size=1.5mm] (\i) at (\i/1.5,0) {};
    \node at (\i/1.5,-0.35) {\i};
}
\draw[-, line width=3pt, black] (1) edge[bend left=60] (6);
\draw[-, line width=3pt, black] (2) edge[bend left=60] (5);
\draw[-, line width=3pt, black] (3) edge[bend left=60] (4);
\draw[-, ultra thick, rainbow-red] (1) edge[bend left=60] (6);
\draw[-, ultra thick, rainbow-orange] (2) edge[bend left=60] (5);
\draw[-, ultra thick, rainbow-yellow] (3) edge[bend left=60] (4);
\end{tikzpicture}
\hfill 
\begin{tikzpicture}[scale=0.7]
\foreach \i in {0,...,7}
{
    \node[fill=black!25,circle,draw,inner sep = 0pt, outer sep = 0pt, minimum size=1.5mm] (\i) at (\i/1.5,0) {};
    \node at (\i/1.5,-0.35) {\i};
}
\draw[-, line width=3pt, black] (0) edge[bend left=60] (7);
\draw[-, line width=3pt, black] (1) edge[bend left=60] (6);
\draw[-, line width=3pt, black] (2) edge[bend left=60] (5);
\draw[-, line width=3pt, black] (3) edge[bend left=60] (4);
\draw[-, ultra thick, rainbow-red] (0) edge[bend left=60] (7);
\draw[-, ultra thick, rainbow-orange] (1) edge[bend left=60] (6);
\draw[-, ultra thick, rainbow-yellow] (2) edge[bend left=60] (5);
\draw[-, ultra thick, rainbow-green] (3) edge[bend left=60] (4);
\end{tikzpicture}
\hfill 
\begin{tikzpicture}[scale=0.7]
\foreach \i in {0,...,8}
{
    \node[fill=black!25,circle,draw,inner sep = 0pt, outer sep = 0pt, minimum size=1.5mm] (\i) at (\i/1.5,0) {};
    \node at (\i/1.5,-0.35) {\i};
}
\draw[-, line width=3pt, black] (0) edge[bend left=60] (8);
\draw[-, line width=3pt, black] (1) edge[bend left=60] (7);
\draw[-, line width=3pt, black] (2) edge[bend left=60] (6);
\draw[-, line width=3pt, black] (3) edge[bend left=60] (5);
\draw[-, ultra thick, rainbow-red] (0) edge[bend left=60] (8);
\draw[-, ultra thick, rainbow-orange] (1) edge[bend left=60] (7);
\draw[-, ultra thick, rainbow-yellow] (2) edge[bend left=60] (6);
\draw[-, ultra thick, rainbow-green] (3) edge[bend left=60] (5);
\end{tikzpicture}
\caption{Rainbows whose corresponding minimal generating sets realize the complexity: $R[0]$ for $n=6$, the complete rainbow for $n=7$, and $R[4]$ for $n=8$.}\label{fig:exampleofrainbows}
\end{figure}

\vspace{-5mm}

\begin{definition}
    We call a rainbow $R$ \emph{composable} if there exists a rainbow $R \cup \{ \alpha \}$ for some arc $\alpha$. We say that $R \cup \{ \alpha \}$ is the \emph{composition} of $R$ with $\alpha$.
\end{definition}

We shall now construct a sequence of operations from any rainbow with fewer than $\lvert \frac{n}{2} \rvert$ to a composable rainbow, showing the largest rainbow must have the maximum number of arcs.  Some of these operations will be performed on single arcs, but it will also be necessary to perform operations on {blocks} of arcs. 

\begin{definition}
    A \emph{block} of arcs within a rainbow $R=\{\ul{x_i} \to \ul{y_i}\}_{0 \leqslant i \leqslant m} $ is a consecutive subset of the arcs of $R$, i.e., a set of the form  $B= \{ \ul{x_i} \to \ul{y_i} \}_{I \leqslant i \leqslant J}$. Note $x_i \geqslant x_I + (i-I)$ and $y_i \leqslant y_I - (i-I)$ for all $I \leqslant i \leqslant J$.

    We say a class $\ul{x}$ is \emph{available} if no arc in $R$ has $\ul{x}$ as an endpoint.
    A block $B$ is \emph{full on the left} if $x_i = x_I + (i-I)$ for all $I \leqslant i \leqslant J$ (that is, no classes between $\ul{x_I}$ and $\ul{x_J}$ are available).

    Likewise, $B$ is \emph{full on the right} if $y_i = y_I - (i-I)$ for all $I \leqslant i \leqslant J$ (no class between $\ul{y_J}$ and $\ul{y_I}$ is available) and $B$ is \emph{full} if it is full on both the left and the right. 

We define the following important blocks for a given rainbow $R$:
\begin{itemize}
    \item the \emph{left block} of $R$, denoted $B_{\ell}(R)$ or $B_\ell$ when $R$ is clear from context, is the largest block that contains the outermost arc, $\ul{x_0} \to \ul{y_0}$, and is full on the left,
    \item the \emph{right block} of $R$, denoted $B_{r}(R)$ or $B_r$, is the largest block that contains the outermost arc and is full on the right, 
    \item the \emph{outer block} of $R$, denoted $B_{\mathcal{O}}(R)$ or $B_{\mathcal{O}}$, is $B_\ell \cap B_r$.
\end{itemize}
An example of these three blocks is depicted in \cref{fig:blocks}.
\end{definition}

\begin{figure}[h!]
\centering
\begin{tikzpicture}[scale=0.65]
\foreach \i in {0,...,12}
{
    \node[fill=black!25,circle,draw,inner sep = 0pt, outer sep = 0pt, minimum size=1.5mm] (\i) at (\i/1.5,0) {};
    \node at (\i/1.5,-0.35) {\i};
}
\draw[-, line width=3pt, black] (0) edge[bend left=60] (12);
\draw[-, line width=3pt, black] (1) edge[bend left=60] (11);
\draw[-, line width=3pt, black,dashed] (3) edge[bend left=60] (10);
\draw[-, line width=3pt, black,dashed] (5) edge[bend left=60] (9);
\draw[-, line width=3pt, black] (6) edge[bend left=60] (7);
\draw[-, ultra thick, rainbow-red] (0) edge[bend left=60] (12);
\draw[-, ultra thick, rainbow-red] (1) edge[bend left=60] (11);
\draw[-, ultra thick, rainbow-yellow,dashed] (3) edge[bend left=60] (10);
\draw[-, ultra thick, rainbow-yellow,dashed] (5) edge[bend left=60] (9);
\draw[-, ultra thick, rainbow-blue] (6) edge[bend left=60] (7);
\end{tikzpicture}
\caption{The block $B_\ell$ is the two outer arcs (solid red).  The block $B_r$ is the four outer arcs (solid red and dashed yellow). In this example, $\ul{a+1} = \ul{2}$ is available, so $B_\mathcal{O} = B_\ell$.}\label{fig:blocks}
\end{figure}
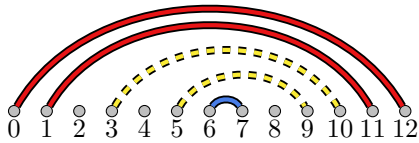

\begin{proposition}
    For any rainbow $R$, the outer block $B_\mathcal{O}$ is equal to $B_\ell$ or $B_r$.
\end{proposition}
\begin{proof}
    Let $\ul{a} \to \ul{b}$ be the innermost arc of $B_{\mathcal{O}}$.   If $a+1 \geqslant b-1$, there cannot be any arc $\ul{x} \to \ul{y}$ with $a < x < y <b$.  So, all arcs of $R$ are included in $B_\mathcal{O}$ and thus $B_\mathcal{O} = B_\ell = B_r$.

    If $a+1 < b-1$, then the classes $\ul{a+1}$ and $\ul{b-1}$ cannot both be unavailable, or else $\ul{a+1} \to \ul{b-1}$ would be included in $B_{\mathcal{O}}$. Thus $\ul{a+1}$ is available, which implies $B_{\mathcal{O}} = B_\ell$ or $\ul{b-1}$ is available, so that $B_{\mathcal{O}} = B_r$. 
\end{proof}

We are now ready to describe the operations that will be needed to transform a given rainbow into a rainbow with the maximum number of arcs. 
\begin{proposition}\label{prop:rainbow_ops}
    The following operations on rainbows are non-decreasing on the size of a rainbow: 

Operations on single arcs: 
\begin{enumerate}
    \item Composition (i.e. union) of disjoint arcs [strictly increasing],
    \item Translation of single arcs towards the center,
    \item Contraction of large single arcs towards the center,
    \item Expansion of small single arcs away from the center.
\end{enumerate}

Operations on blocks:
\begin{enumerate}[resume]
    \item Reflection of a rainbow,
    \item Translation of a block of arcs towards the center,
    \item Contraction of the left side of an outer block of arcs towards the center,
    \item Expansion of the left side of an inner block of arcs away from the center.
\end{enumerate}

\end{proposition}

\renewcommand*{\proofname}{Proof of \cref{prop:rainbow_ops} (1)--(4).}
\begin{proof}
First let us see that composition strictly increases the size of a rainbow.  

\begin{enumerate}
        \item \textbf{Composition:}  Consider any rainbow $R = \{ \alpha_i \}_{i=0}^m$.  Suppose $\ul{w}$ and $\ul{z}$ are available in $R$, with $w < z$.  If $x_m < w < z < y_m$, or $w < x_0 <y_0 < z$, or there exists some $j < m$ such that $x_j < w < x_{j+1}$ and $y_{j+1} < z < y_i$, then the composition, or union, of $R$ and the arc $ \ul{w} \to \ul{z}$ is a valid rainbow. 
    
        Using the notation $\widetilde{R}=\{\ul{w} \to \ul{z}\}\cup\{\alpha_i\}_{i=0}^m$ for this composition, we then have
    \[
|\widetilde{R}|=|\ul{w} \to \ul{z}|+\sum_{i=0}^m|\alpha_i|.
    \]
        Since $\ul{w} \to \ul{z}$ must represent at least one arrow in $\Sub(G)$, we conclude $|\widetilde{R}|>|R|$.
\end{enumerate}

The conditions for performing translating, contracting, or expanding an arc $\ul{x} \to \ul{y}$ can all be written as comparisons between the length of the arc ($y-x$) and the distance from the endpoints of the range to the endpoints of the arc ($x$ and $n-y$).
One way to think about these conditions is that non-decreasing operations move towards equality between these three quantities (the length and the distance to each endpoint from the end of the range). 

\begin{enumerate}[resume]

    \item \textbf{Translation of a single arc:}
    
    Consider $\alpha: \ul{x} \to \ul{y}$, an arc in $R$, where $\ul{x+1}$ and $\ul{y+1}$ are both available.  If $x + y < n$, then we claim replacing $\alpha$ with the arc $\ul{x+1} \to \ul{y+1}$ (i.e., translating $\alpha$ to the right by one unit) is non decreasing on the size of the rainbow.  

    To see this, note that in this case, $1 \leqslant \dfrac{n-y}{x+1}$. Hence
    \begin{align*} 
        |\alpha| & = \frac{n!}{x!(y-x)!(n-y)!}
        \leqslant \frac{n!}{x!(y-x)!(n-y)!} \cdot \frac{n-y}{x+1} \\
        & = \dfrac{n!}{(x+1)!(y-x)!(n-y-1)!}=|\ul{x+1} \to \ul{y + 1}|. 
    \end{align*} 

    By a symmetric argument, replacing $\alpha$ with the arc $\ul{x-1} \to \ul{y-1}$ (i.e., translating $\alpha$ to the left by one unit) when $\ul{x-1}$ and $\ul{y-1}$ are available is non-decreasing on the number of arrows represented by the corresponding rainbow when $x+y>n$.
    
    Note that it's impossible to satisfy both conditions at the same time.  The arcs that satisfy neither condition are arcs of the form $\underline{x} \to \underline{n-x}$.  When $n$ is not equal to 2, 4, or 6, we will see that these non-translatable arcs are precisely the arcs in the largest rainbow. \\

    \item \textbf{Contracting an arc:} Let $\alpha: \ul{x} \to \ul{y}$ be an arc of $R$ where $\ul{x+1}$ is available.   If $y - x > x$, then $\dfrac{y-x}{x+1} \geqslant 1$. Hence 
    \begin{align*}
        |\alpha| & = \dfrac{n!}{x!(y-x)!(n-y)!} \leqslant \dfrac{n!}{x!(y-x)!(n-y)!} \cdot \dfrac{y-x}{x+1} \\
        & =  \dfrac{n!}{(x+1)!(y-x-1)!(n-y)!} = |\ul{x+1} \to \ul{y}|.
    \end{align*}
    Replacing $\alpha$ with the arc $\ul{x+1} \to \ul{y}$ (i.e., contracting $\alpha$ towards the center on the left by one unit) is therefore non-decreasing on the number of arrows represented by $R$ when $y-x > x$.

    Likewise, replacing $\alpha$ with the arc $\ul{x} \to \ul{y-1}$ (i.e., contracting $\alpha$ towards the center on the right by one unit) when $\ul{y-1}$ is available is non-decreasing on the number of arrows represented by $R$ when $y-x > n-y$.
    
    \item \textbf{Expanding an arc:}
    
    Let $\alpha: \ul{x} \to \ul{y}$ be an arc of $R$ where $x-1$ is available.   If $y - x < x$, then $\dfrac{x}{y-x+1} \geqslant 1$. Hence 
    \begin{align*}
        |\alpha| & = \dfrac{n!}{x!(y-x)!(n-y)!} \leqslant \dfrac{n!}{x!(y-x)!(n-y)!} \cdot \dfrac{x}{y-x+1} \\
        & = \dfrac{n!}{(x-1)!(y-x+1)!(n-y)!} = |\ul{x-1} \to \ul{y}|.
    \end{align*}
    Replacing $\alpha$ with the arc $\ul{x-1} \to \ul{y}$ (i.e., expanding $\alpha$ away from the center on the left by one unit) is therefore non-decreasing on the number of arrows represented by $R$ when $y-x < x$.

    Likewise, replacing $\alpha$ with the arc $\ul{x} \to \ul{y-1}$ (i.e., expanding $\alpha$ the right by one unit) when $\ul{y+1}$ is available is non-decreasing on the number of arrows represented by $R$ when $y-x < n-y$.
    
        \end{enumerate}
\end{proof}

To summarize \cref{prop:rainbow_ops} (1)-(4), the operations given in \cref{table:singleArcTable} can be performed on $\ul{x} \to \ul{y}$ whenever the new arc's endpoints are available in the original rainbow and are non-decreasing on the corresponding number of arrows when the given condition is met.

\begin{table}[h]
    \centering
    \renewcommand{\arraystretch}{1.5}
\begin{tabular}{|l|c|c|} \hline 
    Operation name & Replace $\underline{x}\to \underline{y}$ with & Condition  \\ \hline \hline 
    Translate right & $\underline{x+1}\to \underline{y+1}$ & $x+y < n$ \\ \hline 
    Translate left & $\underline{x-1}\to \underline{y-1}$ & $ x+y > n$ \\ \hline 
    Contract right & $\underline{x}\to \underline{y-1}$ & $y-x > n-y$ \\ \hline 
    Contract left & $\underline{x+1}\to \underline{y}$ &  $y-x > x$\\ \hline 
    Expand right & $\underline{x}\to \underline{y+1}$ & $y-x < n-y$ \\ \hline 
    Expand left & $\underline{x-1}\to \underline{y}$ & $y-x < x$ \\ \hline 
\end{tabular}
    \caption{A summary of the operations on single arcs and the conditions under which the operations are non-decreasing. }
    \label{table:singleArcTable}
\end{table}

In some cases, it will be necessary or simply more efficient to simultaneously apply an operation to all arcs in a particular block; this is the subject of \cref{prop:rainbow_ops} (5)-(8). In order to prove some of these operations are non-decreasing on the corresponding number of arrows, we will need the following combinatorial lemmas.

\begin{lemma} \label{lem:ineq}
    For all odd $n \geqslant 3$ and all even $n \geqslant 8$,  
    \[ \sum_{i=0}^{ \lfloor \frac{n-1}{2} \rfloor } \binom{n}{i+1, n-2i-1, i} < \sum_{i=0}^{ \lfloor \frac{n-1}{2} \rfloor } \binom{n}{i, n-2i, i} . \]
    
\end{lemma}
\renewcommand*{\proofname}{Proof.}
\begin{proof}
First, consider the case where $n$ is even, so $\lfloor \frac{n-1}{2} \rfloor = \frac{n}{2} - 1$.  To prove the desired inequality, we will show the following difference is positive: 
            \begin{align*}
           \sum_{i=0}^{(n/2)-1} \binom{n}{i, n-2i, i} - \sum_{i=0}^{(n/2)-1} \binom{n}{i+1, n-2i-1, i} 
            &= \left(\sum_{i \geqslant 0} \binom{n}{i, n-2i, i} \right) - \binom{n}{\frac{n}{2}, 0, \frac{n}{2}} - \sum_{i\geqslant 0} \binom{n}{i+1, n-2i-1, i} \\
            &= \binom{n}{0}_2 - \binom{n}{\frac{n}{2}} - \binom{n}{1}_2 \\
            &= \gamma_n - \binom{n}{\frac{n}{2}}.
        \end{align*}
        Here $\binom{n}{k}_2$ is the coefficient of $t^{n+k}$ in $(1+t+t^2)^n$ and $\gamma_n$ is the $n^{th}$ Riordan number \cite[A005043]{oeis}.
        
        To show $\gamma_n - \binom{n}{\frac{n}{2}}$  is positive for all even $n \geqslant 8$, we proceed by induction.  For $n =8$, $\gamma_n = 91$ and $\binom{n}{n/2} = 70$, so $\gamma_n - \binom{n}{n/2} > 0$.  

        Now suppose $n > 8$ and $\gamma_n - \binom{n}{n/2} > 0$.  We will show that $\gamma_{n+2} - \binom{n+2}{\frac{n}{2} + 1}$ is positive by comparing the growth of the Riordan numbers to the growth of the binomial terms.

        The Riordan numbers are a strictly increasing sequence (after $\gamma_1$) that satisfy the recurrence relation $\gamma_{n+2} = \frac{n+1}{n+3} \left[  2\gamma_{n+1} + 3\gamma_n\right]$ \cite{CDY}. Thus, 
        \begin{align*}
            \frac{\gamma_{n+2}}{\gamma_n} &= \frac{\frac{n+1}{n+3} \left[  2\gamma_{n+1} + 3\gamma_n\right]}{\gamma_n}  > \frac{\frac{n+1}{n+3} \left[  2\gamma_{n} + 3\gamma_n\right]}{\gamma_n} \ = \frac{5(n+1)}{n+3}  \geqslant 4
        \end{align*}
        Note that the last inequality depends on the fact that $n \geqslant 7$.

        Now, consider the growth of the binomial terms: 

        \begin{align*}
    \dfrac{\binom{n+2}{(n+2)/2}}{\binom{n}{n/2}} & = \dfrac{\dfrac{(n+2)!}{(\frac{n}{2}+1)!(\frac{n}{2}+1)!}}{\dfrac{n!}{(n/2)!(n/2)!}} = \dfrac{(n+2)(n+1)}{(\frac{n}{2}+1)^2} = \dfrac{4n^2 +12n + 8}{n^2 +4n+4} < \dfrac{4n^2 +16n + 16}{n^2 +4n+4}  = 4 \\
\end{align*}

        So, finally,

        \[ \gamma_{n+2} - \binom{n+2}{\frac{n}{2} + 1} > 4 \gamma_n - \binom{n+2}{\frac{n}{2} + 1} > 4 \gamma_n - 4 \binom{n}{n/2} = 4 \left( \gamma_n - \binom{n}{n/2} \right) > 0.  \]

    When $n$ is odd, the difference we want to show is positive is:
    \begin{align*}
           \sum_{i=0}^{(n-1)/2} \binom{n}{i, n-2i, i} - &\sum_{i=0}^{(n-1)/2} \binom{n}{i+1, n-2i-1, i} \\
            &= \sum_{i \geqslant 0} \binom{n}{i, n-2i, i}  - \sum_{i\geqslant 0} \binom{n}{i+1, n-2i-1, i} \\
            &= \binom{n}{0}_2 - \binom{n}{1}_2  \\
            &= \gamma_n .
        \end{align*}
For all $n > 1$, $\gamma_n >0$, so the difference is positive for all odd $n \geqslant 3$.
\end{proof}

\begin{lemma} \label{lem:ineqSmall}
    For all $n \geqslant 7$ and $\dfrac{n-1}{3} \leqslant I \leqslant \dfrac{n-2}{2}$, 
    \[ \sum_{i=0}^I \binom{n}{i, n-2i-1, i+1} < \sum_{i=0}^I \binom{n}{i+1, n-2i-2, i+1}.  \]
\end{lemma}
\begin{proof}

Let $M = \lfloor \frac{n-1}{2} \rfloor$.  The only case where an $I$ satisfying the conditions of the lemma is equal to $M$ is $I = \frac{n-2}{2}$ for an even $M$. In that case, the identity is known to be true from the previous lemma so, to complete the proof, assume $I < M$. Then we may write:

\begin{align*} \sum_{i=0}^I \binom{n}{i, n-2i-1, i+1} &= \sum_{i=0}^{M} \binom{n}{i, n-2i-1, i+1} - \sum_{i=I+1}^{M-1} \binom{n}{i, n-2i-1, i+1} - \binom{n}{M, 2M-1, M+1} \\
&= \sum_{i=0}^{M} \binom{n}{i, n-2i-1, i+1} - \sum_{i=I+2}^{M} \binom{n}{i-1, n-2i+1, i} - \binom{n}{M, 2M-1, M+1}\end{align*}
Note the middle term is zero if $I = M-1$.
Since $2M +1 \leqslant n$, we have $\binom{n}{M, n-2M-1, M+1} \geqslant 1$ and thus,
    \begin{align*}
        \sum_{i=0}^I \binom{n}{i, n-2i-1, i+1}  
        &\leqslant \sum_{i=0}^{M} \binom{n}{i, n-2i-1, i+1} - \sum_{i=I+2}^{M} \binom{n}{i-1, n-2i+1, i} - 1.
    \end{align*}
Applying the previous lemma to the first sum on the right, 
    \begin{align*}
        \sum_{i=0}^I \binom{n}{i, n-2i-1, i+1} &< \sum_{i=0}^{M} \binom{n}{i, n-2i, i} - \sum_{i=I+2}^{M} \binom{n}{i-1, n-2i+1, i} - 1. \\
    \end{align*}
Using the fact that $\binom{n}{0, n, 0} = 1$, we then have
    \begin{align*}
        \sum_{i=0}^I \binom{n}{i, n-2i-1, i+1} &< 1 + \sum_{i=1}^{M} \binom{n}{i, n-2i, i} - \sum_{i=I+2}^{M} \binom{n}{i-1, n-2i+1, i} - 1 \\
        \\&=\sum_{i=1}^{M} \binom{n}{i, n-2i, i} - \sum_{i=I+2}^{M}  \binom{n}{i-1, n-2i+1, i}.
    \end{align*}
For all $i \geqslant I+2$, we have $3i-1 > n$, and thus $\dfrac{i}{n-2i+1} > 1$. Therefore, for any such $i$,
\[ \binom{n}{i-1, n-2i+1, i} = \dfrac{n}{(i-1)! (n-2i+1)!i!} = \dfrac{n!}{i! (n-2i)! i!} \cdot \dfrac{i}{n-2i+1} > \binom{n}{i, n-2i, i}.  \]

Returning to the main inequality, this implies:
\begin{align*}
        \sum_{i=0}^I \binom{n}{i, n-2i-1, i+1} & < \sum_{i=1}^{M} \binom{n}{i, n-2i, i} - \sum_{i=I+2}^{M}  \binom{n}{i-1, n-2i+1, i} \\
        &< \sum_{i=1}^{M} \binom{n}{i, n-2i, i} - \sum_{i=I+2}^{M}  \binom{n}{i, n-2i, i} \\
        &= \sum_{i=1}^{I+1} \binom{n}{i, n-2i, i} \\
        &= \sum_{i=0}^{I} \binom{n}{i+1, n-2i-2, i+1}.
    \end{align*}
\end{proof}

\renewcommand*{\proofname}{Proof of \cref{prop:rainbow_ops} (5)--(8).}
\begin{proof} \phantom{|}
    \begin{enumerate}
    \addtocounter{enumi}{4}
    \item \textbf{Reflection:} Any rainbow, $R = \{ \underline{x_i} \to \underline{y_i}\}_{i=0}^m$,  represents the same number of arrows as the rainbow, $\widetilde{R} = \{ \ul{n-y_i} \to \ul{n-x_i} \}_{i=0}^m$.  This is due to the symmetry of multinomial coefficients: 
    \[  
        |\widetilde{R}| = \sum_{i=0}^m\binom{n}{n-y_i,(n-x_i)-(n-y_i),n-(n-x_i)} = \sum_{i=0}^m\binom{n}{n-y_i,y_i-x_i,x_i}  = |R|. 
    \]
    \item \textbf{Translation for a block of arcs:}  We will only consider translations of blocks of arcs where each arc satisfies the required inequality for the translation of a single arc.  However, for each arc we require that the new endpoints are \emph{either} available in the original rainbow \emph{or} an endpoint of an arc in the same block to be translated. 

    Let $B = \{\ul{x_i} \to \ul{y_i}\}_{I \leqslant i \leqslant J}$ be a block that is full on the left. Then, for all $I \leqslant i \leqslant J$, $x_i = x_I + (i-I)$ and $y_i \leqslant y_I - (i-I)$. So, 
    \[  x_{i} + y_{i} \leqslant x_I + (i-I) + y_I - (i-I) = x_I +y_I. \]
    Thus, if $x_I + y_I < n$, then the entire block can be translated right, as long as $\ul{x_J + 1}$ and $\ul{y_I +1}$ are available in $R$.

    Likewise, if $B = \{\ul{x_i} \to \ul{y_i}\}_{I \leqslant i \leqslant J}$ is full on the right, $x_I + y_I > n$, and $\ul{x_I - 1}$ and $\ul{y_J -1}$ are available, then the entire block can be translated left.
    
    \item \textbf{Contracting a particular nested block of arcs:} Consider $B_{\mathcal{O}} = \{\ul{x_i} \to \ul{y_i}\}_{0 \leqslant i \leqslant I}$, the outer block of a rainbow, $R$, where $x_0 = 0$ and $\ul{x_I + 1}$ is available.  Then, contracting the block on the left, i.e., replacing $B_{\mathcal{O}}$ in $R$ by $B' = \{\ul{x_i+1} \to \ul{y_i}\}_{0 \leqslant i \leqslant I}$, is a non-decreasing operation when (a) $y_0 = n-1$, or (b) $y_0 = n$ and $I < \frac{n}{3}$.
        \begin{enumerate}
            \item  To prove this operation is non-decreasing when $y_0 = n-1$, we consider two subcases: $I < \frac{n-1}{3}$ and $I \geqslant \frac{n-1}{3}$.  Note that since $x_0 = 0$ and $y_0=n-1$, all arcs in $B_\mathcal{O}$ are of the form $\ul{i} \to \ul{n-i-1}$ for some $0 \leqslant i \leqslant I$.
            
            When $I<\frac{n-1}{3}$, each $\ul{i} \to \ul{n-i-1}$ satisfies the individual criterion to contract on the left since,
                    \begin{align*}
                        y-x &=( n-i-1) - i =  n - 2i-1 \\[-1mm]
                        &\geqslant n - 2I - 1 = n -\left( \frac{n-1}{3}\right)-1= \frac{n-1}{3}   \\[-1mm]
                        &> I\\[-1mm]
                        &\geqslant i = x.
                   \end{align*} 
                It remains to show that for any $\frac{n-1}{3} \leqslant I \leqslant \frac{n-2}{2}$,
                \[ |B_{\mathcal{O}}| = \sum_{i=0}^I \binom{n}{i, n-2i-1, i+1} \leqslant \sum_{i=0}^I \binom{n}{i+1, n-2i-2, i+1} = |B'|.  \]
                In the case where $n \geqslant 7$, this is exactly Lemma \ref{lem:ineqSmall}.

                The only possible integers $n$ and $I$ where $0 \leqslant n < 7$ and $\frac{n-1}{3} \leqslant I \leqslant \frac{n-2}{2}$ are the pair $n=4$ and $I=1$.  In this case, the inequality can be checked by hand. 
            \item Considering the case where $y_0 = n$ and $I < \frac{n}{3}$, we note each arc in $B_\mathcal{O}$ has the form $\ul{i} \to \ul{n-i}$ for some $0 \leqslant i \leqslant I$.
            Each such arc satisfies the individual criterion to contract on the left, since
            \[  y - x = (n-i) -i = n-2i \geqslant n-2I > n - 2\cdot \frac{n}{3} = \frac{n}{3} > I \geqslant i = x.   \]
        \end{enumerate}
        \vspace{-3mm}
    \item \textbf{Expanding a particular nested block of arcs:}  Consider a non-composable rainbow $R$, with outer block $B_{\mathcal{O}} = \{\ul{x_i} \to \ul{y_i}\}_{0 \leqslant i \leqslant I}$, such that $x_0 = 0$, $y_0 =n$, $\ul{x_I + 1}$ is available in $R$, and $I \geqslant \frac{n}{3}$.
    We claim that, in this case, expanding $B_r \setminus B_\ell$ on the left, i.e. replacing $B_r \setminus B_\ell = \{\ul{x_i} \to \ul{y_i}\}_{I < i \leqslant J}$ in $R$ by $B' = \{\ul{x_i-1} \to \ul{y_i}\}_{I < i \leqslant J}$, is a non-decreasing operation. 

    To see this, we first give expressions for $x_i$ and $y_i$ in terms of $n$ and $a$.  Since $B_r$ is unbroken on the right and $y_0 = n$, $y_i = n-i$ for all $i \leqslant J$. In this range ($I < i \leqslant J$) there is at least one unoccupied class between $\ul{x_i}$ and $\ul{0}$, so $x_i \geqslant i+1$.    Therefore,
    \begin{align*} y_i - x_i &= (n-i) - x_i \\
    & \leqslant (n-i) - (i+1) = n - 2i -1 \\
    &\leqslant n - \frac{2n}{3} - 1  = \frac{n}{3} -1 \\
    &< \frac{n}{3} \\
    & \leqslant I < i < x_i .
    \end{align*}
    Since $y_i - x_i < x_i$, the arc $\ul{x_i} \to \ul{y_i}$ can be replaced by $\ul{x_i -1} \to \ul{y_i}$ without decreasing the associated number of generators for all $I < i \leqslant J$. \qedhere
\end{enumerate}
\end{proof}

\renewcommand*{\proofname}{Proof.}

\vspace{-3mm}
We summarize the content of \cref{prop:rainbow_ops} (5)-(8) in \cref{table:blockTable}.

\begin{table}[h!]
    \centering
    \renewcommand{\arraystretch}{1.5}
    \resizebox{\linewidth}{!}{
\begin{tabular}{|c|c|c|p{5cm}|} \hline 
    Operation name & Replace & With & Conditions \\ \hline \hline
    Reflection & $R = \{\ul{x_i} \to \ul{y_i}\}_{i=0}^m$ & $ \{\ul{n-y_i} \to \ul{n-x_i}\}_{i=0}^m$ & \emph{None} \\ \hline 
    Block Translate Right & $B = \{\ul{x_i} \to \ul{y_i}\}_{I \leqslant i \leqslant J}$ & $\{\ul{x_i+1} \to \ul{y_i+1}\}_{I \leqslant i \leqslant J}$ & $B$ full on the left and $x_I + y_I < n$ \\ \hline 
    Block Translate Left & $B = \{\ul{x_i} \to \ul{y_i}\}_{I \leqslant i \leqslant J}$ & $\{\ul{x_i-1} \to \ul{y_i-1}\}_{I \leqslant i \leqslant J}$ & $B$ full on the right and $x_I + y_I > n$ \\ \hline 
    Block Contract Left & $B_{\mathcal{O}} = \{\ul{x_i}\to\ul{y_i}\}_{0 \leqslant i \leqslant I}$ & $\{\ul{x_i+1}\to\ul{y_i}\}_{0 \leqslant i \leqslant I}$ & Outer arc is $\ul{0} \to \ul{n-1}$ or outer arc is $\ul{0} \to \ul{n}$ and $I < \frac{n}{3}$ \\ \hline 
    Block Expand Left & $B_r\setminus B_\ell = \{ \ul{x_i} \to \ul{y_i} \}_{I+1 < i \leqslant J}$ & $\{ \ul{x_i-1} \to \ul{y_i} \}_{I+1 < i \leqslant J}$ & Outer arc is $\ul{0} \to \ul{n}$, $B_\mathcal{O} = B_\ell$, and $I \geqslant \frac{n}{3}$ \\ \hline 
\end{tabular}
}
    \caption{A summary of the operations on blocks of arcs and the conditions under which the operations are non-decreasing.}
    \label{table:blockTable}
\end{table}

Next, we present a flowchart in \cref{fig:alg} which describes an algorithm which takes any rainbow to a composable rainbow.  In the flowchart and the following results, proving the algorithm has the desired properties, we use the notation $x \in R$ to denote $\ul{x}$ is an endpoint of some arc in $R$.

\begin{figure}[h!]
\begin{tikzpicture}
    \draw[fill = gray!20, draw = none] (-7.5, -0.5)--(-7.5, -6.75)--(9, -6.76)--(9,-0.5);
    \node[draw, text width=3cm, text centered] at (0,0.5)(a) {\textbf{Input:} $R$, a rainbow with fewer than $\frac{n+1}{2}$ arcs.};
    \node[draw, text width=3cm, text centered, rounded corners] at (0, -1.5) (b) {(1) Check for composability near $B_\ell$ and $B_r$.};
    \node[draw, text width=3cm, text centered, double] at (7,-1.5) (c){$R$ composable.};
    \draw[->] (a)--(b);
    \draw[->] ([yshift=0.15cm]b.east)--([yshift=0.15cm]c.west) node[midway, label=above: {$0, n \notin R$}]{};
    \draw[->] ([yshift=-0.15cm]b.east)--([yshift=-0.15cm]c.west) node[midway, label=below:{\small $a_\ell, b_\ell \notin R$ or $a_r, b_r \notin R$}] {};
    \node[draw, text width = 3cm, text centered, rounded corners] at (0, -3) (d) {(2) Check for two spaces right.};
    \node[draw, text width = 3cm, text centered, double] at (7, -3) (e){Translate $B_\ell$ right. Result composable.};
    \draw[->] (b)--(d);
    \draw[->] (d)--(e) node[midway, label=below: {$n, n-1 \notin R$}]{};
    \node[draw, text width = 3cm, text centered, rounded corners] at (0, -4.5) (f) {(3) Reduce to case where $B_\mathcal{O} = B_\ell$.};
    \draw[->](d)--(f);
    \node[draw, text centered, double] at (-5, -4.5) (x) {Reflect.};
    \draw[->] (f)--(x) node[midway, label=above:{$B_\mathcal{O} \neq B_\ell$}]{};
    \draw[->] (x) --(-5,-3)--(d) {};
    \node[draw, text width=3cm, text centered, rounded corners] at (0, -6) (g) {(4) Check for space on the left.};
    \draw[->] (f)--(g) {};
    \node[draw, text centered, text width = 3cm, double] at (7, -6) (h) {Translate $B_r$ left. Result composable.};
    \draw[->] (g)--(h) node[midway, label=above:{$0, 1 \notin R$}]{};
    \node[draw, text centered, double, text width = 2cm] at (-5.5, -6) (y) {Translate $B_r$ left.};
    \draw[->] (y)--(-5.5,-5)--(-6, -5)--(-6,-1.5)--(b) {};
    \draw[->] (g)--(y) node[midway, label=above:{$0 \notin R, 1 \in R$}]{};
    \node[draw, text centered, rounded corners, text width = 3cm] at (0, -7.5) (i){(5) Check for inner translation left.};
    \node[draw, text centered, double, text width = 4cm] at (6.5, -7.5) (j) {Translate $B_r \setminus B_\ell$ left. Result composable.};
    \draw[->] (i)--(j) node[midway, label=above:{$a+2 \notin R$}]{};
    \draw[->] (g)--(i);
    \node[draw, text centered, rounded corners, text width = 3cm] at (0, -9) (k) {(6) Check widest arc length.};
    \draw[->] (i)--(k);
    \node[draw, text centered, text width = 4cm, double] at (6.5,-9) (l) {Contract $B_\mathcal{O}$ on the left. Result composable.};
    \draw[->] (k)--(l) node[midway, label=above:{$n \notin R$}]{};
    \node[draw, text centered, text width = 3cm, rounded corners] at (0, -10.5) (m) {(7) Check number of arcs in $B_\mathcal{O}$.};
    \node[draw, text centered, text width = 2cm, double] at (-5, -10.5) (n) {Contract $B_\mathcal{O}$ on the left.};
    \draw[->] (n)--(-5,-8)--(-7, -8)--(-7, -1.5)--(b);
    \draw[->] (k)--(m);
    \draw[->] (m)--(n) node[midway, label=above:{$< \frac{n+3}{3}$}]{};
    \node[draw, text centered, text width = 3cm, double] at (-4.5, -12) (o) {Expand $B_r \setminus B_\ell$ on the left.};
    \draw[->] (o)--(-7,-12)--(-7, -1.5)--(b);
    \draw[->] (m)--(0, -12)--(o) node[midway, label=below:{$\geqslant \frac{n+3}{3}$}]{};
\end{tikzpicture}
\caption{An algorithm which takes a rainbow to a composable rainbow.}\label{fig:alg}
\end{figure}
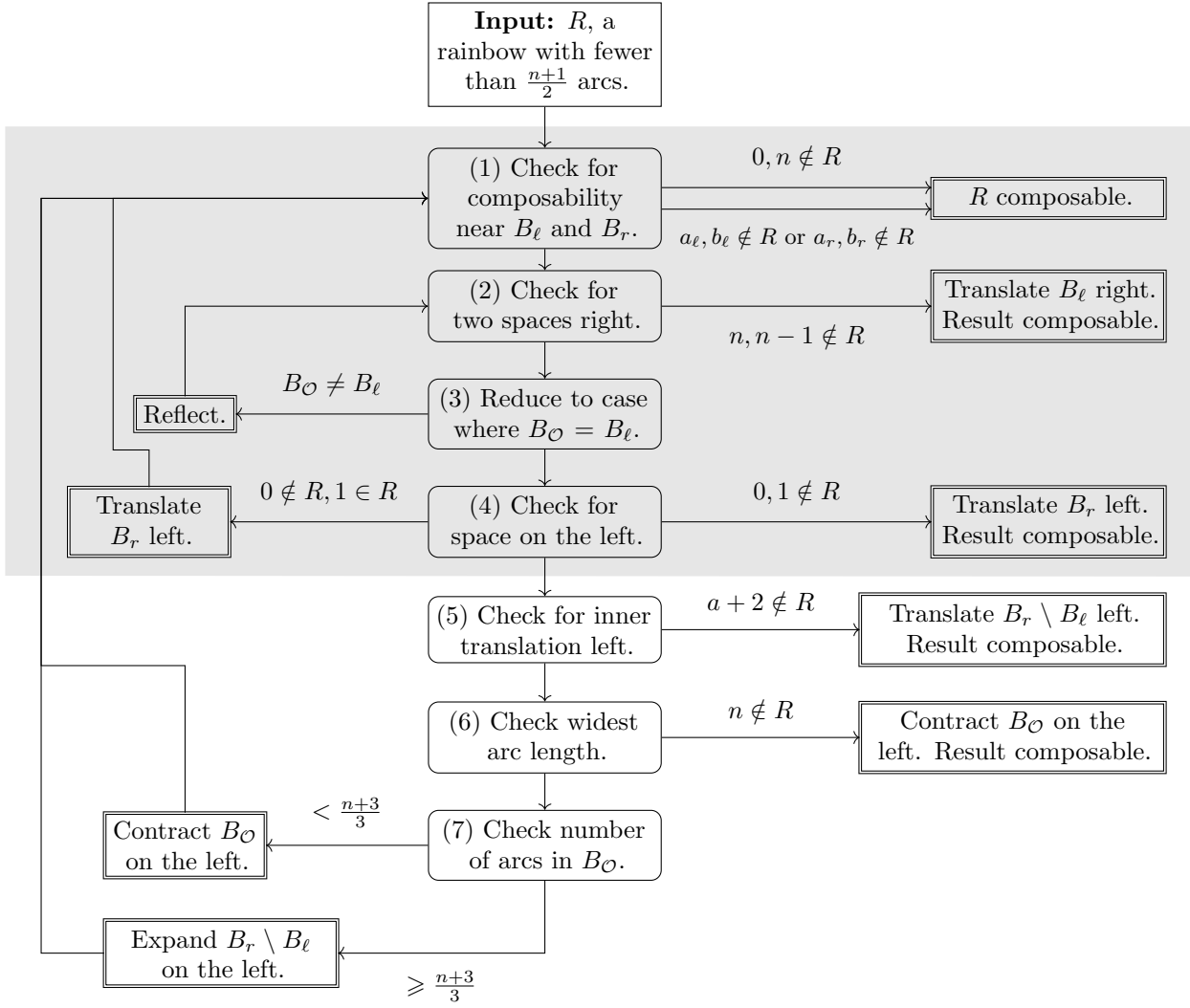

\begin{proposition}
    The output of the algorithm presented in \cref{fig:alg}, if it terminates, is a valid, composable rainbow corresponding to no fewer generators than the input rainbow.     
\end{proposition}

\begin{proof}
    We check that each step, given an valid rainbow as an input, produces a valid rainbow without reducing the number of corresponding generators.   Step (1) makes no changes to $R$.  If the algorithm does not terminate at this step, $R$ must contain 0 or $n$.
    
    Step (2) translates $B_\ell = \{\ul{x_0}+i \to \ul{y_i}\}_{0 \leqslant i \leqslant k}$ to the right, if $\ul{n}$ and $\ul{n-1}$ are available in $R$.  By the definition of $B_\ell$, $\ul{x_0+k+1}$ is also available, so there is room to translate $B_\ell$ to the left. In this case, $x_0 = 0$, since $n \notin R$.  Hence, translating $B_\ell$ to the right is a non-decreasing operation, as $B_\ell$ is full on the left and the outermost arc, $\ul{0} \to \ul{y_0}$ satisfies $x+y<n$.  

    Step (3) reflects the full rainbow if $B_\mathcal{O} \neq B_\ell$.  Reflection of a rainbow is always a valid non-decreasing operation.  Step (4) translates $B_r $ to the left, if $0 \notin R$. By an argument symmetric to the one in Step (2), this is a valid non-decreasing operation. 

    If a rainbow $R$ proceeds to Step (5), then $B_\mathcal{O}(R) = B_\ell (R)$  (so $a+1 \notin R$), $B_\ell \subsetneq B_r$ (else $a+1, b-1 \notin R$), and the outermost arc of $R$ is $\ul{0} \to \ul{n}$ or $\ul{0} \to \ul{n-1}$.   (In the flowchart, the box in grey applies operations to a rainbow until it satisfies these requirements.  So, for all steps outside the grey box, we assume $R$ has these properties.)
    
    Consider for such a rainbow, $B_r \setminus B_\ell = \{x_{k+1}, x_{k+2} \ldots, x_l, y_0 - l, \ldots, y_0-(k+2), y_0 - (k+1) \}$.  Both $\ul{x_{k+1} - 1}$ and $\ul{y_0-l-1}$ are available, so there is room to translate the block $B_r \setminus B_\ell$ to the left.  If $\ul{a+2} = \ul{x_k +2}$ is available, then $x_{k+1} \geqslant a+3 = x_k + 3$ so,
    \[ x_{k+1} + y_0 - (k+1) \geqslant x_k + 3 + y_0 - (k+1) = k + 3 + y_0 - (k+1) = y_0 + 2 \geqslant (n-1) + 2 > n. \]
    Since the sum of the outer endpoints of $B_r \setminus B_\ell$ is greater than $n$ and $B_r \setminus B_\ell$ is full on the right, translation of the block $B_r \setminus B_\ell$ to the left is non-decreasing. 
    
    In Step (6), $B_\mathcal{O}$ is contracted on the left if $n \notin R$.  Since we can now assume $B_\mathcal{O}=B_\ell$, we know $\ul{a+1}$ is available and there is thus room to contract $B_\mathcal{O}$ on the left.  When $n \notin R$, the outermost arc of $B_\mathcal{O}$ is $\ul{0} \to \ul{n-1}$, so $B_\mathcal{O}$ meets the conditions that guarantee contraction on the left is non-decreasing. 

    Any rainbow input into Step (7) has an outer arc $\ul{0} \to \ul{n}$ and has $\ul{a+1}$ available.   Consulting the table of block operations, we see that either the conditions are met to guarantee contracting $B_\mathcal{O}$ on the left is non-decreasing (true when the number of arcs is less than $\frac{n+3}{3}$) or the conditions are met to guarantee expanding $B_r \setminus B_\ell$ on the left is non-decreasing.  Since $\ul{a+1}$ is available, there must be space between $B_\mathcal{O}$ and $B_r \setminus B_\ell$ on the left, allowing for either one of these operations.
\end{proof}

\begin{proposition}
    The algorithm presented in \cref{fig:alg} terminates after performing at most four operations. 
\end{proposition}
\begin{proof}
    When a rainbow satisfying  $0, n \notin R$, $a_\ell+1, b_\ell-1 \notin R$, $a_r+1, b_r-1 \notin R$, or $n, n-1 \notin R$ is input, the algorithm terminates after at most one operation.  Henceforth, assume the input rainbow does not fall into one of those categories (and thus passes through Steps (1) and (2) unchanged).  

    For such an input rainbow, let $R_1$ be the rainbow output by Step (3) when $R_0$ is input.  If $B_\mathcal{O}(R_0) = B_\ell(R_0)$, we have $R_1 = R_0$. Otherwise, Step (3) reflects $R_0$, resulting in an $R_1$ where $B_\mathcal{O}(R_1) = B_\ell(R_1)$.  

    Let $R_2$ be the result of applying Step (4) to $R_1$.  Then, if $0 \notin R_1$,  $R_2$ is the result of translating $B_r(R_1)$ to the left.  Otherwise, $R_2 = R_1$.  Translating $B_r(R_1)$ to the left maintains the spacing between the arcs in $B_r(R_1)$, including the distance between the inner arc of $B_\mathcal{O} = B_\ell \subseteq B_r$ and $B_r \setminus B_\ell$. So, $B_\mathcal{O}(R_2) = B_\ell(R_2)$. 
    
    If $0 \in R_1$ then $R_2=R_1$ and we proceed to Step (5). If $0, 1 \notin R_1$ then $0, n \notin R_2$, so the algorithm terminates with the composable rainbow, $R_2$.    However, if $0 \notin R_1$ but $1 \in R_1$, then $R_2$ is re-input to Step (1).  For such an $R_1$, it must be that $0 \in R_2$.  Thus, if the algorithm doesn't terminate when an $R_2$ containing zero is checked for composability, $n$ or $n-1$ must be an endpoint of $R_2$ as well.  Since  $B_\mathcal{O}(R_2) = B_\ell(R_2)$, $R_2$ will not be reflected in Step (3). Then, since $0 \in R_2$, Step (4) likewise results in no operation.  So, exiting the grey box, we have a rainbow $R_2$ which is the result of at most two operations performed on $R_0$, satisfying $B_\mathcal{O}(R_2) = B_\ell(R_2)$ with outer arc $\ul{0} \to \ul{n}$ or $\ul{0} \to \ul{n-1}$.

    If the algorithm does not terminate at Step (5) or Step (6)---in either case, after a total of at most three operations---then the outer arc of $R_2$ is $\ul{0} \to \ul{n}$.  This means $R_2$ cannot have resulted from a translation on $B_r(R_1)$, so $R_2$ is in fact the result of applying at most one operation to $R_0$.   Let $R_3$ be the rainbow that results when the operation in Step (7) is applied to such an $R_2$. When $R_3$ is re-input to the algorithm, if it is found to be composable in Step (1) or can be translated to form a composable rainbow in Step (2), the algorithm terminates after at most four operations on $R_0$.  Otherwise, we consider two cases for an $R_3$ that proceeds to Step (3). 

    If $R_3$ is the result of contracting $B_\mathcal{O}(R_2)$ on the left, $R_3$ has outer arc $\ul{1}\to \ul{n}$.  If $B_\mathcal{O}(R_3)\neq B_r(R_3)$, a reflection will be performed in Step (3), resulting in $R_4$, which has outer arc $\ul{0} \to \ul{n-1}$ and satisfies $B_\mathcal{O}(R_4)=B_\ell(R_4)$.  On the other hand, if $B_\mathcal{O}(R_3)=B_\ell(R_3)$, no operation will be performed in Step (3) and a translation will be performed in Step (4). This results in a rainbow, $R_4$, which has outer arc $\ul{0} \to \ul{n-1}$ and satisfies $B_\mathcal{O}(R_4)=B_\ell(R_4)$.   At this point, the grey box is exited with a rainbow, $R_4$, which is the result of a total of three (or fewer) operations on $R_0$.  Since $n \notin R_4$, the algorithm terminates with one more operation in Step (5) or Step (6).  

    Addressing the remaining case, suppose $R_3$ is the result of expanding $B_r(R_2) \setminus B_\ell(R_2)$ on the left.  Let $\ul{a_r^2} \to \ul{b_r^2}$ be the innermost arc of $B_r(R_2)$ and $\ul{a_r^3} \to \ul{b_r^3}$ be the innermost arc of $B_r(R_3)$.  Then $a_r^3 = a_r^2 - 1$ and $b_r^3 = b_r ^2$.  The class $\ul{b_r^3 - 1} = \ul{b_r^2 - 1}$ is available in $R_3$ by the definition of $B_r$ and the class $\ul{a_r^3+1} = \ul{a_r^2}$ is available since the source of each arc in $B_r(R_2)$ has been moved left, leaving the space previously occupied by the rightmost source available.  So, once $R_3$ is re-input to the algorithm, it will be identified as composable in Step (1) and the algorithm will terminate after a total of at most three operations. 
\end{proof}

\begin{example}
    The rainbows in \cref{fig:ex1} are composable near $B_\mathcal{O}$, so the algorithm terminates immediately.  
    \vspace{-10mm}
    \begin{figure}[h!]
\centering
\begin{tikzpicture}[scale=0.7]
\foreach \i in {0,...,7}
{
    \node[fill=black!25,circle,draw,inner sep = 0pt, outer sep = 0pt, minimum size=1.5mm] (\i) at (\i/1.5,0) {};
    \node at (\i/1.5,-0.35) {\i};
}
\draw[-, line width=3pt, black] (1) edge[bend left=60] (6);
\draw[-, line width=3pt, black] (2) edge[bend left=60] (4);
\draw[-, ultra thick, rainbow-red] (1) edge[bend left=60] (6);
\draw[-, ultra thick, rainbow-orange] (2) edge[bend left=60] (4);
\end{tikzpicture}
\hspace{2cm}
\begin{tikzpicture}[scale=0.7]
\foreach \i in {0,...,8}
{
    \node[fill=black!25,circle,draw,inner sep = 0pt, outer sep = 0pt, minimum size=1.5mm] (\i) at (\i/1.5,0) {};
    \node at (\i/1.5,-0.35) {\i};
}
\draw[-, line width=3pt, black] (0) edge[bend left=60] (7);
\draw[-, line width=3pt, black] (1) edge[bend left=60] (6);
\draw[-, line width=3pt, black] (3) edge[bend left=60] (4);
\draw[-, ultra thick, rainbow-red] (0) edge[bend left=60] (7);
\draw[-, ultra thick, rainbow-red] (1) edge[bend left=60] (6);
\draw[-, ultra thick, rainbow-orange] (3) edge[bend left=60] (4);
\end{tikzpicture}
\caption{The arc $\ul{0} \to \ul{n}$ can be added to the left rainbow.  The arc $\ul{a+1} \to \ul{b-1}$ can be added to the right rainbow.}\label{fig:ex1}
\end{figure}
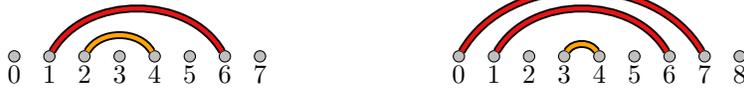

The next set of rainbows in \cref{fig:ex2} will be composable after a single translation.  The arcs to be translated are denoted with dashed lines. 

    \begin{figure}[h!]
\centering
\begin{tikzpicture}[scale=0.7]
\foreach \i in {0,...,9}
{
    \node[fill=black!25,circle,draw,inner sep = 0pt, outer sep = 0pt, minimum size=1.5mm] (\i) at (\i/1.5,0) {};
    \node at (\i/1.5,-0.35) {\i};
}
\draw[-, line width=3pt, black, dashed] (0) edge[bend left=60] (7);
\draw[-, line width=3pt, black, dashed] (1) edge[bend left=60] (5);
\draw[-, line width=3pt, black] (3) edge[bend left=60] (4);
\draw[-, ultra thick, rainbow-red, dashed] (0) edge[bend left=60] (7);
\draw[-, ultra thick, rainbow-red, dashed] (1) edge[bend left=60] (5);
\draw[-, ultra thick, rainbow-orange] (3) edge[bend left=60] (4);
\end{tikzpicture}
\hfill 
\begin{tikzpicture}[scale=0.7]
\foreach \i in {0,...,9}
{
    \node[fill=black!25,circle,draw,inner sep = 0pt, outer sep = 0pt, minimum size=1.5mm] (\i) at (\i/1.5,0) {};
    \node at (\i/1.5,-0.35) {\i};
}
\draw[-, line width=3pt, black, dashed] (2) edge[bend left=60] (9);
\draw[-, line width=3pt, black, dashed] (4) edge[bend left=60] (8);
\draw[-, line width=3pt, black] (5) edge[bend left=60] (6);
\draw[-, ultra thick, rainbow-red, dashed] (2) edge[bend left=60] (9);
\draw[-, ultra thick, rainbow-red, dashed] (4) edge[bend left=60] (8);
\draw[-, ultra thick, rainbow-orange] (5) edge[bend left=60] (6);
\end{tikzpicture}
\hfill 
\begin{tikzpicture}[scale=0.7]
\foreach \i in {0,...,9}
{
    \node[fill=black!25,circle,draw,inner sep = 0pt, outer sep = 0pt, minimum size=1.5mm] (\i) at (\i/1.5,0) {};
    \node at (\i/1.5,-0.35) {\i};
}
\draw[-, line width=3pt, black] (0) edge[bend left=60] (9);
\draw[-, line width=3pt, black, dashed] (3) edge[bend left=60] (8);
\draw[-, line width=3pt, black] (4) edge[bend left=60] (6);
\draw[-, ultra thick, rainbow-red] (0) edge[bend left=60] (9);
\draw[-, ultra thick, rainbow-orange, dashed] (3) edge[bend left=60] (8);
\draw[-, ultra thick, rainbow-yellow] (4) edge[bend left=60] (6);
\end{tikzpicture}
\caption{From left to right: translate $B_\ell$ right [Step (2)], translate $B_r$ left [Step (4)], translate $B_r\setminus B_\ell$ left [Step (5)].}\label{fig:ex2}
\end{figure}
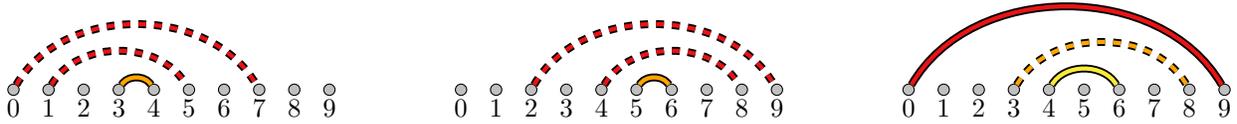

This set of rainbows in \cref{fig:ex3} demonstrate the potential effects of Steps (6) and (7). The arcs to be contracted or expanded are denoted with dashed lines.  Note that each of these rainbows satisfies the conditions of all rainbows that proceed past Step (4): $B_\mathcal{O} = B_\ell$ and the outermost arc is $\ul{0}\to \ul{n}$ or $\ul{0}\to \ul{n-1}$.

    \begin{figure}[h!]
\centering
\begin{tikzpicture}[scale=0.75]
\foreach \i in {0,...,9}
{
    \node[fill=black!25,circle,draw,inner sep = 0pt, outer sep = 0pt, minimum size=1.5mm] (\i) at (\i/1.5,0) {};
    \node at (\i/1.5,-0.35) {\i};
}
\draw[-, line width=3pt, black, dashed] (0) edge[bend left=60] (8);
\draw[-, line width=3pt, black, dashed] (1) edge[bend left=60] (7);
\draw[-, line width=3pt, black] (3) edge[bend left=60] (6);
\draw[-, line width=3pt, black] (4) edge[bend left=60] (5);
\draw[-, ultra thick, rainbow-red, dashed] (0) edge[bend left=60] (8);
\draw[-, ultra thick, rainbow-red, dashed] (1) edge[bend left=60] (7);
\draw[-, ultra thick, rainbow-orange] (3) edge[bend left=60] (6);
\draw[-, ultra thick, rainbow-orange] (4) edge[bend left=60] (5);
\end{tikzpicture}
\hfill 
\begin{tikzpicture}[scale=0.75]
\foreach \i in {0,...,9}
{
    \node[fill=black!25,circle,draw,inner sep = 0pt, outer sep = 0pt, minimum size=1.5mm] (\i) at (\i/1.5,0) {};
    \node at (\i/1.5,-0.35) {\i};
}
\draw[-, line width=3pt, black, dashed] (0) edge[bend left=60] (9);
\draw[-, line width=3pt, black, dashed] (1) edge[bend left=60] (8);
\draw[-, line width=3pt, black, dashed] (2) edge[bend left=60] (7);
\draw[-, line width=3pt, black] (5) edge[bend left=60] (6);
\draw[-, ultra thick, rainbow-red, dashed] (0) edge[bend left=60] (9);
\draw[-, ultra thick, rainbow-red, dashed] (1) edge[bend left=60] (8);
\draw[-, ultra thick, rainbow-red, dashed] (2) edge[bend left=60] (7);
\draw[-, ultra thick, rainbow-orange] (5) edge[bend left=60] (6);
\end{tikzpicture}
\vspace{-10mm}

\begin{tikzpicture}[scale=0.6]
\foreach \i in {0,...,21}
{
    \node[fill=black!25,circle,draw,inner sep = 0pt, outer sep = 0pt, minimum size=1.5mm] (\i) at (\i/1.5,0) {};
    \node at (\i/1.5,-0.35) {\i};
}
\draw[-, line width=3pt, black] (0) edge[bend left=60] (21);
\draw[-, line width=3pt, black] (1) edge[bend left=60] (20);
\draw[-, line width=3pt, black] (2) edge[bend left=60] (19);
\draw[-, line width=3pt, black] (3) edge[bend left=60] (18);
\draw[-, line width=3pt, black] (4) edge[bend left=60] (17);
\draw[-, line width=3pt, black] (5) edge[bend left=60] (16);
\draw[-, line width=3pt, black] (6) edge[bend left=60] (15);
\draw[-, line width=3pt, black] (7) edge[bend left=60] (14);
\draw[-, line width=3pt, black, dashed] (9) edge[bend left=60] (13);
\draw[-, line width=3pt, black] (10) edge[bend left=60] (11);
\draw[-, ultra thick, rainbow-red] (0) edge[bend left=60] (21);
\draw[-, ultra thick, rainbow-red] (1) edge[bend left=60] (20);
\draw[-, ultra thick, rainbow-red] (2) edge[bend left=60] (19);
\draw[-, ultra thick, rainbow-red] (3) edge[bend left=60] (18);
\draw[-, ultra thick, rainbow-red] (4) edge[bend left=60] (17);
\draw[-, ultra thick, rainbow-red] (5) edge[bend left=60] (16);
\draw[-, ultra thick, rainbow-red] (6) edge[bend left=60] (15);
\draw[-, ultra thick, rainbow-red] (7) edge[bend left=60] (14);
\draw[-, ultra thick, rainbow-orange, dashed] (9) edge[bend left=60] (13);
\draw[-, ultra thick, rainbow-yellow] (10) edge[bend left=60] (11);
\end{tikzpicture}
\caption{Left: contract $B_\mathcal{O}$ on the left when $n \notin R$ [Step (6)], \\
Right: contract $B_\mathcal{O}$ on the left when $n \in R$ and $B_\mathcal{O}$ has fewer than $\frac{n+3}{3}$ arcs [Step (7)], \\
Center: expand $B_r \setminus B_\ell$ on the left when $n \in R$ and $B_\mathcal{O}$ has at least $\frac{n+3}{3}$ arcs [Step (7)].}\label{fig:ex3}
\end{figure}
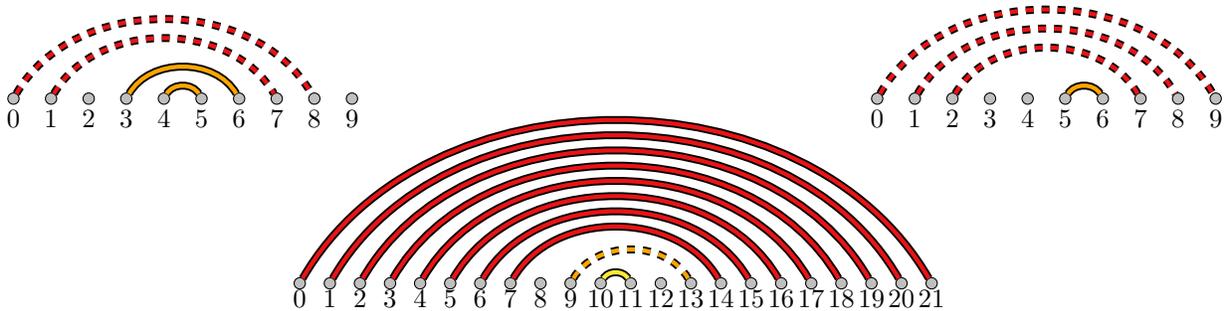

\end{example}

For any rainbow with fewer than $\lfloor \frac{n+1}{2} \rfloor$ arcs, the algorithm presented in \cref{fig:alg} gives a sequence of non-decreasing moves taking the given rainbow to a composable rainbow (a rainbow with room to add an additional arc). Composing a rainbow with an additional arc is strictly increasing on the number of generators, so this results in the following. 

\begin{corollary}
    \label{numArcs}
    The maximal sized rainbow for $C_{p_1 p_2 \cdots p_n}$ has $\lfloor \frac{n+1}{2} \rfloor$ arcs.
\end{corollary}

When $n$ is even, all rainbows with the maximal number of arcs have the form $R[x]$ for some $0 \leqslant x \leqslant n$ (\cref{def:maxRb}).
\begin{lemma}
    \label{evenOptions}
    When $n$ is even, the maximal rainbow is realized by $R[n/2]$ or both $R[0]$ and $R[n]$. 
\end{lemma}
\begin{proof}
    
When $n$ is even, a rainbow with $\lfloor \frac{n+1}{2} \rfloor = \frac{n}{2}$ arcs has exactly one of the $n+1$ classes not included in the rainbow.  This results in two blocks of arcs: a block of arcs $\underline{i} \to \underline{n-i}$ for $0 \leqslant i<x$ and a block of arcs $\underline{i} \to \underline{n-i+1}$ for $x < i \leqslant \frac{n}{2}$.  The size of the generating set corresponding to the rainbow excluding $\underline{x}$ is thus given by
\begin{align*} 
    |R[x]| & = \sum_{i=0}^{x-1}  \binom{n}{i, (n-i)-i, n-(n-i)} + \sum_{i=x+1}^{n/2} \binom{n}{i, (n-i+1)-i, n-(n-i+1)} \\
    & = \sum_{i=0}^{x-1} \binom{n}{i, n-2i, i} + \sum_{i=x+1}^{n/2} \binom{n}{i, n-2i+1, i-1}   \\
    &= \sum_{i=0}^{x-1} \binom{n}{i, n-2i, i} + \sum_{i=x}^{(n/2)-1} \binom{n}{i+1, n-2i-1, i}. 
\end{align*}

If the excluded term is $\underline{x}$ where $x \leqslant \frac{n-1}{3}$, then for all $i \leqslant x$, 
    \begin{align*}
        i & \leqslant \frac{n-1}{3} \\
        i+1 & \leqslant n - 2i \\
        \dfrac{n!}{i! (n-2i)! i!} &\leqslant \dfrac{n!}{(i+1)! (n-2i-1)!i!} \\
        \binom{n}{i, n-2i, i} &\leqslant \binom{n}{i+1, n-2i-1, i}.
    \end{align*}
So,
\[
    \lvert R[x] \rvert = \sum_{i=0}^{x-1} \binom{n}{i, n-2i, i} + \sum_{i=x}^{(n/2)-1} \binom{n}{i+1, n-2i-1, i}  \leqslant \sum_{i=0}^{(n/2)-1} \binom{n}{i+1, n-2i-1, i}. 
\]

\noindent The term on the right of the inequality above is the size of $R[0]$ (equivalently, the size of $R[n]$).

On the other hand, if the excluded term is $\underline{x}$ where $x > \frac{n-1}{3}$, we can likewise compute
\[ 
    \sum_{i=0}^{x-1} \binom{n}{i, n-2i, i} + \sum_{i=x}^{(n/2)-1} \binom{n}{i+1, n-2i-1, i}  < \sum_{i=0}^{(n/2)-1} \binom{n}{i, n-2i, i}. 
\]
\noindent The term on the right of the inequality above is the size of $R[n/2]$.

Thus every $\lvert R[m] \rvert $ is bounded above by $\lvert R[n/2] \rvert$ or $\lvert R[0] \rvert = \lvert R[n] \rvert$.
\end{proof}

We are now ready to prove the central theorem of this section. 

\begin{theorem}\label{thm:maincube}
 Let $G = C_{p_1 \cdots p_n}$ for $p_i$ distinct primes. Then
    \[
    \mathfrak{c}(G) \geqslant \begin{cases}
        \sum\limits_{i=0}^{\floor{\frac{n-1}{2}}} \binom{n}{n-i} \times \binom{n-i}{i+1} & n = 2, 4, 6 \\
        &\\
        \sum\limits_{i=0}^{\floor{\frac{n-1}{2}}} \binom{n}{n-i} \times \binom{n-i}{i} & n \neq 2, 4, 6. \\
        \end{cases}
    \]
This is realized by the complete rainbow when $n$ is odd, uniquely among rainbows by $R[n/2]$ when $n \geqslant 8$ is even, and by both $R[0]$ and $R[n]$ when $n=2, 4,$ or 6.
\end{theorem}

\begin{proof}
    We will show the described rainbows are maximal among rainbows with $\lfloor \frac{n+1}{2} \rfloor$ arcs.  Together with Corollary \ref{numArcs}, this completes the proof.

    \begin{enumerate}[label=(\alph*)]
        \item When $n$ is odd, $\lfloor \frac{n+1}{2} \rfloor = \frac{n+1}{2}$.  The complete rainbow is the only possible rainbow with $\frac{n+1}{2}$ arcs on the $n+1$ endpoints $\underline{0}, \underline{1}, \ldots, \underline{n}$.  
        
        The complete rainbow has arcs $\ul{i} \to \ul{n-i}$ for all $0 \leqslant i \leqslant \frac{n-1}{2}$ The size of this rainbow is thus, \[\sum\limits_{i=0}^{\frac{n-1}{2}} \binom{n}{i, (n-i)-i, n - (n-i)} =\sum\limits_{i=0}^{\frac{n-1}{2}} \binom{n}{i, n-2i, i} = \sum\limits_{i=0}^{\frac{n-1}{2}} \binom{n}{n-i} \times \binom{n-i}{i}.\]
    \end{enumerate}

        \begin{enumerate}[label=(\alph*), resume]
            \item For $n=2, 4,$ or $6$, we compute the possible maximal rainbows as given by Lemma \ref{evenOptions}: 

\begin{center}
\begin{tabular}{|c|c|c|c|c|c|c|c|} \cline{0-1}\cline{4-5}\cline{7-8}
            \multicolumn{2}{|c|}{$n=2$} &  & \multicolumn{2}{|c|}{$n=4$}  & & \multicolumn{2}{|c|}{$n=6$}  \\ \cline{0-1}\cline{4-5}\cline{7-8}
              $\underline{x}$ excluded & $|R[x]|$ & \hspace{1cm} &  $\underline{x}$ excluded & $|R[x]|$ & \hspace{1cm} &    $\underline{x}$ excluded & $|R[x]|$ \\ \cline{0-1}\cline{4-5}\cline{7-8}
                 \underline{0} & 2 &   & $\underline{0} $ & 16 & &   $\underline{0}$ & 126 \\
              \underline{1} & 1 & & \underline{2} & 13 & & $\underline{3}$ & 121 \\
               \underline{2} & 2 & & \underline{4} & 16 & & $\underline{6}$ & 126 \\ \cline{0-1}\cline{4-5}\cline{7-8}
\end{tabular}
\end{center}

In each case, the complexity of $C_{p_1 p_2 \cdots p_n}$ is realized by both $R[0]$ and $R[n]$.  It can be checked in each case that the given formula agrees with $|R[0]|=|R[n]|$, or the formula can be deduced from the fact that $R[n]$ has arcs $\ul{i} \to \ul{n-i-1}$ for all $0 \leqslant i \leqslant \frac{n}{2}-1$

        \item Finally, we consider the case where $n$ is even and greater than six. By Lemma \ref{lem:ineq},
        \[ |R[0]| = \sum_{i=0}^{ \frac{n}{2} -1 } \binom{n}{i+1, n-2i-1, i} < \sum_{i=0}^{ \frac{n}{2}-1} \binom{n}{i, n-2i, i} = \sum_{i=0}^{ \frac{n}{2}-1} \binom{n}{i} \times \binom{n-i}{i}  =  |R[n/2]|. \]
        \cref{evenOptions} implies $R[n/2]$ is the unique rainbow maximizing size. \qedhere
 \end{enumerate}
\end{proof}

\section{\texorpdfstring{Complexity of $C_{p^nq}$}{Complexity of C\_p\^nq}}\label{sec:grids}

\tikzset{
    every path/.style={-{>[sep = 1mm]}
    }
}

In this section, we consider the case of $G = C_{p^nq}$ where $p$ and $q$ are distinct primes and $n \in \mathbb{N}_{>0}$. The lattice $\Sub(G)$ is isomorphic to $[n]\times [1]$. Graphically we will draw the $[n]$ horizontally and the $[1]$ vertically, and use a coordinate system with $(i,j)$ referring to $C_{p^iq^j}\leqslant C_{p^nq}$. \cref{fig:Cp2qwidth} gives the subgroup lattice of $C_{p^2q}$.

\begin{figure}[h!]
    \centering
    \begin{tikzpicture}[scale=2]
        \node[below left] at (0,0) {$(0,0)$};
        \node[below] at (1,0) {$(1,0)$};
        \node[below right] at (2,0) {$(2,0)$};
        \node[above left] at (0,1) {$(0,1)$};
        \node[above] at (1,1) {$(1,1)$};
        \node[above right] at (2,1) {$(2,1)$};
        \grid{2}{1}
    \end{tikzpicture}
    \caption{$\Sub(C_{p^2q})$.}
    \label{fig:Cp2qwidth}
\end{figure}

Using \cref{prop:width} we can immediately conclude the following:

\begin{lemma}
    Let $G = C_{p^nq}$ where $p$ and $q$ are distinct primes. Then $ \w (G) = n+1$.
\end{lemma}

\begin{proof}
    The meet-irreducible subgroups of $C_{p^nq}$ are 
    \[
C_{p^n},\ C_{q},\ C_{pq},\ C_{p^2q},\dots, C_{p^{n-1}q}, 
\]
hence the result.
\end{proof}

The remainder of this section is devoted to the calculation of the complexity $\mathfrak{c}(C_{p^nq})$. We will denote the arrow $(C_{p^a q^x}, C_{p^b q^y}) \in \mathbb{I}(\Sub(C_{p^n q^m}))$, by $(a, x; b, y).$ That is, $(a,x; b,y)$ refers to the arrow from $(a,x)$ to $(b,y)$ in $[n]\times[1].$

We will consider three types of non-trivial arrows in $[n] \times [1]$:
\begin{itemize}
    \item \textit{Top Arrows}: arrows of the form $(i1,;j,1)$ (with $i<j$)
    \item \textit{Diagonal Arrows}: arrows of the form $(i,0;j,1)$ (with $i\leqslant j$)
    \item \textit{Bottom Arrows}: arrows of the form $(i,0;j,0)$ (with ($i < j$)
\end{itemize}

Note that in this definition, vertical arrows of the form $(i,0;i,1)$ are considered diagonal.

\begin{lemma}\label{lemma:farleft}
    Let $S$ be a minimal generating set for a transfer system $\mathsf{T} \in \mathsf{Tr}(C_{p^nq})$. Then at most one arrow in $S$ of each type (top, bottom, or diagonal) can start at the far left (i.e., can start at either $(0,0)$ or $(0,1)$.
\end{lemma}

\begin{proof}
    Given two arrows of the same type starting at the far left, whichever one has the right-most terminal endpoint will induce the other by restriction (for example, see \cref{fig:farleft}). Thus, since $S$ is minimal, it can only include one of these arrows.
    \begin{figure}[ht]
        \centering 
        \begin{tikzpicture}
            \draw (0,0) -- (5,1);
            \draw[dashed] (0,0) -- (3,1);
            \draw[bend right] (0,0) to (4,0);
            \draw[bend right, dashed] (0,0) to (2,0);
            \draw[bend left] (0,1) to (2,1);
            \draw[bend left, dashed] (0,1) to (1,1);
            \grid{6}{1}
        \end{tikzpicture}
        \caption{Each dashed arrow is generated by the solid arrow of the same type by restriction.}
        \label{fig:farleft}
    \end{figure}
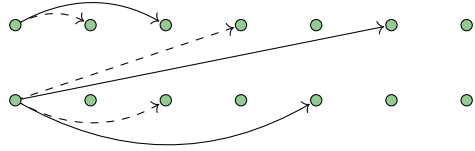
\end{proof}

\begin{lemma}\label{lemma:typepairs}
    Let $S$ be a minimal generating set for a transfer system $\mathsf{T} \in \mathsf{Tr}(C_{p^nq})$. Let $T,\ B$, and $D$ be the numbers of top, bottom, and diagonal arrows, respectively, in $S$. We have:
    \begin{enumerate}[leftmargin=*]
        \item $T+B\leqslant n$, and in the case of equality, $\langle S\rangle$ contains the arrow $(0,0;n,0)$.
        \item $T+D\leqslant n+1$, and in the case of equality, $\langle S\rangle$ contains the arrow $(0,0;n,1)$.
        \item $B+D\leqslant n+1$, and in the case of equality, $\langle S\rangle$ contains the arrow $(0,0;n,0)$.
    \end{enumerate}
\end{lemma}

\begin{proof}
    We use proof by induction for each part. The base case is immediate, as when $n=0$ there is a single possible arrow in $\Sub(C_{p^0q})$, which is vertical, hence diagonal.

    \begin{enumerate}[leftmargin=*]
        \item Let $\mathcal{U}$ be the set of top and bottom arrows in $S$, so that $\mathcal{U}$ is a minimal generating set for some transfer system on $C_{p^nq}.$ Assume the claim is true for $\mathsf{Tr}(C_{p^{n-1}q})$.
        
        Let $\mathcal{U}'$ be the set obtained from $\mathcal{U}$ by removing the arrows starting at the far left. The lattice obtained by removing the elements $(0, 0), (0, 1) \in [n] \times [1]$ is isomorphic to $[n-1] \times [1]$. Thus the set $\mathcal{U}'$ gives a minimal generating set for a transfer system on $C_{p^{n-1}q},$ so that $|\mathcal{U}'| \leqslant n-1$ by the induction hypothesis. 

        If $|\mathcal{U}'| \leqslant n-2,$ then since $\mathcal{U}$ can be obtained from $\mathcal{U}'$ by adding top and/or bottom arrows starting on the far left, \cref{lemma:farleft} gives $|\mathcal{U}|\leqslant |\mathcal{U}'|+2\leqslant n.$

        If $|\mathcal{U}'| = n-1,$ then by the induction hypothesis, $\langle\mathcal{U}'\rangle$ (and therefore $\langle\mathcal{U}\rangle$) contains the arrow $(1,0;n,0).$ Again by \cref{lemma:farleft}, $\mathcal{U}$ can contain at most one top and one bottom arrow, starting at the far left, in addition to the arrows in $\mathcal{U}'$. However, if $\mathcal{U}$ contains a top arrow from the left, by restriction $\langle\mathcal{U}\rangle$ must contain the arrow $(0,0;1,0)$ (implied by the top arrow), and each arrow of the form $(1,0;i,0)$ (implied by $(1,0;n,0)$). Then by transitivity, $\langle\mathcal{U}\rangle$ must contain $(0,0;i,0)$ for each $i$. See \cref{fig:farlefttopbottom}.

        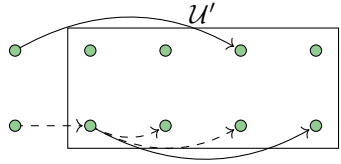
\begin{figure}[h]
            \centering
            \begin{tikzpicture}
                \gridbox{1}{0}{4}{1}
                \draw[bend right] (1,0) to (4,0);
                \draw[bend right,dashed] (1,0) to (3,0);
                \draw[bend right,dashed] (1,0) to (2,0);
                \draw[bend left] (0,1) to (3,1);
                \draw[dashed] (0,0) to (1,0);
                \grid{4}{1};
                \node at (2.5,1.5) {$\mathcal{U}'$};
        \end{tikzpicture}
        \caption{$\mathcal{U}$ cannot have both a top and bottom arrow from the far left if $|\mathcal{U}'|=n-1.$}
        \label{fig:farlefttopbottom}
        \end{figure}

        Thus if $\mathcal{U}$ contains a top leftmost arrow, this arrow, together with the arrows in $\mathcal{U}',$ generate all bottom leftmost arrows. Since $\mathcal{U}$ is a minimal generating set, it therefore cannot contain a bottom leftmost arrow. In other words, if $\mathcal{U}$ contains a top leftmost arrow, it cannot contain a bottom leftmost arrow, so $\mathcal{U}$ can contain at most one leftmost arrow in total, which gives $|\mathcal{U}|\leqslant |\mathcal{U}'|+1\leqslant n.$
        
        It remains to show that if $|\mathcal{U}|=n$, then $\langle S\rangle$ contains $(0,0;n,0).$ The argument above shows this is true if $|\mathcal{U}'|=n-1,$ but the following argument applies in either case.
        
        If no arrow in $\langle\mathcal{U}\rangle$ ends at the far right, then $\mathcal{U}$ is a minimal generating set for a transfer system on $C_{p^{n-1}q}$ containing $n$ top or bottom arrows, contradicting the induction hypothesis. So, let $i$ be minimal with the property that $\langle\mathcal{U}\rangle$ contains an arrow of the form $(i,0;n,0).$ If $i=0,$ then we are done, so assume $i>0.$
            
        First notice that no arrow in $\mathcal{U}$ can start before $i$ and end at or after $i$. Indeed if this were the case, in $\langle\mathcal{U}\rangle,$ there would be an arrow of the form $(j,0;k,0)$ with $j<i\leqslant k,$. This would imply $(j,0;i,0)$ by restriction, which in turn would imply $(j,0;n,0)$ by transitivity, contradicting the minimality of $i$. See \cref{fig:farlefttopbottom2}.

        \begin{figure}[h]
            \centering
            \begin{tikzpicture}
                \gridbox{0}{0}{1}{1}
                \gridbox{2}{0}{4}{1}
                \node at (0.5,1.75) {$\mathcal{U}_1$};
                \node at (3,1.75) {$\mathcal{U}_2$};
                \node[above] at (2,0) {$i$};
                \node[above] at (1,0) {$j$};
                \node[above] at (3,0) {$k$};
                \node[above] at (4,0) {$n$};
                \draw [bend right, dashed] (2,0) to (4,0);
                \draw [bend left] (1,1) to (3,1);
                \draw [bend right, dashed] (1,0) to (3,0);
                \draw [bend right, dashed] (1,0) to (2,0);
                \draw [bend right, dashed] (1,0) to (4,0);
                \grid{4}{1};
        \end{tikzpicture}
        \caption{An arrow in $\mathcal{U}$ starting before $i$ and ending at or after $i$ contradicts the minimality of $i$.}
        \label{fig:farlefttopbottom2}
        \end{figure}
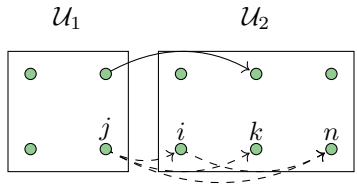

        So $\mathcal{U}$ gives minimal generating sets $\mathcal{U}_1$ and $\mathcal{U}_2$ on $C_{p^{i-1}q}$ and $C_{p^{n-i}q},$ respectively. We have
        \[
        n = |\mathcal{U}| = |\mathcal{U}_1| + |\mathcal{U}_2| \leqslant i-1+n-i=n-1.
        \]
        This is a contradiction, so we must have $i=0$; thus $\langle\mathcal{U}\rangle$ (and therefore $\langle S\rangle$) contains $(0,0;n,0)$, as claimed.
        
        \item Let $\mathcal{U}$ be the set of top and diagonal arrows in $S$. The argument above works identically to show that $|\mathcal{U}|\leqslant n+1$, with the $+1$ coming from the base case.
        
        There are some subtleties in showing $\langle S\rangle$ must contain $(0,0;n,1)$ if $|\mathcal{U}|=n+1.$ First note there must be an $i$ such that there is a top or diagonal arrow starting at $(i,0)$ or $(i,1)$ and ending at $(n,1)$, or else this would give a transfer system on $C_{p^{n-1}q}$ with $n+1$ top or diagonal arrows. 
        
        Let $i$ be minimal with this property, and assume $i>0.$ As before, this gives minimal generating sets $\mathcal{U}_1$ and $\mathcal{U}_2$ for transfer systems on $C_{p^{i-1}q}$ and $C_{p^{n-i}q},$ respectively. By the induction hypothesis we have $|\mathcal{U}_1|\leqslant i$ and $|\mathcal{U}_2|\leqslant n-i+1.$ We consider two cases:

        \textbf{Case 1}. If $\mathcal{U}_2$ doesn't contain any diagonal arrows, then $\mathcal{U}_2$ contains only top arrows, so $|\mathcal{U}_2|\leqslant n-i,$ so that 
        \[
        n+1=|\mathcal{U}_1|+|\mathcal{U}_2|\leqslant i+n-i=n,
        \]
        a contradiction.
        
        \textbf{Case 2}. If $\mathcal{U}_2$ does contain a diagonal, it forces the arrow $\alpha=(i-1,0;i-1,1).$ Therefore $\mathcal{U}_1$ cannot contain any arrow forced by $\mathcal{U}_1\cup\{\alpha\}$. Furthermore, $\alpha$ cannot be forced by anything in $\mathcal{U}_1.$ So, $\alpha$ can be added to $\mathcal{U}_1$ while maintaining minimality (of $\mathcal{U}_1)$, which means $|\mathcal{U}_1|\leqslant i-1$.
        \begin{figure}[h]
            \centering
            \begin{tikzpicture}
                \gridbox{0}{0}{1}{1}
                \gridbox{2}{0}{4}{1}
                \node at (0.5,1.75) {$\mathcal{U}_1$};
                \node at (3,1.75) {$\mathcal{U}_2$};
                \draw (3,0) to (4,1);
                \node[above] at (2,0) {$i$};
                \draw[dashed] (1,0) -- (1,1) node [midway, left] {$\alpha$};
                \grid{4}{1};
        \end{tikzpicture}
        \caption{The arrow $\alpha$ can be added to $\mathcal{U}_1$.}
        \end{figure}
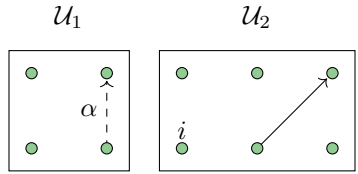

        Thus we have 
        \[
        n+1=|\mathcal{U}_1|+|\mathcal{U}_2|\leqslant i-1+n-i+1=n,
        \]
        again giving a contradiction.

        Thus we must have $i=0,$ so there must be an arrow of the form $(0,0;n,1)$ or $(0,1;n,1)$ in $\langle \mathcal{S}\rangle$. In the second case, notice that since $|\mathcal{U}|=n+1,$ $\mathcal{U}$ must contain a diagonal, but any diagonal forces the arrow $(0,0;0,1).$ This arrow together with $(0,1;n,1)$ forces $(0,0;n,1)$ by transitivity, completing the proof.

        \item For the set of bottom or diagonal arrows in $S$, the arguments used for the top or diagonal arrows work identically.\qedhere
    \end{enumerate}
\end{proof}

\begin{corollary}
\label{corollary-n1bound}
Let $S$ be a minimal generating set for a transfer system $\mathsf{T} \in \mathsf{Tr}(C_{p^nq})$. Then 
\[
|S|\leqslant \begin{cases}
    3k+1&\text{if } n=2k\\
    3k+2&\text{if } n=2k+1.
\end{cases}
\]
\end{corollary}

\begin{proof}
    Let $T,\ B,$ and $D$ be the number of top, bottom, and diagonal arrows in $S$.  By \cref{lemma:typepairs} we have
    \[
    2\cdot|S|=(T+B)+(T+D)+(B+D)\leqslant n+(n+1)+(n+1)=3n+2,
    \]
    which gives
    \[
    |S|\leqslant\left\lfloor\frac{3n}{2}\right\rfloor+1=\begin{cases}
            3k+1&\text{if } n=2k\\
    3k+2&\text{if } n=2k+1
    \end{cases}
    \]
    as required.
\end{proof}

Now that we have proved an upper bound to the size of a minimal generating set it remains to prove that this upper bound can be realized.

\begin{lemma}
\label{lemma-n1rainbows}
    For each $n \in \NN$ there exists a partial rainbow on $\Sub(C_{p^n q})$ of size $3k+1$ if $n=2k$ or size $3k+2$ if $n=2k+1$.
\end{lemma}
\begin{proof}
    The lattice $\Sub(C_{p^n q})$ is isomorphic to $[n]\times[1]$, so its image under the map $P \colon \Sub(G) \to \NN$ is $[n+1]$. If $n=2k$, consider the rainbow $R$ on $[n+1]$ given by the pairs $\{(0, n+1), \dots, (k, k +1)\}$. If $n=2k + 1$ instead, consider the rainbow $R'$ given by the pairs $\{(1, n+1), \dots, (k+1, k+2)\}$.

    The preimages $S$ and $S'$ in $\Sub(C_{p^n q})$ of $R$ and $R'$ are partial rainbows as $C_{p^n q}$ is abelian. Each of the pairs on the rainbows $R$ and $R'$ has three pairs in its preimage in $\Sub(C_{p^n q})$, except for $(0,n+1)$ and $(1, n+1)$ which have $1$ and $2$ respectively. Therefore $\lvert S \rvert = 3k+1$ and $\lvert S' \rvert = 3k+2$.
\end{proof}

\begin{corollary}\label{cor:Cpnq-complexity}
    The complexity of $C_{p^n q}$ is $3k+1$ if $n=2k$ or $3k+2$ if $n=2k+1$.
\end{corollary}
\begin{proof}
    There are partial rainbows of these sizes by \cref{lemma-n1rainbows}, which give minimal generating sets by \cref{proposition-rainbowsareminimal}. These are the upper bounds for sizes of minimal generating sets by \cref{corollary-n1bound}.
\end{proof}

\section{\texorpdfstring{Complexity of $C_{p^n q^m}$}{Complexity of C\_p\^nq\^m}}\label{sec:grids2}

In the previous section we computed the complexity of groups of the form $G=C_{p^nq}$, and in particular we proved that this complexity is always realized by a rainbow. In this section we will shift our attention to groups of the form $G = C_{p^nq^m}$. In this instance we immediately see new behavior:

\begin{example}\label{ex:notarainbow}

Let $G=C_{p^2q^2}$. Then one can see a maximally-sized partial rainbow in \cref{fig:maxrainbowp2q2}. However, it can be computed that the complexity of $G$ is 7, and this is realized uniquely by \cref{fig:doublerainbows}. That is, the complexity is not realized by a partial rainbow.

\begin{figure}[h!]
\centering
\begin{minipage}{.495\textwidth}
      \centering
    \[\begin{gathered}
\begin{tikzpicture}[scale= 0.5]
\node[fill=dark-green!50,circle,draw,inner sep = 0pt, outer sep = 0pt, minimum size=1.5mm] (1) at (0,0) {};
\node[fill=dark-green!50,circle,draw,inner sep = 0pt, outer sep = 0pt, minimum size=1.5mm] (p) at (2,0) {};
\node[fill=dark-green!50,circle,draw,inner sep = 0pt, outer sep = 0pt, minimum size=1.5mm] (p2) at (4,0) {};
\node[fill=dark-green!50,circle,draw,inner sep = 0pt, outer sep = 0pt, minimum size=1.5mm] (q) at (0,2) {};
\node[fill=dark-green!50,circle,draw,inner sep = 0pt, outer sep = 0pt, minimum size=1.5mm] (qp) at (2,2) {};
\node[fill=dark-green!50,circle,draw,inner sep = 0pt, outer sep = 0pt, minimum size=1.5mm] (qp2) at (4,2) {};
\node[fill=dark-green!50,circle,draw,inner sep = 0pt, outer sep = 0pt, minimum size=1.5mm] (q2) at (0,4) {};
\node[fill=dark-green!50,circle,draw,inner sep = 0pt, outer sep = 0pt, minimum size=1.5mm] (q2p) at (2,4) {};
\node[fill=dark-green!50,circle,draw,inner sep = 0pt, outer sep = 0pt, minimum size=1.5mm] (q2p2) at (4,4) {};
\draw[->, black!70] (1) edge[black!70, bend left=5] (q2p);
\draw[->, black!70] (1) edge[black!70, bend right=5] (qp2);
\draw[->, black!70] (q) edge[black!70] (q2);
\draw[->, black!70] (q) edge[black!70] (qp);
\draw[->, black!70] (p) edge[black!70] (p2);
\draw[->, black!70] (p) edge[black!70] (qp);
\end{tikzpicture}
\end{gathered}\]
\captionof{figure}{A maximally sized rainbow for $G=C_{p^2q^2}$.}
    \label{fig:maxrainbowp2q2}
\end{minipage}
\begin{minipage}{.495\textwidth}
 \centering
    \[\begin{gathered}
\begin{tikzpicture}[scale= 0.5]
\node[fill=dark-green!50,circle,draw,inner sep = 0pt, outer sep = 0pt, minimum size=1.5mm] (1) at (0,0) {};
\node[fill=dark-green!50,circle,draw,inner sep = 0pt, outer sep = 0pt, minimum size=1.5mm] (p) at (2,0) {};
\node[fill=dark-green!50,circle,draw,inner sep = 0pt, outer sep = 0pt, minimum size=1.5mm] (p2) at (4,0) {};
\node[fill=dark-green!50,circle,draw,inner sep = 0pt, outer sep = 0pt, minimum size=1.5mm] (q) at (0,2) {};
\node[fill=dark-green!50,circle,draw,inner sep = 0pt, outer sep = 0pt, minimum size=1.5mm] (qp) at (2,2) {};
\node[fill=dark-green!50,circle,draw,inner sep = 0pt, outer sep = 0pt, minimum size=1.5mm] (qp2) at (4,2) {};
\node[fill=dark-green!50,circle,draw,inner sep = 0pt, outer sep = 0pt, minimum size=1.5mm] (q2) at (0,4) {};
\node[fill=dark-green!50,circle,draw,inner sep = 0pt, outer sep = 0pt, minimum size=1.5mm] (q2p) at (2,4) {};
\node[fill=dark-green!50,circle,draw,inner sep = 0pt, outer sep = 0pt, minimum size=1.5mm] (q2p2) at (4,4) {};
\draw[->, black!70] (q) edge[black!70] (q2p);
\draw[->, black!70] (q) edge[black!70, bend left] (qp2);
\draw[->, black!70] (1) edge[black!70, bend left = 10] (q2p2);
\draw[->, black!70] (p) edge[black!70, bend right] (q2p);
\draw[->, black!70] (p) edge[black!70] (qp2);
\draw[->, black!70] (p2) edge[rainbow-red] (qp2);
\draw[->, black!70] (q2) edge[rainbow-red] (q2p);
\end{tikzpicture}
\end{gathered}\]
\captionof{figure}{A generating set that realizes the complexity for $G=C_{p^2q^2}$.}
    \label{fig:doublerainbows}
\end{minipage}
\end{figure}
\end{example}

Although \cref{ex:notarainbow} tells us that partial rainbows do not realize the complexity for some groups of the form $G = C_{p^nq^m}$, we will conjecture in this section that in all those exceptions the complexity is realized by a small modification of a specific kind of partial rainbows, which we will call  \emph{double rainbows}.

\begin{definition}
\label{definition-simpledouble}
    Given an interval $(C_{p^a q^x}, C_{p^b q^y}) \in \mathbb{I}(\Sub(C_{p^n q^m}))$, which we again denote $(a, x; b, y)$, its \emph{midpoint} is $(a+x+b+y)/2$. That is, the midpoint of the image of the interval under the map $P \colon \Sub(G) \to \NN$ of \cref{definition-rankmap}.
    
    Denote by $R_M \subset \mathbb{I}(\Sub(C_{p^n q^m}))$  the set of all intervals of $\Sub(C_{p^n q^m})$ with midpoint $M$. 
    \begin{itemize}
        \item If $n+m$ is odd we say that $R_{(n+m)/2}$ is a \emph{simple maximal partial rainbow} for $C_{p^n q^m}$. 
        \item If $n+m$ is even we say that
        \begin{itemize}
            \item $R_{(n+m-1)/2}$ (or $R_{(n+m+1)/2}$) is a \emph{simple maximal partial rainbow} for 
        $C_{p^n q^m}$
        \item $R_{(n+m)/2}$ is a \emph{double maximal partial rainbow} for $C_{p^n q^m}$.
        \end{itemize}
    \end{itemize}
    Note that these are indeed partial rainbows in the sense of \cref{definition-partialrainbow}, and they are maximal among partial rainbows for $C_{p^n q^m}$. We will generally just write \textit{simple rainbows} and \textit{double rainbows} for simple maximal partial rainbows and double maximal partial rainbows.
\end{definition}

\begin{example}
    In the case of $G=C_{p^2q^2}$ the simple maximal rainbow $R_{3/2}$ is exactly given in \cref{fig:maxrainbowp2q2} (the other simple maximal rainbow $R_{5/2}$ is the obvious symmetric reflection).

    The maximal double rainbow will consist of those intervals $(a, x; b, y)$ with $(a+x+b+y)/2 = 2$. These intervals are those displayed in \cref{fig:doublerainbows}, excluding the two red arrows. We are able to add two arrows to this double rainbow, yielding a minimal generating set, and this set realizes the complexity.
\end{example}

The phenomena about being able to add two independent arrows to the simple double rainbow is not particular to the specific case of $n=m=2$ as the next result shows.

\begin{lemma}\label{lem:addingtodouble}
    Let $n \geqslant m \geqslant 2$ and $n+m$ even. Then there exists a minimal generating set of size two more than the size of $R_{(n+m)/2}$.
\end{lemma}
\begin{proof}
    Let $R=R_{(n+m)/2}$. On the grid $[n] \times [m]$ we mark the two points $A=(n-\frac{n+m}{2}, m)$ and $B=(\frac{n+m}{2},0)$. Let $N$ consist of all edges of length one that start or end in $A$ or $B$. $N$ contains four edges if $n=m$ and six edges if $n>m$. If $n=m$ there are two edges of length two in $R$ that can be obtained by composing edges in $N$, while if $n>m$ there are four such composed edges. Let $S$ denote $R$ with these edges that can be obtained from $N$ removed. See \cref{fig:addingtwo} for an example.
    
    We want to prove that $S \cup N$ is a minimal generating set. This set is almost a partial rainbow; its projection under the map $P \colon \Sub(G) \to \NN$ consist of strictly nested arcs, except for the arcs $(\frac{n+m}{2}-1, \frac{n+m}{2})$, $(\frac{n+m}{2}, \frac{n+m}{2}+1)$, and $(\frac{n+m}{2}-1, \frac{n+m}{2}+1)$. Due to this, we will be able to use most, but not all, of the argument given in \cref{proposition-rainbowsareminimal} for partial rainbows.

    Assume that there is an edge $a=(K,H) \in S$ with $a \in \langle (S \cup N)\setminus a \rangle$. Then by \cref{lemma-generatingedge} there must exist an edge $b=(K',H')\in (S \cup N)\setminus a$ with $K \leqslant K', H \leqslant H'$. Such an edge $b$ cannot exist in $S$ because $S$ is a partial rainbow (and therefore minimal). If $b\in N,$ then the only possibilities for the edge $a$ are precisely those that we removed from $R$ to obtain $S$, contradicting $a\in S.$

    For the other case, assume that there is an edge $a=(K,H) \in N$ with $a \in \langle (S \cup N)\setminus a \rangle$. Since $a$ has length one, and the group is abelian, $a$ must be a restriction of an edge in $(S \cup N)\setminus a$. But we can see that this is impossible due to the way that we constructed $N$ and $S$.
\end{proof}
    \begin{figure}[h]
    \centering
    \[\begin{gathered}
\begin{tikzpicture}[scale= 0.85]
\node[fill=dark-green!50,circle,draw,inner sep = 0pt, outer sep = 0pt, minimum size=1.5mm] (1) at (0,0) {};
\node[fill=dark-green!50,circle,draw,inner sep = 0pt, outer sep = 0pt, minimum size=1.5mm] (p) at (2,0) {};
\node[fill=dark-green!50,circle,draw,inner sep = 0pt, outer sep = 0pt, minimum size=1.5mm] (p2) at (4,0) {};
\node[fill=dark-green!50,circle,draw,inner sep = 0pt, outer sep = 0pt, minimum size=1.5mm] (q) at (0,2) {};
\node[fill=dark-green!50,circle,draw,inner sep = 0pt, outer sep = 0pt, minimum size=1.5mm] (qp) at (2,2) {};
\node[fill=dark-green!50,circle,draw,inner sep = 0pt, outer sep = 0pt, minimum size=1.5mm] (qp2) at (4,2) {};
\node[fill=dark-green!50,circle,draw,inner sep = 0pt, outer sep = 0pt, minimum size=1.5mm] (q2) at (0,4) {};
\node[fill=dark-green!50,circle,draw,inner sep = 0pt, outer sep = 0pt, minimum size=1.5mm] (q2p) at (2,4) {};
\node[fill=dark-green!50,circle,draw,inner sep = 0pt, outer sep = 0pt, minimum size=1.5mm] (q2p2) at (4,4) {};
\node[fill=dark-green!50,circle,draw,inner sep = 0pt, outer sep = 0pt, minimum size=1.5mm] (p3) at (6,0) {};
\node[fill=dark-green!50,circle,draw,inner sep = 0pt, outer sep = 0pt, minimum size=1.5mm] (qp3) at (6,2) {};
\node[fill=dark-green!50,circle,draw,inner sep = 0pt, outer sep = 0pt, minimum size=1.5mm] (q2p3) at (6,4) {};
\node[fill=dark-green!50,circle,draw,inner sep = 0pt, outer sep = 0pt, minimum size=1.5mm] (p4) at (8,0) {};
\node[fill=dark-green!50,circle,draw,inner sep = 0pt, outer sep = 0pt, minimum size=1.5mm] (qp4) at (8,2) {};
\node[fill=dark-green!50,circle,draw,inner sep = 0pt, outer sep = 0pt, minimum size=1.5mm] (q2p4) at (8,4) {};
\draw[->, black!70] (q2) edge[bend left = 20,rainbow-red, dash dot] (q2p2);
\draw[->, black!70] (q) edge[bend left = 20,rainbow-blue] (qp4);
\draw[->, black!70] (q) edge[rainbow-blue] (q2p3);
\draw[->, black!70] (1) edge[rainbow-blue] (q2p4);
\draw[->, black!70] (p2) edge[bend left = 20,rainbow-red, dash dot] (p4);
\draw[->, black!70] (qp) edge[bend left = 20,rainbow-blue] (qp3);
\draw[->, black!70] (p) edge[rainbow-blue] (qp4);
\draw[->, black!70] (p2) edge[rainbow-blue] (q2p2);
\draw[->, black!70] (p) edge[rainbow-blue] (q2p3);
\draw[->, black!70] (qp) edge[rainbow-red, dash dot] (q2p2);
\draw[->, black!70] (p2) edge[rainbow-red, dash dot] (qp3);
\draw[->, black!70] (p2) edge[rainbow-green, dashed] (p3);
\draw[->, black!70] (p3) edge[rainbow-green, dashed] (p4);
\draw[->, black!70] (p3) edge[rainbow-green, dashed] (qp3);
\draw[->, black!70] (q2) edge[rainbow-green, dashed] (q2p);
\draw[->, black!70] (q2p) edge[rainbow-green, dashed] (q2p2);
\draw[->, black!70] (qp) edge[rainbow-green, dashed] (q2p);
\end{tikzpicture}
\end{gathered}\]
\caption{The case $n=4$, $m=2$. In red (dash-dotted) we see the four edges removed from $R_{(n+m)/2}$, in blue (solid) we see the rest of $R_{(n+m)/2}$, which we called $S$, and in olive (dashed) we see the six edges of $N$.}
    \label{fig:addingtwo}
\end{figure}
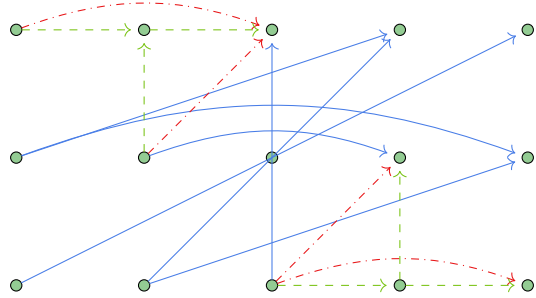

From computer calculations, along with the results presented in \cref{tab:valuesforsimple} we are led to conjecture the following:

\begin{conjecture}\label{conjecture:simpledouble}
    For $n \geqslant m \geqslant 2$, the complexity of $C_{p^n q^m}$ is realized by a simple rainbow, unless $n+m$ is even and $m=2$, in which case the complexity is realized by the set obtained from a double rainbow in \cref{lem:addingtodouble}.
\end{conjecture} 

The goal of the rest of this section is to compute the size of these simple and double rainbows. We will also see that the simple rainbows grow slightly larger as the small dimension $m$ increases, which is the reason why the unusual behavior in the above conjecture only appears when $m=2$.

\begin{definition}
    Let $n$ and $m$ be positive integers.
    \begin{itemize}
        \item  The \emph{simple rainbow number} $\SR(n, m)$ is the size of a simple maximal rainbow on $G = C_{p^n q^m}$.
        \item If $n, m \geqslant 2$ and $n+m$ is even, the \emph{double rainbow number} $\DR(n, m)$ is the size of the double maximal rainbow on $G = C_{p^n q^m}$.
    \end{itemize}

    Note that $\SR(n,m) = \SR(m,n)$ and $\DR(n,m) = \DR(m,n)$, as the subgroup lattices of $C_{p^n q^m}$ and $C_{p^m q^n}$ are isomorphic.
\end{definition}

\begin{table}
    \begin{minipage}{.45\linewidth}
      \centering
      \resizebox{\columnwidth}{!}{
        \begin{tabular}{|c|c|c|c|c|c|c|c|c|c|c|}
\hline
 \diagbox{$n$}{$m$}& 2 & 3 & 4 & 5 & 6 & 7 & 8 & 9 & 10 & 11 \\ \hline
2 & 6 & 10 & 12 & 16 &18 & 22 &24 & 28 & 30 & 34 \\ \hline
3 & 10 & 14& 20 &24 & 30 &34 & 40 &44 & 50 &54 \\ \hline
4 &12& 20 & 26 & 35 & 41 & 50 &56 & 65 & 71& 80 \\ \hline
5 & 16 & 24 & 35 &44 & 56 &65 & 77 & 86 & 98 & 107\\ \hline
6 &18 & 30 & 41 & 56 & 68& 84 & 96 & 112 &124 & 140 \\ \hline
7 & 22 & 34 & 50 & 65 & 84 & 100 & 120 & 136 & 156 & 172\\ \hline
8 &24 & 40 & 56& 77 & 96& 120 & 140 & 165 &185 & 210 \\ \hline
9 & 28 & 44& 65 &86& 112 & 136 & 165 & 190& 220 & 245 \\ \hline
10 & 30 & 50 & 71 & 98 & 124 & 156 & 185 & 220 & 250 & 286 \\ \hline
11 & 34 & 54 & 80 &107 & 140 & 172 & 210 & 245 & 286 & 322\\ \hline
\end{tabular}}
    \end{minipage}
    \begin{minipage}{0.1\linewidth}
        \centering
    \end{minipage}
    \begin{minipage}{.45\linewidth}
      \centering
        \resizebox{\columnwidth}{!}{
        \begin{tabular}{|c|c|c|c|c|c|c|c|c|c|c|}
\hline
\diagbox{$n$}{$m$} & 2 & 3 & 4 & 5 & 6 & 7 & 8 & 9 & 10 & 11 \\ \hline
2 & {5} &  \cellcolor{rainbow-blue!30}& {11} & \cellcolor{rainbow-blue!30}& {17} &  \cellcolor{rainbow-blue!30}& {23} &  \cellcolor{rainbow-blue!30}& {29} & \cellcolor{rainbow-blue!30}\\ \hline
3 &  \cellcolor{rainbow-blue!30}& {12} &  \cellcolor{rainbow-blue!30}& {22} &  \cellcolor{rainbow-blue!30}&{32} &  \cellcolor{rainbow-blue!30}& {42} &  \cellcolor{rainbow-blue!30}& {52} \\ \hline
4 & {11}& \cellcolor{rainbow-blue!30}& {24} & \cellcolor{rainbow-blue!30}& {39} & \cellcolor{rainbow-blue!30}& {54} & \cellcolor{rainbow-blue!30}& {69} & \cellcolor{rainbow-blue!30}\\ \hline
5 & \cellcolor{rainbow-blue!30}& {22} & \cellcolor{rainbow-blue!30}&{41} & \cellcolor{rainbow-blue!30}&{62} & \cellcolor{rainbow-blue!30}& {83} & \cellcolor{rainbow-blue!30}&{104}\\ \hline
6 &{17} & \cellcolor{rainbow-blue!30}& {39} & \cellcolor{rainbow-blue!30}& {65}& \cellcolor{rainbow-blue!30}& {93} & \cellcolor{rainbow-blue!30}& {121} & \cellcolor{rainbow-blue!30}\\ \hline
7 & \cellcolor{rainbow-blue!30}& {32} & \cellcolor{rainbow-blue!30}& {62} & \cellcolor{rainbow-blue!30}& {96} & \cellcolor{rainbow-blue!30}& {132} & \cellcolor{rainbow-blue!30}& {181} \\ \hline
8 & {23} & \cellcolor{rainbow-blue!30}& {54}& \cellcolor{rainbow-blue!30}& {93} & \cellcolor{rainbow-blue!30}& {136} & \cellcolor{rainbow-blue!30}&{181} & \cellcolor{rainbow-blue!30}\\ \hline
9 & \cellcolor{rainbow-blue!30}& {42}& \cellcolor{rainbow-blue!30}& {83}& \cellcolor{rainbow-blue!30}& {132} & \cellcolor{rainbow-blue!30}&{185}& \cellcolor{rainbow-blue!30}& {240} \\ \hline
10 & {29} & \cellcolor{rainbow-blue!30}&{69} & \cellcolor{rainbow-blue!30}& {121} & \cellcolor{rainbow-blue!30}& {181}& \cellcolor{rainbow-blue!30}& {245}& \cellcolor{rainbow-blue!30}\\ \hline
11 & \cellcolor{rainbow-blue!30}& {52} & \cellcolor{rainbow-blue!30}&{104} & \cellcolor{rainbow-blue!30}& {168} & \cellcolor{rainbow-blue!30}&{240} & \cellcolor{rainbow-blue!30}& {316}\\ \hline
\end{tabular}}
    \end{minipage} 
\caption{The table on the left represents the values of $\SR(n,m)$. The table on the right represents the values of $\DR(n,m)$ when $m+n$ is even}\label{tab:valuesforsimple}
\end{table}

We will begin by computing $\SR(n,m)$. By looking at \cref{tab:valuesforsimple} we see that on each row or column the numbers grow linearly, but with a different step according to whether $n+m$ is even or odd, just as in the case of $C_{p^nq}$ in \cref{sec:grids}. We will start by computing the simple rainbow numbers on the diagonal and subdiagonal, for which we require an intermediary lemma. By an \emph{arithmetic progression} of length $k$ we mean a sequence of integers $a, a+d, a+2d, \dots, a(k-1)d$ for $d$ some integer.

\begin{lemma}
\label{lemma-arithmetic}
    Let $n \geqslant 1$. Then the number of arithmetic progressions of length $3$ in the set $\{0,1, \cdots n\}$ is $\floor{\frac{n}{2}}\ceil{\frac{n}{2}}$.
\end{lemma}
\begin{proof}
    Let $d$ be the difference in the arithmetic progression. Then we see that there are $n+1-2d$ arithmetic progressions in our set with the given $d$. Next, note that the largest $d$ that we can have is $\floor{\frac{n}{2}}$ else we would land outside of our set.

    Therefore the total number of arithmetic progressions of length $3$ can be computed as
    \[\sum_{d=0}^{\floor{\frac{n}{2}}} (n+1-2d) = \floor{\frac{n}{2}} (n+1) - 2 \sum_{d=0}^{\floor{\frac{n}{2}}} d = \floor{\frac{n}{2}} (n+1) - \floor{\frac{n}{2}}(\floor{\frac{n}{2}}+1) = \floor{\frac{n}{2}}\ceil{\frac{n}{2}}.\qedhere\]
\end{proof}

\begin{lemma}\label{lemma:singlecount}
Let $n \geqslant 2$. Then
    \[
        \SR(n+1,n) = \binom{n+3}{3}.
    \]
    and
    \[
        \SR(n,n) = \binom{n+3}{3}-\floor*{\frac{n+2}{2}} \ceil*{\frac{n+2}{2}}.
    \]
\end{lemma}
\begin{proof}
    Let $n \geqslant 2$. Let $A(n) \subset \mathcal{P}(\{-1, 0, 1, \dots, n, n+1\})$ denote the set of subsets of $\{-1, 0, 1, \dots, n, n+1\}$ with three elements. We will denote such a three element set as $(c, d, e)$, always ordered from smallest to biggest.

    Let $B(n) = R_{(2n+1)/2} \subset \mathbb{I}(\Sub(C_{p^{n+1} q^n}))$ denote the simple rainbow for $C_{p^{n+1} q^n}$. Recall from \cref{definition-simpledouble} that this is the set of intervals in $\Sub(C_{p^{n+1} q^n})$, represented by tuples $(a, x; b, y) = (C_{p^a q^x}, C_{p^b q^y})$, that satisfy that $0 \leqslant a \leqslant b \leqslant n+1$, $0 \leqslant x \leqslant y \leqslant n$, and $(a+x+b+y)/2 = (2n+1)/2$.

    We will decompose each of $A(n)$ and $B(n)$ into three parts, and use this decomposition to prove both of the statements of the lemma.

    Let $A_0(n) \subset A(n)$ denote the subset of those three element sets that form an arithmetic progression. Let $A_l(n) \subset A(n)$ denote the subset of three element sets $(c, d, e)$ that satisfy that $d-c < e-d$, or in other words the middle element is closer to the lowest element than to the highest one. Let $A_h(n) \subset A(n)$ denote the subset of three element sets $(c, d, e)$ that satisfy that $d-c > e-d$, or in other words the middle element is closer to the highest element than to the lowest one.

    The set $A(n)$ is the disjoint union of $A_0(n)$, $A_l(n)$, and $A_h(n)$. Clearly $A_l(n)$ and $A_h(n)$ are symmetrical, and therefore $\lvert A_l(n) \rvert = \lvert A_h(n) \rvert$.

    As for the set $B(n)$, let $B_0(n) \subset B(n)$ denote the subset of those intervals $(a, x; b, y)$ with $a=0$. Let $B_l(n) \subset B(n)$ denote the subset of intervals $(a, x; b, y)$ with $a\geqslant 1$ and slope smaller than $1$, i.\ e.\ such that $b-a > y-x$ (these go \emph{low}). Lastly, Let $B_h(n) \subset B(n)$ denote the subset of intervals $(a, x; b, y)$ with $a\geqslant 1$ and slope bigger than $1$, i.\ e.\ such that $b-a < y-x$ (these go \emph{high}).

    The set $B(n)$ is the disjoint union of $B_0(n)$, $B_l(n)$, and $B_h(n)$. Note that the slope of an interval in $B(n)$ cannot be exactly $1$, as if $b-a = y-x$ then $b+x+y+a$ is even and cannot equal $2n+1$. The sets $B_l(n)$ and $B_h(n)$ are symmetrical as one is obtained from the other by reflecting the square sublattice of $\Sub(C_{p^{n+1} q^n})$ that had the column with the $C_{q^x}$ removed. Therefore $\lvert B_l(n) \rvert = \lvert B_h(n) \rvert$.

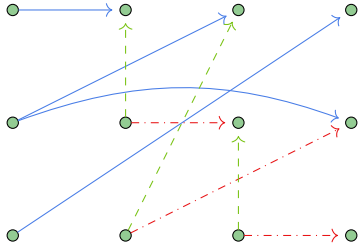
\begin{figure}[h]
    \centering
    \[\begin{gathered}
\begin{tikzpicture}[scale= 0.75]
\node[fill=dark-green!50,circle,draw,inner sep = 0pt, outer sep = 0pt, minimum size=1.5mm] (1) at (0,0) {};
\node[fill=dark-green!50,circle,draw,inner sep = 0pt, outer sep = 0pt, minimum size=1.5mm] (p) at (2,0) {};
\node[fill=dark-green!50,circle,draw,inner sep = 0pt, outer sep = 0pt, minimum size=1.5mm] (p2) at (4,0) {};
\node[fill=dark-green!50,circle,draw,inner sep = 0pt, outer sep = 0pt, minimum size=1.5mm] (q) at (0,2) {};
\node[fill=dark-green!50,circle,draw,inner sep = 0pt, outer sep = 0pt, minimum size=1.5mm] (qp) at (2,2) {};
\node[fill=dark-green!50,circle,draw,inner sep = 0pt, outer sep = 0pt, minimum size=1.5mm] (qp2) at (4,2) {};
\node[fill=dark-green!50,circle,draw,inner sep = 0pt, outer sep = 0pt, minimum size=1.5mm] (q2) at (0,4) {};
\node[fill=dark-green!50,circle,draw,inner sep = 0pt, outer sep = 0pt, minimum size=1.5mm] (q2p) at (2,4) {};
\node[fill=dark-green!50,circle,draw,inner sep = 0pt, outer sep = 0pt, minimum size=1.5mm] (q2p2) at (4,4) {};
\node[fill=dark-green!50,circle,draw,inner sep = 0pt, outer sep = 0pt, minimum size=1.5mm] (p3) at (6,0) {};
\node[fill=dark-green!50,circle,draw,inner sep = 0pt, outer sep = 0pt, minimum size=1.5mm] (qp3) at (6,2) {};
\node[fill=dark-green!50,circle,draw,inner sep = 0pt, outer sep = 0pt, minimum size=1.5mm] (q2p3) at (6,4) {};
\draw[->, black!70] (q2) edge[rainbow-blue] (q2p);
\draw[->, black!70] (q) edge[bend left = 20,rainbow-blue] (qp3);
\draw[->, black!70] (q) edge[rainbow-blue] (q2p2);
\draw[->, black!70] (1) edge[rainbow-blue] (q2p3);
\draw[->, black!70] (p2) edge[rainbow-red, dash dot] (p3);
\draw[->, black!70] (qp) edge[rainbow-red, dash dot] (qp2);
\draw[->, black!70] (p) edge[rainbow-red, dash dot] (qp3);
\draw[->, black!70] (p) edge[rainbow-green,dashed] (q2p2);
\draw[->, black!70] (qp) edge[rainbow-green, dashed] (q2p);
\draw[->, black!70] (p2) edge[rainbow-green, dashed] (qp2);
\end{tikzpicture}
\end{gathered}\]
\caption{The set $B(2)$, with the subsets $B_0(2), B_l(2)$ and $B_h(2)$ colored in blue (solid), red (dash-dot) and olive (dashed) respectively.}
    \label{fig:enter-label}
\end{figure}

The following formulas induce mutually inverse bijections between $A_0(n)$ and $B_0(n)$, and between $A_l(n)$ and $B_l(n)$, for each $n \geqslant 2$.
    \begin{equation*}
        \label{eq-bijectiongrid}
        \begin{aligned}
        (c, d, e) & \mapsto  (c+e-2d, d; n-c, n+1-e+d) \\
        (x-n-1+y+a, x, x+n+1-y) & \mapsfrom  (a, x; b, y)
    \end{aligned}
    \end{equation*}
    To begin proving this, note that $(c+e-2d)+d+(n-c)+(n+1-e+d)=2n+1$. If $(c, d, e) \in A(n)$ we know that $-1\leqslant c < d<e \leqslant n+1$, and this shows that $0\leqslant d \leqslant n+1-e+d \leqslant n$ and $c+e-2d\leqslant n-c \leqslant n+1$. 
    
    Note that $c+e-2d \geqslant 0$ need not hold, so this formula will not map $A(n)$ to $B(n)$, but if $(c, d, e) \in A_0(n)$ then by definition $c+e-2d = 0$ and therefore $(c+e-2d, d, n-c, n+1-e+d) \in B_0(n)$. If $(c, d, e) \in A_l(n)$ then by definition $c+e-2d \geqslant 1$, and $(n-c) - (c+e-2d) > (n+1-e+d)-d$ because $d>c$, so the slope of the resulting interval is smaller than $1$, and therefore $(c+e-2d, d, n-c, n+1-e+d) \in B_l(n)$.

    For the opposite map, for $(a, x; b, y)\in B(n)$ we know that $a+x+y = 2n+1 -b$, and $b\leqslant n + 1$, therefore $-1 \leqslant x-n-1+y+a$. Since $x \leqslant y \leqslant n$, we obtain that $x< x+n+1-y \leqslant n+1$. 
    
    As before, in general $x-n-1+y+a < x$ need not hold, which means that this formula does not map $B(n)$ to $A(n)$. However, if $(a, x; b, y)\in B_0(n)$ then $a=0$ and because $y \leqslant n$ we obtain that $x-n-1+y+a < x$. Additionally since $a=0$ we have $x-(x-n-1+y+a)=(x+n+1-y) -x$ which means that the three terms form an arithmetic progression, and $(x-n-1+y+a, x, x+n+1-y) \in A_0(n)$. If $(a, x; b, y)\in B_l(n)$ then $y+a < b+x$ and $y+a+b+x=2n+1$ which together yields that $y+a < n+1$, and therefore $x-n-1+y+a < x$. Lastly, because $a > 0$ we have that $x-(x-n-1+y+a)<(x+n+1-y) -x$, which altogether means that $(x-n-1+y+a, x, x+n+1-y) \in A_l(n)$.

    It can be checked directly that the formulas in \eqref{eq-bijectiongrid} are mutual inverses. We have constructed bijections between $A_0(n)$ and $B_0(n)$, and between $A_l(n)$ and $B_l(n)$. Since by symmetry $\lvert A_l(n) \rvert = \lvert A_h(n) \rvert$ and $\lvert B_l(n) \rvert = \lvert B_h(n) \rvert$, we obtain that $\SR(n+1,n) = \lvert B(n) \rvert = \lvert A(n) \rvert =\binom{n+3}{3}$.

    The sublattice of $\Sub(C_{p^{n+1} q^n})$ with the column of the $C_{q^x}$ removed is isomorphic to $\Sub(C_{p^n q^n})$, and the set $B_l(n) \coprod B_h(n)$ under that isomorphism is precisely a simple rainbow, therefore $\SR(n,n)=\lvert A_l(n) \coprod A_h(n) \rvert = \binom{n+3}{3} - \lvert A_0(n) \rvert=\binom{n+3}{3} - \floor*{\frac{n+2}{2}} \ceil*{\frac{n+2}{2}}$ by \cref{lemma-arithmetic}.
\end{proof}

We can also now obtain the double rainbow numbers on the diagonal.

\begin{lemma}
\label{lemma:doublecount}
    $\DR(n+1,n+1)=\SR(n+1,n) + \floor*{\frac{n+1}{2}} \ceil*{\frac{n+1}{2}}$.
\end{lemma}
\begin{proof}
    Consider the function $\mathbb{I}(\Sub(C_{p^{n+1} q^n})) \to \mathbb{I}(\Sub(C_{p^{n+1} q^{n+1}}))$ that sends a tuple $(a, x; b, y) = (C_{p^a q^x}, C_{p^b q^y})$ to $(a, x; b, y+1) = (C_{p^a q^x}, C_{p^b q^{y+1}})$. This function maps $R_{(2n+1)/2}$ to $R_{(2n+2)/2}$ injectively. The complement $C$ of its image in $R_{(2n+2)/2}$ contains precisely all tuples $(a, x; b, y) \in \DR(n+1,n+1)$ with $x=y$. Then sending $(a, x; b, y) \in C$ to the three element sequence $(a, (a+b)/2, b)$ gives a bijection between $C$ and the set of three-element arithmetic progressions in $\{0, 1, \dots, n, n+1\}$, which has size $\floor*{\frac{n+1}{2}} \ceil*{\frac{n+1}{2}}$ by \cref{lemma-arithmetic}.
\end{proof}

Finally, we can prove that the values of the simple and double rainbow numbers on each row grow linearly.

\begin{lemma}\label{lemma-linearstep}
    Let $n \geqslant m \geqslant 2$. Then
    \[
    \SR(n+2,m) = \SR(n,m) + \binom{m+2}{2}.
    \]
    Furthermore, if $n+m$ is even,
    \[
    DR(n+2,m) = DR(n,m) + \binom{m+2}{2}.
    \]
\end{lemma}
\begin{proof}
We will treat the simple and double rainbow case simultaneously. Sending $C_{p^a q^x}$ to $C_{p^{a+1} q^x}$ embeds the lattice $\Sub(C_{p^n q^m})$ into the lattice $\Sub(C_{p^{n+2} q^m})$, and this map sends $R_M$ into $R_{M+1}$. Therefore we obtain embeddings of simple and double rainbows, and the complements of these embeddings are given by the intervals $(a, x; b, y)$ that satisfy that either $a=0$ or $b=n+2$ and with midpoint $(a+x+b+y)/2=M+1$ for the midpoint $M$ in the definition of the starting rainbow. See Figure \ref{fig:RainbowEmbedding}.

\begin{figure}[h]
\centering
\tikzset{every path/.style={-{>[sep = 1.5mm]}}}
\begin{tikzpicture}
    \draw[red] (0,0) to (1,2);
    \draw[blue] (0,1) to (0,2);
    \draw[blue] (0,1) to (1,1);
    \draw[blue] (1,0) to (1,1);
    \grid{1}{2};
    \gridbox{0}{0}{1}{2};
\end{tikzpicture}
\hspace{2cm}
\begin{tikzpicture}
    \draw[red] (1,0) to (2,2);
    \draw[blue] (1,1) to (1,2);
    \draw[blue] (1,1) to (2,1);
    \draw[blue] (2,0) to (2,1);
    \draw[red,dashed] (0,1) to (2,2);
    \draw[red,dashed,bend left] (0,1) to (3,1);
    \draw[red,dashed] (1,0) to (3,1);
    \draw[blue, dashed] (0,2) to (1,2);
    \draw[blue,dashed] (2,0) to (3,0);
    \draw[olive,dashed] (0,0) to (3,2);
    \grid{3}{2};
    \gridbox{1}{0}{2}{2};
\end{tikzpicture}
\tikzset{every path/.style={}}
\caption{The simple maximal rainbow $R_{3/2}$ in $C_{pq^2}$ (left) and its embedding into $R_{5/2}$ in $C_{p^3q^2}$ (right). The complement of the image of $R_{3/2}$ in $R_{5/2}$ is shown with dashed arrows.}
\label{fig:RainbowEmbedding}
\end{figure}
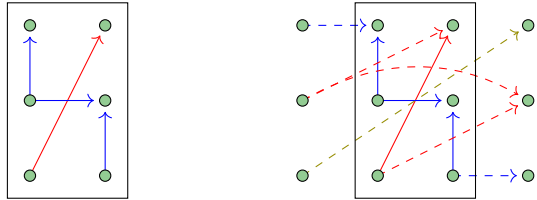

Let $T$ denote the set of pairs $(x,y)$ with $0 \leqslant x \leqslant y \leqslant m$. The cardinality of $T$ is $\binom{m+2}{2}$. Let $U$ be the set of pairs $(a,b)$ with $0 \leqslant a \leqslant b \leqslant n+2$ and such that either $a=0$ or $b=n+2$. The map that sends $(a,b)$ to $a+b$ gives a bijection between $U$ and the set $\{0, \dots, 2n+4\}$.

Recall that the $M$ in the definition of the rainbows is one of $(n+m-1)/2$, $(n+m)/2$, or $(n+m+1)/2$. This means that for each $(x, y) \in T$, we know that $2M + 2 -x-y$ lies in the set $\{0, \dots, 2n+4\}$. Therefore there exists a unique pair $(a,b) \in U$ with $x+y+a+b = 2M +2$, which means that $(a, x; b, y)$ lies in the complement of the image of the relevant rainbow for $(n,m)$ in the rainbow for $(n+2,m)$. Therefore this complement has the same cardinality as $T$, which is $\binom{m+2}{2}$.
\end{proof}

\begin{lemma}
    $\SR(n,m) - DR(n,m) = \ceil*{\frac{m}{2}} $ for $n \geqslant m \geqslant 2$ and $n+m$ even.
\end{lemma}
\begin{proof}
    Putting together \cref{lemma:singlecount} and \cref{lemma:doublecount} we obtain that for each $m\geqslant 3$,
    \begin{align*}
\SR(m,m) - DR(m,m) &=\SR(m,m)-\SR(m,m-1)-\floor*{\frac{m}{2}} \ceil*{\frac{m}{2}} \\
&= \binom{m+3}{3}-\floor*{\frac{m+2}{2}} \ceil*{\frac{m+2}{2}} - \binom{m+2}{3} - \floor*{\frac{m}{2}} \ceil*{\frac{m}{2}} \\
&= \binom{m+2}{2}-\floor*{\frac{m(m+2)}{2}+1} \\
&= \ceil*{\frac{m}{2}},
    \end{align*}
    
    while $\SR(2,2) - DR(2,2)=1= \ceil*{\frac{2}{2}}$ by inspection. Then \cref{lemma-linearstep} proves the statement for $n > m$.
\end{proof}

Our computations of the size of simple and double rainbows allow us to refine \cref{conjecture:simpledouble} to the following:

\begin{conjecture}\label{conj:enumeratedouble}
    Let $n \geqslant m \geqslant 2$. Then
    \[
        \mathfrak{c}(C_{p^nq^m}) = \begin{cases} 
        \DR(n,m)+2=1+6k & n=2 \text{ and } m=2k \text{ or } m=2 \text{ and } n=2k\\[3pt]
        \SR(n,m)= \binom{m+3}{3}-\floor*{\frac{m+2}{2}} \ceil*{\frac{m+2}{2}}+ \frac{n-m}{2} \binom{m+2}{2} & n,m \geqslant 3 \text{ and } n+m \text{ even}\\[3pt]
        \SR(n,m)=\binom{m+3}{3} + \frac{n-m-1}{2} \binom{m+2}{2} &  n+m \text{ odd}
        \end{cases}
    \]
\end{conjecture}

\begin{remark}\label{rem:lowbounds}
While \cref{conj:enumeratedouble} is conjectural, it does give a lower bound on the complexity. Thus, it only remains to prove that this coincides with the upper bound for a full understanding of the complexity.
\end{remark}

\newpage

\addtocontents{toc}{\vspace{1\baselineskip}}

\section{Further Observations and Directions}\label{sec:future}

The results in this paper came out of the development of a new algorithm for enumerating $N_\infty$ operads. However, it has also led to many new avenues of exploration in homotopical algebra. In this final section we discuss some of these avenues that became apparent to the authors during the development of this paper. We anticipate that many of these questions will lead to further intricate links between homotopical combinatorics, group theory, homotopy theory, and even statistics.

\subsection{Applications to incomplete $G$-commutative monoids}\label{subsec:nalg}

Fix a transfer system $\mathsf{T}$ and consider the category of incomplete Tambara functors. Defined in \cite{BH18}, these should be thought of as Tambara functors $R$ which only have norms 
\[
    N_H^K : R(G/H) \rightarrow R(G/K) 
\]
along subgroup inclusions $H \subset K$ which are in $\mathsf{T}$. These norms are assumed to satisfy various kinds of relations. Given two incomplete Tambara functors $R$ and $R'$ with the same underlying Green functor, one might guess that that the following is true. If, for all arrows $(H,K)$ in a generating set for $\mathsf{T}$, the norm $N_H^K$ for $R$ and $R'$ agree, then $R$ and $R'$ are isomorphic as incomplete Tambara functors.

In \cite{HH16} Hill and Hopkins define a $G$-commutative monoid object in a $G$-symmetric monoid category. When $G$ is a cyclic $p$-group, Hill and Mazur in \cite{HM19} show that the $G$-commutative monoids in the category of $G$-Mackey functors are precisely the $G$-Tambara functors; this result was extended to all finite groups $G$ in Havarneanu's thesis \cite{Hav18}. One might additionally expect the existence of a definition of ``incomplete $G$-commutative monoid'' objects in a $G$-symmetric monoidal category as a commutative monoid along with norm maps indexed by some transfer system. These norm maps would be expected to satisfy relations similar to those imposed in the usual $G$-commutative monoid case.

In this situation, one might hope that in the specification of the data of an incomplete $G$-commutative monoid, it would merely be enough to specify the norms corresponding to the elements of a basis of a transfer system. Rubin's algorithm could then hypothetically be used to define the remaining norms.

Specializing to the incomplete Tambara functor setting, suppose we are given a collection of initial norms $N_H^K$ just for $H \subset K$ in a generating set $B$ for $\mathsf{T}$. We can attempt to define the remaining norms by following Rubin's algorithm. For any $g \in G$ there is a norm $N_{gHg^{-1}}^{gKg^{-1}}$ which is related to $N_H^K$ by conjugation in $R$. If all the restrictions in $R$ are surjective (for example in the Burnside, representation, or constant Tambara functors), then a norm $N_K^H$ determines, for any $L \subset H$, a non-canonical candidate norm function
\[
    N_{L \cap K}^L : R(G/L\cap K) \rightarrow R(G/L) 
\]
obtained as follows. For each $x \in R(G/L\cap K)$, choose an element $y \in R(G/K)$ which restricts to $x$. Then the restriction of $N_K^H(y)$ defines an element of $R(G/L)$. It is unclear whether or not this choice can be made in a sufficiently nice way so as to end up even with a multiplicative function. In any case, at this stage, we define the norms along any remaining edge of $\mathsf{T}$ by composing the candidate norms; Rubin's algorithm then implies that we have defined candidate norms for all edges in $\mathsf{T}$. 

It would be very interesting and useful to know, especially for the complete transfer system, sufficient conditions on the generating set $B$ of $\mathsf{T}$ and the initial norms $N_H^K$ for $H \subset K$ in $B$ which imply that the norms obtained by this construction actually yield an incomplete Tambara functor. There is also a version of this construction for incomplete Mackey functors and incomplete Green functors, although it seems less of immediate interest to equivariant higher algebra.

In principal, this hypothetical picture could have applications to $N_\infty$-algebras in equivariant stable homotopy theory. Namely, let $\mathcal{O}$ be an $N_\infty$-algebra, which is determined up to equivalence by its associated transfer system. In \cite{Yang}, Yang shows that the category of \emph{normed} $C_p$-rings (heuristically, these may be thought of as $C_p$-commutative monoids in the homotopy category of genuine $C_p$-spectra) is actually equivalent to the category of $C_p$-$\mathbb{E}_\infty$-rings. Yang conjectures that such a theorem is true for any $G$. If this is true, and the conjectural picture of the previous two paragraphs is true, then the problem of showing a genuine $G$-spectrum $R$ is a $\mathcal{O}$-algebra reduces merely to showing that $R$ is an $\mathbb{E}_\infty$-algebra in $Sp^{BG}$, and that $R$ has a norm map for every arrow in a basis for $\mathsf{T}$.

\subsection{Complexity for elementary abelian $p$-groups}

We have studied the width and complexity for two large families of groups in \cref{sec:subes} and \cref{sec:grids}. The next class of groups that one might consider are the elementary abelian $p$-groups. The conjectured behavior of this is very similar as to what we saw in \cref{sec:subes} with a shift happening at $n=6$ as opposed to $n=8$. We record our preliminary observations and conjecture here. We start by recording the following lemma:

\begin{lemma}\label{lem:subgroup_counting}
    Let $0 \leqslant k \leqslant n$ and $p$ prime. Then the number of subgroups isomorphic to $C_p^k$ in $C_p^n$ is given by the following \emph{Gaussian binomial coefficient}:
    \[
    \binom{n}{k}_p = \dfrac{(p^n-1)(p^n-p)\cdots (p^n-p^{k-1}) }{(p^k-1)(p^k-p)\cdots (p^k-p^{k-1})}.
    \]
\end{lemma}

\begin{proof}
    We view $C_{p}^n$ as an $n$-dimensional vector space over $\mathbb{F}_p$. The $k$-dimensional subspaces of this vector space are in bijective correspondence with the isomorphism class of the subgroup $C_p^k$. This amounts to picking a linearly independent subset with $k$ elements. There are $(p^n-1)(p^n-p)\cdots (p^n-p^{k-1})$ ways of picking these subsets, but we need to only consider distinct ones. The number of ordered bases a $k$-dimensional vector space over $\mathbb{F}_p$ is given as $(p^k-1)(p^k-p)\cdots (p^k-p^{k-1})$. The result follows.
\end{proof}

We now once again appeal to finding maximal rainbows. Hence, we need to enumerate the arcs $\underline{x} \to \underline{y}$. In terms of the group, this is the number of subgroup inclusions from all of the $C_p^x$ to the $C_p^y$. 

\begin{lemma}
    Let $0 \leqslant x < y \leqslant n$ and $p$ prime. Then
    \[
    |\underline{x} \to \underline{y}| = \binom{n}{y}_p \binom{y}{x}_p.
    \]
\end{lemma}

\begin{proof}
    By \cref{lem:subgroup_counting} there are $\binom{n}{y}_p$ subgroups of $C_{p}^n$ of rank $y$. Then each of these has $\binom{y}{x}_p$ subgroups of rank $x$. Hence the result.
\end{proof}

Computational evidence leads us to the following conjecture:

\begin{conjecture}
    Let $G = (C_{p})^n$. Then
    \[
    \mathfrak{c}(G) = 
    \begin{cases}
        \sum\limits_{i=0}^{\frac{n-1}{2}}  \binom{n}{n-i}_p \times \binom{n-i}{i}_p & n \text{ odd,} \\
        &\\
        \sum\limits_{i=0}^{\frac{n}{2}-1}  \binom{n}{n-i}_p \times \binom{n-i}{i}_p & n \geqslant 6 \text{ even,} \\
         &\\
         \sum\limits_{i=0}^{\frac{n}{2} -1} \binom{n}{i}_p \times \binom{n-i}{i+1}_p & n=2,4.
    \end{cases}
    \]
\end{conjecture}

If we compare this to the result fo square-free abelian groups, we see startling similarity. The difference is that the change in behavior arises at $n=6$, and that the binomial coefficients have been replaced by Gaussian coefficients. The obstruction to proving this result comes from needing $q$-analogues of the results used in \cref{sec:subes}, in particular, it would be highly desirable to find a $q$-analogue of the recursion for the Riordan numbers.

\subsection{Distribution of generators}

Denote by $\mathsf{Tr}(G)_i$ the collection of transfer systems $\mathsf{T}$ such that $\mathfrak{m}(\mathsf{T}) = i$. If one plots $i$ against $|\mathsf{Tr}(G)_i|$ then an interesting phenomenon occurs. For example, \cref{fig:s5stats} gives the aforementioned plot for when $G = S_5$.

\begin{figure}[H]
    \centering
    \includegraphics[width=0.75\linewidth]{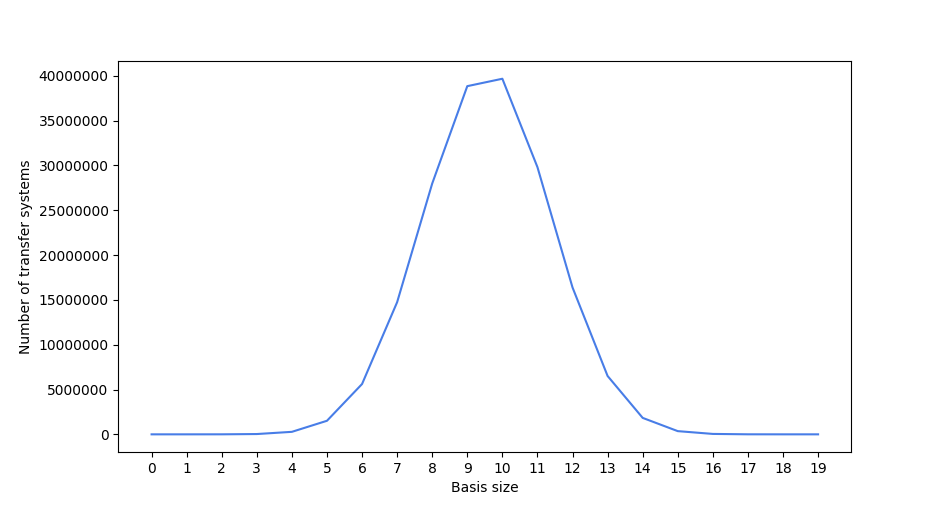}
    \caption{A plot showing the number of transfer systems for $S_5$ with a given minimal base size.}\label{fig:s5stats}
    \label{fig:s5_stats}
\end{figure}

From \cref{fig:s5_stats} we notice that we have something that looks like a binomial distribution. In fact, we obtain the same shape for all examples that we have computed. We anticipate that there should be a theorem which dictates that this is the case. 

Having a proof of this result would be beneficial to running calculations. Indeed, when we are running code to compute statistics for transfer systems, more often than not we do not even have a sensible estimate of  how long the code will take to run. With this result we would have a good guess of when the code is approximately half way as soon as the growth rate of $|\mathsf{Tr}(G)|_i$ decreases.

\subsection{Rainbow groups}

We have extensively used rainbows in our calculations for complexity. In some cases, such as in $C_{p_1 \cdots p_n}$ and for elementary abelian $p$-groups, we conjecture that the complexity is indeed realized by a rainbow. There are, however, many cases (even in the abelian setting) where rainbows do not realize the complexity, such as for $C_{p^2q^2}$. If a group does have the property that its complexity is realized by a rainbow, we shall say that $G$ is a \emph{rainbow group}. It would be interesting to classify those groups which are rainbow groups using inherent properties of either the group or the lattice of subgroups.

\subsection{Lossless generation}\label{subsec:lifting}

In \cite{BMO_lift} the notion of a \emph{lossless} group was introduced. These are groups $G$ such that transfer systems for $G$ can be detected by certain transfer systems on the poset $\Sub(G)/G$. This was used to great effect, for example, to give a recursive enumeration of transfer systems for dihedral groups $D_{p^n}$ for $p$ an odd prime in \cite{BMO_enumeration}. While \cite{BMO_lift}  proves that many groups are lossless, there are many important groups which are not lossless (that is, they are lossy). The first lossy group that was discussed was the group $S_4$, which Rubin notes has an issue with having to track which Klein 4 group lives in which dihedral group. 

The idea of generation should be compatible with the concept of being lossless. That is, the complexity of a lossless group $G$ is the same as the complexity that one would obtain from considering transfer systems on the poset $\Sub(G)/G$. Note, however, that the group $F_8$ that was discussed in \cref{rem:F8spicy} is lossy. As such, even for lossy groups, the width cannot be detected on $\Sub(G)/G$.

\subsection{Filtrations of the lattice of transfer systems}\label{sec:filtration}

Let $G$ be a finite group with collection of transfer systems $\mathsf{Tr}(G)$. While the algorithm developed in this paper give us a way to compute $|\mathsf{Tr}(G)|$, it does not yet aid in finding closed forms or recursions for this count. However, it does provide a potential route to simplifying the problem. Write $\mathsf{Tr}(G)_i$ for the collection of transfer systems $\mathsf{T}$ such that $\mathfrak{m}(\mathsf{T}) = i$. Then this stratifies the collection of transfer systems for $G$.

The key idea here is that enumerating $\mathsf{Tr}(G)_i$ is far more tangible than the enumeration of the entirety of $\mathsf{Tr}(G)$. When $i=1$ this simply counts the number of conjugacy classes of relations $(H, K)$ in $\mathrm{Sub}(G)$. Even the $i=2$ case would be of interest, say for abelian groups of square-free order.

Even in the case that $G=C_{p^n}$ this method has interesting combinatorics. Indeed, Franchere et. al \cite{fooqw}
 prove that $|\mathsf{Tr}(C_{p^n})_i|$ is exactly the $(n,i)$-th Narayana number, that is
 \[
|\mathsf{Tr}(C_{p^n})_i| = \frac{1}{n}\binom{n}{i}\binom{n}{i-1}.
 \]
 By summing these up as $i$ ranges from 0 to the complexity of $G$, one can then obtain a count for the entire collection of transfer systems.

\subsection{Complexity realizers}\label{sub:realizers}

This paper has been concerned mainly with computing the complexity of certain groups $G$. We have seen computationally, somewhat surprisingly, in some cases such as in $C_{p_1 \cdots p_n}$ or $S_5$ that there is a unique transfer system which realizes this complexity. Possibly more strange is the fact that for groups such as $S_4$ there are multiple transfer systems that realize the complexity. \cref{tab:complex} records the number of transfer systems realizing the complexity of $G = C_{p^iq^j}$. Already from this simple collection we see some wild behavior such as the off diagonal being all ones.

We believe that there may be a genuine property of the group $G$ which may underline when the complexity of $G$ is realized by a unique transfer system, building on from the ideas presented in \cref{sec:filtration}.

\begin{table}[h]
\begin{tabular}{|c|c|c|c|c|c|c|c|c|}
\hline
 & 0 & 1 & 2 & 3 & 4 & 5 & 6 & 7 \\ \hline
0 & \cellcolor{rainbow-blue!50} 1 & 1 & 1 & 1 & 1 & 1 & 1 & 1 \\ \hline
1 &  & \cellcolor{rainbow-blue!50}4 & 1 & 14 & 4 & 98 & 25 & 882 \\ \hline
2 &  &  & \cellcolor{rainbow-blue!50}1 & 1 & 1 & 4 & \cellcolor{rainbow-red!50} & \cellcolor{rainbow-red!50} \\ \hline
3 &  &  &  & \cellcolor{rainbow-blue!50}3 & 1 & \cellcolor{rainbow-red!50} & \cellcolor{rainbow-red!50} & \cellcolor{rainbow-red!50} \\ \hline
\end{tabular}
\caption{A table recording the number of transfer systems whose minimal generating basis realizes the complexity of the group for $G=C_{p^iq^j}$. The blue squares highlight the diagonal, while the red squares indicate cases where we have not been able to complete the computation due to computational constraints.}\label{tab:complex}
\end{table}

\subsection{Generators and duality, model structures, and compatibility}

The theory and properties of transfer systems has been widely researched and extended beyond just enumerating transfer systems:
\begin{itemize}
    \item Work of Franchere et al. \cite{fooqw}, building on work of Balchin et al. \cite{bbpr} introduces a duality  on the collection of transfer systems for an abelian group.
    \item Work of Franchere et al. \cite{fooqw} identified that transfer systems for abelian groups are equivalent to weak factorization systems on the lattice of subgroups. This was followed by work of Balchin et al. \cite{BOOR, BMO_comp} which considered those pairs of weak factorization systems which gave rise to (pre)model structures.
    \item Blumberg--Hill introduced the concept of bi-incomplete transfer systems in \cite{BH_bitambara}, which was abstracted to the notion of compatible pairs of transfer systems by Chan in \cite{Chan_tambara}. Compatible pairs for certain groups were considered by Hill et al. \cite{hill2022countingcompatibleindexingsystems} and Mazur et al. \cite{mazur2024uniquelycompatibletransfersystems}.
\end{itemize}

This of course leads to a discussion of how the idea of minimal generating sets interacts with these various concepts. For example, can one ask if it is possible to check if a pair of transfer systems is compatible by just checking it on given generating sets.

\addtocontents{toc}{\vspace{1\baselineskip}}

\appendix
\newpage 

\section{Algorithms}\label{appendix:a}

In this appendix we present pseudocode of a naive algorithm which produces all transfer systems by using \cref{cons:rubin} interpreted as a closure operator. As input for the final algorithm, we first need to implement Rubin's algorithm. This is presented in \cref{alg:three} below:

\begin{algorithm}[h]
\caption{The \texttt{transferClosure} function which implements Rubin's closure operator for a given set $S$ and group $G$. It is assumed that the $S_i$ are stored as types which do not admit duplicates (such as a \texttt{C++} set).
In practice this algorithm can be vastly improved via some dynamic programming. In particular for each pair $(K,H)$ one computes, and stores, the possible conjugates $(gKg^{-1},gKg^{-1})$ and intersections $(L \cap K, L)$ in an array, and for each pair $(K_1, H_1)$ and $(K_2, H_2)$ one stores the pair $(K_1, H_2)$ if it exists in a two-dimensional array. With this data preprocessed it is then possible to replace many of the computations and loops in this algorithm with simple lookups vastly improving the performance. Another benefit of this preprocessing is that the algorithm no longer needs to know about the group $G$, and all of the group theoretic operations are outsourced to a system such as \texttt{Sage}.}\label{alg:three}
\KwData{A collection $S = \{(H_i,K_i)\}_{i \in I}$ of subgroups $H < K \leqslant G$ and the finite group $G$.}
\KwResult{The transfer system $S_3 = \langle S \rangle$.}
~\\
\Begin(\texttt{transferClosure}${(S,\, G)}$:){
$S_0 \gets S $\;
\For{$(K,H) \in S_0$}{
    \For{$g \in G$}{
        $S_1\textbf{.insert} (gKg^{-1}, gHg^{-1})$\;
    }
}
~\\
$S_2 \gets S_1$\;
\For{$(K,H) \in S_1$}{
    \For{$L \leqslant H$}{
        $S_2\textbf{.insert}(L \cap K, L)$\;
    }
}
~\\
diff $\gets$ 1\;
$S_3 \gets S_2$\;
\While{$\mathrm{diff} \,{{!}{=}}\,\, 0$}{
    old\_size = $S_3\textbf{.size()}$\;
    \For{$(K_1,H_1) \in S_3$}{
            \For{$(K_2,H_2) \in S_3$}{
                \If{$H_1 == K_2$}{
                    $S_3\textbf{.insert}(K_1,H_2)$\;
                }
            }
    }
    diff = $S_3\textbf{.size()}-$old\_size\;
}
~\\
\Return{$S_3$\;}
}
~\\
\end{algorithm}

\newpage

We can now use \cref{alg:three} alongside one of the known algorithms for generating all closed sets of a closure operator. We present this in \cref{alg:four}. We note that this is only a naive implementation of this algorithm and there are many specific optimizations that one can use.

\begin{algorithm}\label{alg:rubin2}
    \caption{The \texttt{transferFind} function which computes, and stores, all transfer systems for a given finite group $G$. The \texttt{transferClosure} function is the algorithm of \cref{alg:three}. It is assumed that $\mathsf{Tr}(G)$, new\_transfer\_systems, and temp\_new\_transfer\_systems are stored as a type which do not admit duplicates (such as a \texttt{C++} set). We highlight that this algorithm is embarrassingly parallelizable by partitioning the new\_transfer\_systems array.}\label{alg:four}
\KwData{A finite group $G$.}
\KwResult{The collection of all transfer systems $\mathsf{Tr}(G)$ of $G$.}
~\\
\Begin(\texttt{transferFind}${(G)}$){
$\mathsf{Tr}(G), \text{new\_transfer\_systems} \gets \{\varnothing\}$\;
diff $\gets$ true\;
\While{$\mathrm{diff} \,{{!}{=}}\,\, 0$}{
    old\_size = $\mathsf{Tr}(G)\textbf{.size()}$\;
     temp\_new\_transfer\_systems\textbf{.clear()}\;
    \For{$\mathcal{R} \in \mathrm{new\_transfer\_systems}$}{
        \For{$(K,H) \in \Sub(G)$}{
            test\_transfer\_system $\gets \texttt{transferClosure}(\mathcal{R} \cup (K,H), G)$\;
            \If{$\mathrm{test\_transfer\_system} \not\in \mathsf{Tr}(G)$}{
            temp\_new\_transfer\_systems\textbf{.insert}(test\_transfer\_system)\;
            $\mathsf{Tr}(G)$\textbf{.insert}(test\_transfer\_system)\;
            }
        }
    }
    diff = $\mathsf{Tr}(G)\textbf{.size()}-$old\_size\;
    new\_transfer\_systems = temp\_new\_transfer\_systems\;
}
~\\
\Return{$\mathsf{Tr}(G)$\;}
}
~\\
\end{algorithm}

\newpage

\begin{example}
    For clarity, let us run \cref{alg:four} by hand for the case of $G = C_{pq}$ for $p \neq q$. While this algorithm is clearly over-kill for this example, hopefully the reader can appreciate how the algorithm would be effective in larger examples.
    
    The algorithm starts by inserting the only transfer system that has basis size zero, that is, the empty transfer system:
\[
\begin{tikzpicture}[scale=0.5, framed]
\node[fill=dark-green!50,circle,draw,inner sep = 0pt, outer sep = 0pt, minimum size=1.5mm] (1) at (0,0) {};
\node[fill=dark-green!50,circle,draw,inner sep = 0pt, outer sep = 0pt, minimum size=1.5mm] (p) at (2,0) {};
\node[fill=dark-green!50,circle,draw,inner sep = 0pt, outer sep = 0pt, minimum size=1.5mm] (q) at (0,2) {};
\node[fill=dark-green!50,circle,draw,inner sep = 0pt, outer sep = 0pt, minimum size=1.5mm] (qp) at (2,2) {};
\draw[->, black!10] (1) edge (p);
\draw[->, black!10] (1) edge (qp);
\draw[->, black!10] (1) edge (q);
\draw[->, black!10] (p) edge (qp);
\draw[->, black!10] (q) edge (qp);
\end{tikzpicture}
  \]      
We then take this transfer system and add all possible edges one at a time and apply \cref{alg:three} to it. In the following pictures, the solid blue edge is the one that we have chosen to add, and the dashed red edges are ones that are then added by \cref{alg:four}.
\begin{table}[h]
\begin{tabular}{ccccc}
\begin{tikzpicture}[scale=0.5, framed]
\node[fill=dark-green!50,circle,draw,inner sep = 0pt, outer sep = 0pt, minimum size=1.5mm] (1) at (0,0) {};
\node[fill=dark-green!50,circle,draw,inner sep = 0pt, outer sep = 0pt, minimum size=1.5mm] (p) at (2,0) {};
\node[fill=dark-green!50,circle,draw,inner sep = 0pt, outer sep = 0pt, minimum size=1.5mm] (q) at (0,2) {};
\node[fill=dark-green!50,circle,draw,inner sep = 0pt, outer sep = 0pt, minimum size=1.5mm] (qp) at (2,2) {};
\draw[->, rainbow-blue!100] (1) edge (p);
\draw[->, black!10] (1) edge (qp);
\draw[->, black!10] (1) edge (q);
\draw[->, black!10] (p) edge (qp);
\draw[->, black!10] (q) edge (qp);
\end{tikzpicture} &
\begin{tikzpicture}[scale=0.5, framed]
\node[fill=dark-green!50,circle,draw,inner sep = 0pt, outer sep = 0pt, minimum size=1.5mm] (1) at (0,0) {};
\node[fill=dark-green!50,circle,draw,inner sep = 0pt, outer sep = 0pt, minimum size=1.5mm] (p) at (2,0) {};
\node[fill=dark-green!50,circle,draw,inner sep = 0pt, outer sep = 0pt, minimum size=1.5mm] (q) at (0,2) {};
\node[fill=dark-green!50,circle,draw,inner sep = 0pt, outer sep = 0pt, minimum size=1.5mm] (qp) at (2,2) {};
\draw[->, rainbow-red!100, dashed] (1) edge (p);
\draw[->, rainbow-blue!100] (1) edge (qp);
\draw[->, rainbow-red!100, dashed] (1) edge (q);
\draw[->, black!10] (p) edge (qp);
\draw[->, black!10] (q) edge (qp);
\end{tikzpicture}
&
\begin{tikzpicture}[scale=0.5, framed]
\node[fill=dark-green!50,circle,draw,inner sep = 0pt, outer sep = 0pt, minimum size=1.5mm] (1) at (0,0) {};
\node[fill=dark-green!50,circle,draw,inner sep = 0pt, outer sep = 0pt, minimum size=1.5mm] (p) at (2,0) {};
\node[fill=dark-green!50,circle,draw,inner sep = 0pt, outer sep = 0pt, minimum size=1.5mm] (q) at (0,2) {};
\node[fill=dark-green!50,circle,draw,inner sep = 0pt, outer sep = 0pt, minimum size=1.5mm] (qp) at (2,2) {};
\draw[->, black!10] (1) edge (p);
\draw[->, black!10] (1) edge (qp);
\draw[->, rainbow-blue!100] (1) edge (q);
\draw[->, black!10] (p) edge (qp);
\draw[->, black!10] (q) edge (qp);
\end{tikzpicture}
&
\begin{tikzpicture}[scale=0.5, framed]
\node[fill=dark-green!50,circle,draw,inner sep = 0pt, outer sep = 0pt, minimum size=1.5mm] (1) at (0,0) {};
\node[fill=dark-green!50,circle,draw,inner sep = 0pt, outer sep = 0pt, minimum size=1.5mm] (p) at (2,0) {};
\node[fill=dark-green!50,circle,draw,inner sep = 0pt, outer sep = 0pt, minimum size=1.5mm] (q) at (0,2) {};
\node[fill=dark-green!50,circle,draw,inner sep = 0pt, outer sep = 0pt, minimum size=1.5mm] (qp) at (2,2) {};
\draw[->, black!10] (1) edge (p);
\draw[->, black!10] (1) edge (qp);
\draw[->, rainbow-red!100, dashed] (1) edge (q);
\draw[->, rainbow-blue!100] (p) edge (qp);
\draw[->, black!10] (q) edge (qp);
\end{tikzpicture}
& 
\begin{tikzpicture}[scale=0.5, framed]
\node[fill=dark-green!50,circle,draw,inner sep = 0pt, outer sep = 0pt, minimum size=1.5mm] (1) at (0,0) {};
\node[fill=dark-green!50,circle,draw,inner sep = 0pt, outer sep = 0pt, minimum size=1.5mm] (p) at (2,0) {};
\node[fill=dark-green!50,circle,draw,inner sep = 0pt, outer sep = 0pt, minimum size=1.5mm] (q) at (0,2) {};
\node[fill=dark-green!50,circle,draw,inner sep = 0pt, outer sep = 0pt, minimum size=1.5mm] (qp) at (2,2) {};
\draw[->, rainbow-red!100,dashed] (1) edge (p);
\draw[->, black!10] (1) edge (qp);
\draw[->, black!10] (1) edge (q);
\draw[->, black!10] (p) edge (qp);
\draw[->, rainbow-blue!100] (q) edge (qp);
\end{tikzpicture}
\end{tabular}
\end{table}

We then add all of these transfer systems to our collection $\mathsf{Tr}(G)$. As all of these are unique, and we have not yet seen them, we have now six transfer systems in $\mathsf{Tr}(G)$. We now take each one of these new transfer systems, and repeat what we did for the empty transfer system. Each row of the below corresponds to each of the five new transfer systems we obtained, using the same coloring scheme as before:
\begin{table}[h]
\begin{tabular}{ccccc}
 & \begin{tikzpicture}[scale=0.5, framed]
\node[fill=dark-green!50,circle,draw,inner sep = 0pt, outer sep = 0pt, minimum size=1.5mm] (1) at (0,0) {};
\node[fill=dark-green!50,circle,draw,inner sep = 0pt, outer sep = 0pt, minimum size=1.5mm] (p) at (2,0) {};
\node[fill=dark-green!50,circle,draw,inner sep = 0pt, outer sep = 0pt, minimum size=1.5mm] (q) at (0,2) {};
\node[fill=dark-green!50,circle,draw,inner sep = 0pt, outer sep = 0pt, minimum size=1.5mm] (qp) at (2,2) {};
\draw[->, black!100] (1) edge (p);
\draw[->, rainbow-blue!100] (1) edge (qp);
\draw[->, rainbow-red!100,dashed] (1) edge (q);
\draw[->, black!10] (p) edge (qp);
\draw[->, black!10] (q) edge (qp);
\end{tikzpicture}
&
\begin{tikzpicture}[scale=0.5, framed]
\node[fill=dark-green!50,circle,draw,inner sep = 0pt, outer sep = 0pt, minimum size=1.5mm] (1) at (0,0) {};
\node[fill=dark-green!50,circle,draw,inner sep = 0pt, outer sep = 0pt, minimum size=1.5mm] (p) at (2,0) {};
\node[fill=dark-green!50,circle,draw,inner sep = 0pt, outer sep = 0pt, minimum size=1.5mm] (q) at (0,2) {};
\node[fill=dark-green!50,circle,draw,inner sep = 0pt, outer sep = 0pt, minimum size=1.5mm] (qp) at (2,2) {};
\draw[->, black!100] (1) edge (p);
\draw[->, black!10] (1) edge (qp);
\draw[->, rainbow-blue!!10] (1) edge (q);
\draw[->, black!10] (p) edge (qp);
\draw[->, black!10] (q) edge (qp);
\end{tikzpicture}
&
\begin{tikzpicture}[scale=0.5, framed]
\node[fill=dark-green!50,circle,draw,inner sep = 0pt, outer sep = 0pt, minimum size=1.5mm] (1) at (0,0) {};
\node[fill=dark-green!50,circle,draw,inner sep = 0pt, outer sep = 0pt, minimum size=1.5mm] (p) at (2,0) {};
\node[fill=dark-green!50,circle,draw,inner sep = 0pt, outer sep = 0pt, minimum size=1.5mm] (q) at (0,2) {};
\node[fill=dark-green!50,circle,draw,inner sep = 0pt, outer sep = 0pt, minimum size=1.5mm] (qp) at (2,2) {};
\draw[->, black!100] (1) edge (p);
\draw[->, rainbow-red!100, dashed] (1) edge (qp);
\draw[->, rainbow-red!100, dashed] (1) edge (q);
\draw[->, rainbow-blue!100] (p) edge (qp);
\draw[->, black!10] (q) edge (qp);
\end{tikzpicture}
&
\begin{tikzpicture}[scale=0.5, framed]
\node[fill=dark-green!50,circle,draw,inner sep = 0pt, outer sep = 0pt, minimum size=1.5mm] (1) at (0,0) {};
\node[fill=dark-green!50,circle,draw,inner sep = 0pt, outer sep = 0pt, minimum size=1.5mm] (p) at (2,0) {};
\node[fill=dark-green!50,circle,draw,inner sep = 0pt, outer sep = 0pt, minimum size=1.5mm] (q) at (0,2) {};
\node[fill=dark-green!50,circle,draw,inner sep = 0pt, outer sep = 0pt, minimum size=1.5mm] (qp) at (2,2) {};
\draw[->, black!100] (1) edge (p);
\draw[->, black!10] (1) edge (qp);
\draw[->, black!10] (1) edge (q);
\draw[->, black!10] (p) edge (qp);
\draw[->, rainbow-blue!100] (q) edge (qp);
\end{tikzpicture}  
\\ 
 &  &  & 
 \begin{tikzpicture}[scale=0.5, framed]
\node[fill=dark-green!50,circle,draw,inner sep = 0pt, outer sep = 0pt, minimum size=1.5mm] (1) at (0,0) {};
\node[fill=dark-green!50,circle,draw,inner sep = 0pt, outer sep = 0pt, minimum size=1.5mm] (p) at (2,0) {};
\node[fill=dark-green!50,circle,draw,inner sep = 0pt, outer sep = 0pt, minimum size=1.5mm] (q) at (0,2) {};
\node[fill=dark-green!50,circle,draw,inner sep = 0pt, outer sep = 0pt, minimum size=1.5mm] (qp) at (2,2) {};
\draw[->, black!100] (1) edge (p);
\draw[->, black!100] (1) edge (qp);
\draw[->, black!100] (1) edge (q);
\draw[->, rainbow-blue!100] (p) edge (qp);
\draw[->, black!10] (q) edge (qp);
\end{tikzpicture}
&
\begin{tikzpicture}[scale=0.5, framed]
\node[fill=dark-green!50,circle,draw,inner sep = 0pt, outer sep = 0pt, minimum size=1.5mm] (1) at (0,0) {};
\node[fill=dark-green!50,circle,draw,inner sep = 0pt, outer sep = 0pt, minimum size=1.5mm] (p) at (2,0) {};
\node[fill=dark-green!50,circle,draw,inner sep = 0pt, outer sep = 0pt, minimum size=1.5mm] (q) at (0,2) {};
\node[fill=dark-green!50,circle,draw,inner sep = 0pt, outer sep = 0pt, minimum size=1.5mm] (qp) at (2,2) {};
\draw[->, black!100] (1) edge (p);
\draw[->, black!100] (1) edge (qp);
\draw[->, black!100] (1) edge (q);
\draw[->, black!10] (p) edge (qp);
\draw[->, rainbow-blue!100] (q) edge (qp);
\end{tikzpicture}
\\ 
 \begin{tikzpicture}[scale=0.5, framed]
\node[fill=dark-green!50,circle,draw,inner sep = 0pt, outer sep = 0pt, minimum size=1.5mm] (1) at (0,0) {};
\node[fill=dark-green!50,circle,draw,inner sep = 0pt, outer sep = 0pt, minimum size=1.5mm] (p) at (2,0) {};
\node[fill=dark-green!50,circle,draw,inner sep = 0pt, outer sep = 0pt, minimum size=1.5mm] (q) at (0,2) {};
\node[fill=dark-green!50,circle,draw,inner sep = 0pt, outer sep = 0pt, minimum size=1.5mm] (qp) at (2,2) {};
\draw[->, rainbow-blue!100] (1) edge (p);
\draw[->, black!10] (1) edge (qp);
\draw[->, black!100] (1) edge (q);
\draw[->, black!10] (p) edge (qp);
\draw[->, black!10] (q) edge (qp);
\end{tikzpicture}
&
\begin{tikzpicture}[scale=0.5, framed]
\node[fill=dark-green!50,circle,draw,inner sep = 0pt, outer sep = 0pt, minimum size=1.5mm] (1) at (0,0) {};
\node[fill=dark-green!50,circle,draw,inner sep = 0pt, outer sep = 0pt, minimum size=1.5mm] (p) at (2,0) {};
\node[fill=dark-green!50,circle,draw,inner sep = 0pt, outer sep = 0pt, minimum size=1.5mm] (q) at (0,2) {};
\node[fill=dark-green!50,circle,draw,inner sep = 0pt, outer sep = 0pt, minimum size=1.5mm] (qp) at (2,2) {};
\draw[->, rainbow-red!100, dashed] (1) edge (p);
\draw[->, rainbow-blue!100] (1) edge (qp);
\draw[->, black!100] (1) edge (q);
\draw[->, black!10] (p) edge (qp);
\draw[->, black!10] (q) edge (qp);
\end{tikzpicture}  &  & 
\begin{tikzpicture}[scale=0.5, framed]
\node[fill=dark-green!50,circle,draw,inner sep = 0pt, outer sep = 0pt, minimum size=1.5mm] (1) at (0,0) {};
\node[fill=dark-green!50,circle,draw,inner sep = 0pt, outer sep = 0pt, minimum size=1.5mm] (p) at (2,0) {};
\node[fill=dark-green!50,circle,draw,inner sep = 0pt, outer sep = 0pt, minimum size=1.5mm] (q) at (0,2) {};
\node[fill=dark-green!50,circle,draw,inner sep = 0pt, outer sep = 0pt, minimum size=1.5mm] (qp) at (2,2) {};
\draw[->, black!10] (1) edge (p);
\draw[->, black!10] (1) edge (qp);
\draw[->, black!100] (1) edge (q);
\draw[->, rainbow-blue!100] (p) edge (qp);
\draw[->, black!10] (q) edge (qp);
\end{tikzpicture}
&
\begin{tikzpicture}[scale=0.5, framed]
\node[fill=dark-green!50,circle,draw,inner sep = 0pt, outer sep = 0pt, minimum size=1.5mm] (1) at (0,0) {};
\node[fill=dark-green!50,circle,draw,inner sep = 0pt, outer sep = 0pt, minimum size=1.5mm] (p) at (2,0) {};
\node[fill=dark-green!50,circle,draw,inner sep = 0pt, outer sep = 0pt, minimum size=1.5mm] (q) at (0,2) {};
\node[fill=dark-green!50,circle,draw,inner sep = 0pt, outer sep = 0pt, minimum size=1.5mm] (qp) at (2,2) {};
\draw[->, rainbow-red!100, dashed] (1) edge (p);
\draw[->, rainbow-red!100, dashed] (1) edge (qp);
\draw[->, black!100] (1) edge (q);
\draw[->, black!10] (p) edge (qp);
\draw[->, rainbow-blue!100] (q) edge (qp);
\end{tikzpicture} \\ 
 \begin{tikzpicture}[scale=0.5, framed]
\node[fill=dark-green!50,circle,draw,inner sep = 0pt, outer sep = 0pt, minimum size=1.5mm] (1) at (0,0) {};
\node[fill=dark-green!50,circle,draw,inner sep = 0pt, outer sep = 0pt, minimum size=1.5mm] (p) at (2,0) {};
\node[fill=dark-green!50,circle,draw,inner sep = 0pt, outer sep = 0pt, minimum size=1.5mm] (q) at (0,2) {};
\node[fill=dark-green!50,circle,draw,inner sep = 0pt, outer sep = 0pt, minimum size=1.5mm] (qp) at (2,2) {};
\draw[->, rainbow-blue!100] (1) edge (p);
\draw[->, rainbow-red!100, dashed] (1) edge (qp);
\draw[->, black!100] (1) edge (q);
\draw[->, black!100] (p) edge (qp);
\draw[->, black!10] (q) edge (qp);
\end{tikzpicture}
&
\begin{tikzpicture}[scale=0.5, framed]
\node[fill=dark-green!50,circle,draw,inner sep = 0pt, outer sep = 0pt, minimum size=1.5mm] (1) at (0,0) {};
\node[fill=dark-green!50,circle,draw,inner sep = 0pt, outer sep = 0pt, minimum size=1.5mm] (p) at (2,0) {};
\node[fill=dark-green!50,circle,draw,inner sep = 0pt, outer sep = 0pt, minimum size=1.5mm] (q) at (0,2) {};
\node[fill=dark-green!50,circle,draw,inner sep = 0pt, outer sep = 0pt, minimum size=1.5mm] (qp) at (2,2) {};
\draw[->, rainbow-red!100, dashed] (1) edge (p);
\draw[->, rainbow-blue!100] (1) edge (qp);
\draw[->, black!100] (1) edge (q);
\draw[->, black!100] (p) edge (qp);
\draw[->, black!10] (q) edge (qp);
\end{tikzpicture}  &  &  & \begin{tikzpicture}[scale=0.5, framed]
\node[fill=dark-green!50,circle,draw,inner sep = 0pt, outer sep = 0pt, minimum size=1.5mm] (1) at (0,0) {};
\node[fill=dark-green!50,circle,draw,inner sep = 0pt, outer sep = 0pt, minimum size=1.5mm] (p) at (2,0) {};
\node[fill=dark-green!50,circle,draw,inner sep = 0pt, outer sep = 0pt, minimum size=1.5mm] (q) at (0,2) {};
\node[fill=dark-green!50,circle,draw,inner sep = 0pt, outer sep = 0pt, minimum size=1.5mm] (qp) at (2,2) {};
\draw[->, rainbow-red!100, dashed] (1) edge (p);
\draw[->, rainbow-red!100, dashed] (1) edge (qp);
\draw[->, black!100] (1) edge (q);
\draw[->, black!100] (p) edge (qp);
\draw[->, rainbow-blue!100] (q) edge (qp);
\end{tikzpicture} \\
 &  \begin{tikzpicture}[scale=0.5, framed]
\node[fill=dark-green!50,circle,draw,inner sep = 0pt, outer sep = 0pt, minimum size=1.5mm] (1) at (0,0) {};
\node[fill=dark-green!50,circle,draw,inner sep = 0pt, outer sep = 0pt, minimum size=1.5mm] (p) at (2,0) {};
\node[fill=dark-green!50,circle,draw,inner sep = 0pt, outer sep = 0pt, minimum size=1.5mm] (q) at (0,2) {};
\node[fill=dark-green!50,circle,draw,inner sep = 0pt, outer sep = 0pt, minimum size=1.5mm] (qp) at (2,2) {};
\draw[->, black!100] (1) edge (p);
\draw[->, rainbow-blue!100] (1) edge (qp);
\draw[->, rainbow-red!100, dashed] (1) edge (q);
\draw[->, black!10] (p) edge (qp);
\draw[->, black!100] (q) edge (qp);
\end{tikzpicture}
&
\begin{tikzpicture}[scale=0.5, framed]
\node[fill=dark-green!50,circle,draw,inner sep = 0pt, outer sep = 0pt, minimum size=1.5mm] (1) at (0,0) {};
\node[fill=dark-green!50,circle,draw,inner sep = 0pt, outer sep = 0pt, minimum size=1.5mm] (p) at (2,0) {};
\node[fill=dark-green!50,circle,draw,inner sep = 0pt, outer sep = 0pt, minimum size=1.5mm] (q) at (0,2) {};
\node[fill=dark-green!50,circle,draw,inner sep = 0pt, outer sep = 0pt, minimum size=1.5mm] (qp) at (2,2) {};
\draw[->, black!100] (1) edge (p);
\draw[->, rainbow-red!100, dashed] (1) edge (qp);
\draw[->, rainbow-blue!100] (1) edge (q);
\draw[->, black!10] (p) edge (qp);
\draw[->, black!100] (q) edge (qp);
\end{tikzpicture}
&
\begin{tikzpicture}[scale=0.5, framed]
\node[fill=dark-green!50,circle,draw,inner sep = 0pt, outer sep = 0pt, minimum size=1.5mm] (1) at (0,0) {};
\node[fill=dark-green!50,circle,draw,inner sep = 0pt, outer sep = 0pt, minimum size=1.5mm] (p) at (2,0) {};
\node[fill=dark-green!50,circle,draw,inner sep = 0pt, outer sep = 0pt, minimum size=1.5mm] (q) at (0,2) {};
\node[fill=dark-green!50,circle,draw,inner sep = 0pt, outer sep = 0pt, minimum size=1.5mm] (qp) at (2,2) {};
\draw[->, black!100] (1) edge (p);
\draw[->, rainbow-red!100, dashed] (1) edge (qp);
\draw[->, rainbow-red!100, dashed] (1) edge (q);
\draw[->, rainbow-blue!100] (p) edge (qp);
\draw[->, black!100] (q) edge (qp);
\end{tikzpicture}  & 
\end{tabular}
\end{table}

We add all of these transfer systems into $\mathsf{Tr}(G)$ and remove any redundant copies. At this point we have $|\mathsf{Tr}(G)|=10$, and we have added four new transfer systems. At this point \emph{we} know that we have found all transfer systems because this is a known computation, however the algorithm would not yet terminate. The algorithm would now take these four new transfer systems and repeat the process of adding all possible edges one at a time. It is clear that doing so in each case would result in the complete transfer system which is already in our collection. At this point, no new transfer systems were added in the loop, and the algorithm would terminate.
\end{example}

\addtocontents{toc}{\vspace{1\baselineskip}}

 \bibliography{bib}\bibliographystyle{alpha}

\end{document}